\documentclass[11pt,a4paper]{article}
\pdfoutput=1

\usepackage[a4paper,
            top=2.5cm,
            bottom=2.5cm,
            inner=3.5cm,
            outer=3.5cm]{geometry}
            
\usepackage[utf8]{inputenc}
\usepackage{microtype}
\usepackage{amsmath}
\usepackage{amsfonts}
\usepackage{amssymb}
\usepackage{amsthm}
\usepackage{graphicx}
\usepackage[all]{xy}
\usepackage{rotating}
\usepackage{stmaryrd}

\usepackage{enumerate}
\usepackage{mathtools}
\usepackage{stackrel}
\usepackage[nottoc]{tocbibind}
\usepackage{hyperref}

\usepackage{cleveref}
\usepackage{csquotes}

\theoremstyle{plain}
\newtheorem{thm}{Theorem}[section]
\newtheorem{lem}[thm]{Lemma}
\newtheorem{prop}[thm]{Proposition}
\newtheorem{cor}[thm]{Corollary}

\newtheorem*{mainresult}{Main result}
\newtheorem{question}{Question}
\newtheorem{conclusion}[thm]{Conclusion}

\theoremstyle{definition}
\newtheorem{defn}[thm]{Definition}
\newtheorem{example}[thm]{Example}

\newtheorem*{acknowledgements}{Acknowledgements}

\theoremstyle{remark}
\newtheorem{rem}[thm]{Remark}

\newcommand{\N}{\mathbb{N}}
\newcommand{\Z}{\mathbb{Z}}
\newcommand{\Q}{\mathbb{Q}}
\newcommand{\R}{\mathbb{R}}
\newcommand{\C}{\mathbb{C}}

\newcommand{\cZ}{\mathcal{Z}}

\newcommand{\cC}{\mathcal{C}}
\newcommand{\cD}{\mathcal{D}}
\newcommand{\cI}{\mathcal{I}}
\newcommand{\cCW}{\mathcal{CW}}

\newcommand{\Ct}{\mathrm{C}}
\newcommand{\Cz}{\Ct_0}
\newcommand{\Cc}{\Ct_c}
\newcommand{\Cb}{\Ct_b}
\newcommand{\Ca}{\Ct_a}

\newcommand{\fk}{\mathfrak{k}}
\newcommand{\fl}{\mathfrak{l}}
\newcommand{\fm}{\mathfrak{m}}
\newcommand{\fn}{\mathfrak{n}}
\newcommand{\fo}{\mathfrak{o}}

\newcommand{\ab}{\triangleright}
\newcommand{\filter}{\mathcal{F}}
\newcommand{\filterbasis}{\mathcal{B}}

\newcommand{\Hilbert}{\mathcal{H}}
\newcommand{\grS}{\mathcal{S}}
\newcommand*{\grtensor}{\mathbin{\widehat{\otimes}}}
\newcommand{\Cl}{\mathbb{C}\ell}

\DeclareMathOperator{\Mor}{Mor}
\newcommand{\op}{\mathrm{op}}

\newcommand{\Simpl}{\mathfrak{S}}

\newcommand{\SimplMor}[3][]{\llbracket #2,#3\rrbracket_{#1}}

\DeclareMathOperator{\id}{id}
\DeclareMathOperator{\Cone}{Cone}
\DeclareMathOperator{\Zyl}{Zyl}
\DeclareMathOperator{\Ad}{Ad}

\newcommand{\ev}[1]{\mathop{\textrm{ev}_{#1}}}

\newcommand{\K}{\mathrm{K}}
\newcommand{\KK}{\mathrm{KK}}
\newcommand{\EE}{\mathrm{E}}

\newcommand{\textCstar}{\ensuremath{\mathrm{C}^*\!}}
\newcommand{\Cstar}{\mathrm{C}^*}

\newcommand{\Kom}{{\mathfrak{K}}}
\newcommand{\Asymp}{\mathfrak{A}}

\newcommand{\blank}{-}

\crefname{thm}{Theorem}{Theorems}
\crefname{lem}{Lemma}{Lemmas}
\crefname{defn}{Definition}{Definitions}
\crefname{prop}{Proposition}{Propositions}
\crefname{cor}{Corollary}{Corollaries}
\crefname{equation}{}{}

\author{Christopher Wulff\thanks{Supported by the DFG through the Priority Programme ``Geometry at Infinity'' (SPP 2026; individual project ``Duality and the coarse assembly map'', WU 869/1-1, WU 869/1-2).}}
\title{Generalized asymptotic algebras and $\EE$-theory for non-separable \textCstar-algebras}

\begin{document}
\maketitle

\begin{abstract}
In previous definition of $\EE$-theory, separability of the \textCstar-algebras is needed either to construct the composition product or to prove the long exact sequences.
Considering the latter, the potential failure of the long exact sequences can be traced back to the fact that these $\EE$-theory groups accommodate information about asymptotic processes in which one real parameter goes to infinity, but not about more complicated asymptotics parametrized by directed sets.

We propose a definition for $\EE$-theory which also incorporates this additional information by generalizing the notion of asymptotic algebras. As a consequence, it not only has all desirable products but also all long exact sequences, even for non-separable \textCstar-algebras.

More precisely, our construction yields equivariant $\EE$-theory for $\Z_2$-graded $G$-\textCstar-algebras for arbitrary discrete groups $G$. 

We suspect that our model for $\EE$-theory could be the right entity to investigate index theory on infinite dimensional manifolds.
\end{abstract}

\tableofcontents

\section{Introduction}

Ever since its invention by Kasparov in \cite{KasparovOperatorKFunctor}, $\KK$-theory has become an overpowered tool for computing operator $\K$-theory.
Subsequently, numerous alternative definitions of $\KK$-theory have been given by various authors (we refrain from trying to give a complete list) which are isomorphic to Kasparov's original one for separable \textCstar-algebras.
Most notably, Higson characterized $\KK$-theory of separable \textCstar-algebras as the universal functor from the category of separable \textCstar-algebras to an additive category which is homotopy invariant, split exact and stable \cite{HigsonKKchar}.

Whilst most sub-fields of non-commutative topology only treat separable \textCstar-algebras, there are areas where non-separable ones appear quite natural and are in fact unavoidable. For example, in coarse geometry there are two \textCstar-algebras of particular importance, namely the Roe algebra (cf.\ \cite[Chapter 6]{HigRoe}) and the stable Higson corona (cf.\ \cite{EmeMeyDualizing}), which are essentially never separable.

Unfortunately, most definitions of $\KK$-theory do not work well for non-separable \textCstar-algebras. Usually, separability is either needed to define the composition product, as in Kasparov's original work, or to obtain the long exact sequences. 
For example, Skandalis' prominent definition 
\[\KK_{\mathrm{sep}}(A,B)\coloneqq\varprojlim_{A_1\subset A\text{ separable}}\KK(A_1,B)\]
(cf. \cite[Definition 5.5]{Skandalis_unenotion}), where $\KK(A_1,B)=\varinjlim_{B_1\subset B\text{ separable}}\KK(A_1,B_1)$ already holds with Kasparov's original definition,
 has all products, but taking the inverse limit can completely destroy the long exact sequences.
The only version of $\KK$-theory known to the author that has all desirable properties is the spectral picture developed recently by Bunke, Engel and Land \cite{BunkeEngelLand}, which even satisfies a spectral version of the universal characterization.

Another instance of bivariant $\K$-theory is $\KK$'s sibling $\EE$-theory (cf.\ \cite{ConnesHigson_french,ConHig,GueHigTro,HigGue}).
Up until now, it faced similar difficulties with non-separability: either separability was needed to define the composition product or to obtain the long exact sequences.
\emph{The purpose of the present paper is to resolve these issues.}
In order to understand the idea of our construction, we briefly have to dive into the established definitions of $\EE$-theory and its problems with non-separability.

\subsection{The problem with non-separability in $\EE$-theory and the idea of how to solve it}

Originally, $\EE(A,B)$ was defined in \cite{ConnesHigson_french,ConHig} as the set of homotopy classes of asymptotic morphisms from $\Cz(\R)\otimes A\otimes\Kom$ to $\Cz(\R)\otimes B\otimes\Kom$, where $\Kom$ denotes the \textCstar-algebra of compact operators on a fixed infinite dimensional separable Hilbert space. Recall that here an asymptotic morphism between two \textCstar-algebras $A,B$ was a continuous family of maps $\{\varphi_t\colon A\to B\}_{t\geq 1}$ which asymptotically behave like $*$-ho\-mo\-mor\-phisms as $t\to\infty$. The composition product in $\EE$-theory is then given by composition of homotopy classes of asymptotic morphisms, which works as follows: If $\{\psi_t\colon B\to C\}_{t\geq 1}$ is another asymptotic morphism, then its composition with $\{\varphi_t\}$ is represented by the family $\{\psi_{s(t)}\circ\varphi_t\}$ for a suitable sufficiently fast increasing reparametrization $s\colon [1,\infty)\to [1,\infty)$. The reparametrization is necessary, because for $s=\id$ the family can fail to be an asymptotic morphism.
In this construction of the product, separability of the \textCstar-algebra $A$ is needed to show that suitable reparametrizations exist. 

The newer versions of $\EE$-theory introduced in \cite{GueHigTro,HigGue} fix this problem in a very elegant way. 
Instead of using only asymptotic morphisms from $A$ to $B$, which correspond to $*$-ho\-mo\-mor\-phisms from $A$ into the so-called \emph{asymptotic algebra} $\Asymp B\coloneqq\Cb([1,\infty);B)/\Cz([1,\infty);B)$ of $B$, these articles also consider $*$-ho\-mo\-mor\-phisms into the $n$-fold asymptotic algebra $\Asymp^nB\coloneqq\Asymp\dots\Asymp B$.
An equivalence relation, called $n$-homotopy, is then defined on the set of all $*$-ho\-mo\-mor\-phisms $\varphi\colon A\to\Asymp^nB$ and one obtains the set $\SimplMor[n]{A}{B}$ of all such $n$-homotopy classes $\llbracket\varphi\rrbracket$. They fit into a canonical directed system 
\[\SimplMor[1]{A}{B}\to\SimplMor[2]{A}{B}\to\SimplMor[3]{A}{B}\to\dots\]
whose direct limit will be denoted by $\SimplMor[\infty]{A}{B}$ in our work.
The canonical map from $\SimplMor[1]{A}{B}$, i.\,e.\ the set of homotopy classes of asymptotic morphisms, into the direct limit is a bijection whenever $A$ is separable.
There are canonical composition products
\[\SimplMor[n]{A}{B}\times \SimplMor[m]{B}{C}\to \SimplMor[n+m]{A}{C}\]
which map the classes $\llbracket\varphi\colon A\to\Asymp^nB\rrbracket$ and $\llbracket\psi\colon B\to\Asymp^mC\rrbracket$ to the class represented by the composition 
\[A\xrightarrow{\varphi}\Asymp^nB\xrightarrow{\Asymp^n(\psi)}\Asymp^{m+n}C\,.\]
These products pass to the direct limit, turning the sets $\SimplMor[\infty]{A}{B}$ into the morphism sets of a category whose objects are the \textCstar-algebras.
Defining $\EE$-theory using the sets $\SimplMor[\infty]{\blank}{\blank}$ instead of the sets $\SimplMor[1]{\blank}{\blank}$ of homotopy classes of asymptotic morphisms, one obtains a version where composition products always exist and which agrees with the original one if the first \textCstar-algebra is separable.

In this picture of $\EE$-theory, non-separability only causes problems for exactness for the following reason. Given a short exact sequence 
\[0\to J\to B\xrightarrow{\pi} A\to 0\]
of separable \textCstar-algebras, one can choose a quasicentral approximate unit $\{u_n\}_{n\in\N}$ for $J$ in $B$ and extend it by affine linear interpolation to a continuous function $[1,\infty)\to J,t\mapsto u_t$. Then there is an asymptotic morphism 
\[\sigma\colon\Cz((0,1))\otimes A\to\Asymp J\,,\quad f\otimes\pi(b)\mapsto [t\mapsto f(u_t)b]\]
and the boundary maps $\EE(D,\Cz((0,1))\otimes A)\to \EE(D,J)$ and $\EE(J,D)\to \EE(\Cz((0,1))\otimes A,D)$ in the long exact sequences of $\EE$-theory are given by composition product with the $\EE$-theory class of $\sigma$.
It turns out that this asymptotic morphism not only provides a nice description of the boundary maps, but it actually plays an essential role even in the proof of half-exactness.

In the case of non-separable \textCstar-algebras, however, there are in general no sequential quasicentral approximate units, but only ones which are indexed by general directed sets $\fm$. 
Therefore, we do not obtain an associated asymptotic morphism as in the separable case above. Nevertheless, it inspired us to perform the following analogous construction instead: We can extend a quasicentral approximate unit $\{u_m\}_{m\in\fm}$ by affine linear interpolation to a bounded continuous map $|\Delta^\fm|\to J, x\mapsto u_x$ on the geometric realization $|\Delta^\fm|$ of the full simplical complex with vertex set $\fm$.
Then the exact same formula as in the separable case defines a $*$-ho\-mo\-mor\-phism
\[\sigma\colon\Cz((0,1))\otimes A\to\Simpl^\fm J\coloneqq \frac{\Cb(|\Delta^\fm|;J)}{\Cz(|\Delta^\fm|;J)}\]
(cf.\ \Cref{prop:extensionsimpletoticmorphism}).
We call the \textCstar-algebras $\Simpl^\fm J$ \emph{simpltotic algebras}, see \Cref{defn:simpltoticalgebra}, since they are analogous to the asymptotic algebras but modeled on simplices. 
The ideal $\Cz(|\Delta^\fm|;J)$ consists, as usual, of those bounded continuous functions which ``vanishing at infinity''.
However, what the latter means has to be redefined in this context, because $|\Delta^\fm|$ is in general not a locally compact space.
We will do this in larger generality for topological spaces equipped with a filter in \Cref{sec:funcalg}, which allows a unified treatment of asymptotic and simpletotic algebras as \emph{generalized asymptotic algebras} and is also necessary for a convenient discussion of the composition product.

The main idea of our work is now evident: If we want to have a version of $\EE$-theory which has long exact sequences also in the non-separable case, then we have to define it in such a way that $\sigma$ defines an element in it. 
In other words, instead of defining the $\EE$-theory groups as morphism sets in the asymptotic category, we want to define them as morphism sets in a category whose morphisms are represented by $*$-ho\-mo\-mor\-phisms $A\to\Simpl^\fm B$ instead of $*$-ho\-mo\-mor\-phisms $A\to\Asymp^nB$.
To this end, we define an equivalence relation called $\fm$-homotopy between $*$-ho\-mo\-mor\-phisms $A\to\Simpl^\fm$ and denote the set of $\fm$-homotopy classes by $\SimplMor[\fm]{A}{B}$ (\Cref{defn:Fhomotopies,defn:Fhomotopysets}).
We will then define the morphism sets of our new category as a direct limit 
\[\SimplMor{A}{B}\coloneqq \varinjlim\SimplMor[\fm]{A}{B}\]
over increasingly large directed sets $\fm$ (\Cref{defn:morphismlimit}).
This category will be called the \emph{simpltotic category} and we use it in \Cref{sec:Etheory} to define our new version of $\EE$-theory.

As simple and clear the idea may be, its technical implementation is the complete opposite. In particular, defining the composition in the simpltotic category and proving its main properties turns out to be quite complicated and will take up the whole of \Cref{sec:compositionproduct}.
The problem behind it is that we do not incorporate iterated simpltotic algebras in our definition of the simpletotic category, as \cite{GueHigTro} did with asymptotic algebras. Therefore, we cannot use their elegant construction of the composition and have to use something like a reparametrization similar to the one in \cite{ConnesHigson_french,ConHig} instead.
The product will then be implemented by maps $\SimplMor[\fm]{A}{B}\times \SimplMor[\fn]{B}{C}\to \SimplMor[\fk]{A}{C}$, where unfortunately the directed set $\fk$ will be chosen awfully huge compared to the sizes of $\fm$ and $\fn$ in order to have enough space for the reparametrization.
This is rather impractical for computations and thus it would be interesting to know if the construction can be simplified.

\subsection{Overview of main results and outlook}

The big novelty of this publication is the development of the \emph{simpltotic category}. 
We will do this in even greater generality than in the motivation we gave above, namely for all $\Z_2$-graded $G$-\textCstar-algebras, where $G$ is a discrete group.
This will take up \Cref{sec:funcalg,sec:simpltoticmorphismsets,sec:compositionproduct}.
In the subsequent \Cref{sec:functors,sec:extensions} we will prove some properties of the simpltotic category which will be relevant for the development of $\EE$-theory.

In \Cref{sec:Etheory} we can finally develop our new version of $\EE$-theory. 
This will be done essentially in complete analogy to \cite{HigGue}, but using our simpltotic category instead of the asymptotic category and performing a few further minor modifications.
Most proves will be omitted, because they are almost identical to the ones given in \cite{GueHigTro}. One just needs to replace the properties of the asymptotic category by the corresponding ones of the simpltotic category that we had proven in \Cref{sec:funcalg,sec:simpltoticmorphismsets,sec:compositionproduct,sec:functors,sec:extensions}.
More concretely:
\begin{mainresult}
We will:
\begin{itemize}
\item define $\EE$-theory for all $G$-\textCstar-algebras, where $G$ is a fixed discrete group;
\item show that it has all the products, long exact sequences and further properties that one would like to have;
\item prove that it agrees with the usual definition of $\EE$-theory if $G$ is countable and the \textCstar-algebra in the first entry is separable.
\end{itemize}
\end{mainresult}

The focus of our work lies on getting the products and the long exact sequences right, because these are the most important properties for doing computations and hence for concrete applications.
That said, there are a few questions which we left open, although their answers would be interesting from a theoretical point of view. 

\begin{question}
Does our model of $\EE$-theory satisfy Higson's characterization, that is, is it the universal one with the properties of homotopy invariance, half-exactness and stablility?
\end{question}

The following question was raised by Rufus Willett.

\begin{question}
How far away is our $\EE$-theory group $\EE(A,B)$ for non-separable $A$ away from the group
\[\EE_{\mathrm{sep}}(A,B)\coloneqq\varprojlim_{A_1\subset A\text{ separable}}\EE(A_1,B)\]
\`a la Skandalis \cite{Skandalis_unenotion}? 
Is there some kind of $\varprojlim^1$-sequence relating them?
\end{question}

In contrast to \cite{GueHigTro}, we restrict ourselves to discrete groups $G$ instead of second countable, locally compact groups. The reason is that for topological groups we would need a simpltotic version of \cite[Lemma 1.8]{GueHigTro}, that is, that the sub-\textCstar-algebra of $G$-continuous elements of $\Simpl^\fm B$ is exactly the quotient of the sub-\textCstar-algebra of $G$-continuous elements of $\Cb(|\Delta^\fm|;J)$ by the sub-\textCstar-algebra of $G$-continuous elements of $\Cz(|\Delta^\fm|;J)$. However, the proof of the cited lemma cannot be adapted: It makes heavy use of the Baire category theorem, applied for each $n\in\N$ to a countable family $\{W_{n,m}\}_{m\in\N}$ of closed subsets of $G$. In the simpltotic case, the corresponding family $\{W_{n,m}\}_{m\in\fm}$ is indexed by the directed set $\fm$, and hence Baire cannot be applied if $\fm$ is uncountable.

\begin{question}
Can our construction of $G$-equivariant $\EE$-theory be modified such that it works for non-discrete topological groups $G$, too?
\end{question}

Finally, one should also ask the question of concrete applications of our new model for $\EE$-theory. 
Whilst the previous definitions comprise informations about asymptotic processes in finitely many real variables, the new definition also involves more complex asymptotics over directed sets.
Are there more situations where cycles of our new model arise in a canonical fashion, besides the $*$-ho\-mo\-mor\-phisms $\sigma\colon\Cz((0,1))\otimes A\to\Simpl^\fm J$ mentioned in the previous subsection?

One somewhat speculative idea would be to look for cycles coming from elliptic operators on infinite dimensional manifolds.
Recall from \cite{Dumitrascu} that a generalized Dirac operator $D\colon\Gamma(S)\to\Gamma(S)$ over a complete even dimensional manifold $M$ has a canonical class in $\EE(\Cz(M),\C)$ represented by the asymptotic morphism 
\[\grS\grtensor\Cz(M)\to\Asymp(\Kom(L^2(S)))\,,\qquad f\grtensor g\mapsto [t\mapsto f(t^{-1}D)f]\]
where $\grS$ denotes the \textCstar-algebra $\Cz(\R)$ but equiped with the grading in even and odd functions.\footnote{Similar $\EE$-theory elements exist in longitudinal index theory of foliations, see \cite[Proposition 8.17]{WulffFoliations}, and in coarse index theory, see \cite[Lemma 4.5 \& Definition 4.6]{WulffTwisted}.}
If one is looking for a similar formula on infinite dimensional manifolds, it might not be appropriate to multiply the operator $D$ by $t^{-1}$ homogeneously in all directions. 
Our model would allow to look for cycles in which the rescaling is done with different factors on different finite dimensional subspaces, which is an asymptotic process over a more complicated parameter space than $[1,\infty)$.
The discovery of such classes could prove to be very helpful to index theory on infinite dimensional manifolds.

To demonstrate that scaling differently in different directions can indeed work, we are going to revisit the infinite dimensional Bott periodicity by \cite{HigsonKasparovTrout} in \Cref{sec:infinitedimBott}.

\begin{acknowledgements}
The author heartly thanks Rufus Willett and Erik Guentner for their hospitality during his research visit to the University of Hawaii and for extensive inspiring conversations.
Thanks also goes to Ralf Meyer for giving a helpful hint.
\end{acknowledgements}

\section{Fixing some notation}
\label{sec:notationsandconventions}
Lowercase Fraktur letters $\fk,\fl,\fm,\fn,\dots$ will always denote non-empty directed sets. For each $m\in\fm$ we denote by $\fm\ab m$ the tail $\{n\in\fm\mid n\geq m\}$ of $\fm$ starting at $m$.

Let $\Delta^S$ denote the full simplicial complex with vertex set $S$ for arbitrary $S$. We use the same symbol for the set of all of its simplices, i.\,e.\ the set of all finite nonempty subsets of $S$.
Its geometric realization will be denoted by $|\Delta^S|$. Points of the latter are formal finite convex combinations of the vertices written as $\sum_{s\in S}\lambda_s[s]$.
Note that in this terminology $|\Delta^\sigma|\subset|\Delta^S|$ is the geometric realization of a simplex $\sigma\in\Delta^S$.

Let $G$ denote a fixed discrete group.
Unless stated otherwise, all \textCstar-algebras appearing will be $\Z_2$-graded $G$-\textCstar-algebras. Ungraded \textCstar-algebras are considered as being equipped with the trivial grading. All $*$-ho\-mo\-mor\-phisms are required to preserve the grading and be $G$-equivariant.
The letters $A,B,C,D$ will usually stand for such \textCstar-algebras.
We use the symbol $\otimes$ (resp. $\grtensor$) for the \emph{maximal} (graded) tensor product instead of the minimal one, because all tensor products of \textCstar-algebras under consideration are maximal ones.

\section{Asymptotic and simpltotic algebras}
\label{sec:funcalg}

Generalizing the asymptotic algebra, we make the following definition which will also include the simpltotic algebras.

\begin{defn}
\label{defn:generalizedasymptoticalgebras}
Let $X$ be a topological space equipped with a filter $\filter$ on the underlying set and let $B$ be a \textCstar-algebra. 
As usual, let $\Cb(X;B)$ denote the \textCstar-algebra of bounded continuous functions on $X$ with values in $B$. 
We define
\[\Cz(X;B)\coloneqq\{f\in\Cb(X;B)\mid\forall\varepsilon>0\exists U\in\filter\colon \|f|_U\|<\varepsilon\}\]
and note that the definition remains unchanged if we replace $\filter$ by any of its filter bases. 
In other words: $\Cz(X;B)$ consists of all functions in $\Cb(X;B)$ which converge to zero along the filter.
Associated to $(X,\filter)$ is the following \emph{generalized asymptotic algebra:}
\[\Ca(X;B)\coloneqq \Cb(X;B)/\Cz(X;B)\]
\end{defn}

We have supressed the filter from the notation, because we will consider it as an integral part of the space $X$ itself.

Recall from \ref{sec:notationsandconventions} that all \textCstar-algebras under consideration are implicitely assumed to be graded $G$-\textCstar-algebras. Here, the space $X$ does not carry a $G$-action and hence the $G$-action (and grading) on $\Cb(X;B)$, $\Cz(X;B)$ and $\Ca(X;B)$ comes from the $G$-action (and grading) on $B$ alone.

\begin{example}\label{ex:loccomp}
If we equip a locally compact space $X$ with the filter generated by the complement of all compact subsets, then $\Cz(X;B)$ will be the \textCstar-algebra of functions vanishing at infinity in the classical sense.

In particular, if $T$ denotes the half-line $[1,\infty)$ equipped with this filter, then 
$\Ca(T;B)$ is simply the asymptotic algebra $\Asymp B$.
\end{example}

For any directed set $\fm$ we equip the space $|\Delta^\fm|$ with a canonical filter, namely the one generated by the subsets $|\Delta^{\fm\ab m}|$ with $m\in\fm$. Note that these subsets actually consitute a filter basis and hence we have 
\begin{align*}
\Cz(|\Delta^\fm|;B)&=\{f\in\Cb(|\Delta^\fm|;B)\mid\forall\varepsilon>0\exists m\in\fm\colon \|f|_{|\Delta^{\fm\ab m}|}\|<\varepsilon\}
\\&=\overline{\bigcup_{m\in\fm}\{f\in\Cb(|\Delta^\fm|;B)\mid f|_{|\Delta^{\fm\ab m}|}=0\}}\,.
\end{align*}

\begin{defn}
\label{defn:simpltoticalgebra}
The \emph{$\fm$-simpltotic algebra} of a \textCstar-algebra $B$ is the quotient \textCstar-algebra
\[\Simpl^\fm B\coloneqq \Ca(|\Delta^\fm|;B)= \Cb(|\Delta^\fm|;B)/\Cz(|\Delta^\fm|;B)\,.\]
\end{defn}

The important relation between the asymptotic algebras and the simpletotic algebras will be discussed in \Cref{ex:asymptoticsimpltoticrelation} below. For now we just give a very simple example.
\begin{example}
If $\fm$ has a maximal element $m\coloneqq\max\fm\in\fm$, then $\Simpl^\fm B\cong B$ by evaluation at $m\in|\Delta^\fm|$.
\end{example}

Clearly $\Cb(X;\blank)$, $\Cz(X;\blank)$, $\Ca(X;\blank)$ and hence in particular $\Asymp$ and $\Simpl^\fm$ are functors from the category of \textCstar-algebras to itself.

If $F$ is a functor from the category of \textCstar-algebras to itself and $Y$ is a topological space and $y\in Y$, then the evaluation $*$-ho\-mo\-mor\-phism $\ev{y}\colon \Cb(Y;B)\to B,f\mapsto f(y)$ induces a $*$-ho\-mo\-mor\-phism
\[F(\ev{y})\colon F(\Cb(Y;B))\to F(B)\]
which we also call evaluation at $x$.

\begin{defn}
\label{defn:Fhomotopies}
Let $F$ be a functor from the category of \textCstar-algebras to itself.
We say that two $*$-ho\-mo\-mor\-phisms $\varphi_0,\varphi_1\colon A\to F(B)$ are $F$-homotopic if there is a $*$-ho\-mo\-mor\-phism $\Phi\colon A\to F(\Ct([0,1];B))$ which evaluates to $\varphi_0,\varphi_1$ at $0,1$: $\varphi_i=F(\ev{i})\circ\Phi$.
Such a $\Phi$ is called an \emph{$F$-homotopy} between $\varphi_0$ and $\varphi_1$.

For $F=\Simpl^\fm$ we often say \emph{$\fm$-homotopy/-ic} instead of $\Simpl^\fm$-homotopy/-ic and for $F=\Asymp^n$ ($n\in\N$) we often say \emph{$n$-homotopy/-ic} instead of $\Asymp^n$-homotopy/-ic.
\end{defn}

It is clear that being $\Cb(X;\blank)$-homotopic or $\Cz(X;\blank)$-homotopic are equivalence relations.
Moreover, being $\Asymp^n$-asymptotic is an equivalence relation according to \cite[Proposition 2.3]{GueHigTro}.
Along those lines we can prove the following.

\begin{prop}\label{prop:homotopyequivalencerelation}
Let $F=F_1\circ\dots\circ F_k$, where each $F_i$ is a functor of the type $\Cb(X_i;\blank)$, $\Cz(X_i;\blank)$ or $\Ca(X_i;\blank)$. Additionally we assume that each $X_i$ is a CW-complex and that the filter on it has a basis consisting of subcomplexes. Then being $F$-homotopic is an equivalence relation.
\end{prop}

In its proof we will use the following two lemmas, compare \cite[Lemmas 2.4,2.5]{GueHigTro}.

Recall that a functor $F$ from the category of \textCstar-algebras to itself is called \emph{exact} if for each short exact sequence $0\to J\to B\to B/J\to 0$ the induced sequence $0\to F(J)\to F(B)\to F(B/J)\to 0$ is also exact.

\begin{lem}[{cf.\ \cite[Lemma 2.4]{GueHigTro}}]\label{lem:asymptoticalgebrasexact}
Let $X$ be a CW-complex as in \Cref{prop:homotopyequivalencerelation}.
Then the functors $\Cz(X;\blank)$, $\Cb(X;\blank)$ and $\Ca(X;\blank)$ are exact.
\end{lem}
\begin{proof}
Exactness of the latter follows from exactness of the first two by a diagram chase. For the first two, the not totally obvious part is surjectivity.

Let $f\in \Cb(X;B/J)$. We construct a preimage $\bar f\in \Cb(X;B)$ by recursion over the $i$-skeleton $X_i$.

For $i=0$ we choose $\bar f_0\colon X_0\to B$ such that $\|\bar f_0(x)\|=\|f(x)\|$ for all $x\in X_0$.

Now assume that the lift $\bar f_i\colon X_i\to B$ of $f|_{X_i}$ has already been constructed and consider a closed $(i+1)$-cell $Z\subset X$. 
Since 
\[\xymatrix{\Ct(Z;B)\ar[r]\ar[d] &\Ct(Z;B/J)\ar[d]\\\Ct(\partial Z;B)\ar[r]&\Ct(\partial Z;B/J)} \]
is a pullback diagram,
 the pair $(\bar f_i|_{\partial Z},f|_Z)$ lies in the image of the canonical $*$-ho\-mo\-mor\-phism $\Ct(Z;B)\to \Ct(\partial Z;B)\oplus \Ct(Z;B/J)$ and hence we can find a preimage $\bar f_Z$ with norm $\|\bar f_Z\|=\max\{\|\bar f_i|_{\partial Z}\|,\|f|_Z\|\}$. These extensions fit together to yield the lift $\bar f_{i+1}\colon X_{i+1}\to B$ on the $(i+1)$-skeleton.

All $\bar f_i$ together unite to give a continuous map $\bar f\colon X\to B$ which lifts $f$.
Note that by construction we have $\|\bar f|_Y\|=\|f|_Y\|$ for all subcomplexes $Y\subset X$.
This shows $\bar f\in\Cb(X;B)$ and if $f\in\Cz(X;B/J)$ then also $\bar f\in\Cz(X;B)$
\end{proof}

\begin{lem}[{cf.\ \cite[Lemma 2.5]{GueHigTro}}]\label{lem:asympalgebrapullback}
Let $\varphi_1\colon B_1\to B$ and $\varphi_2\colon B_2\to B$ be $*$-ho\-mo\-mor\-phisms, one of which is surjective, and let 
\[B_1\oplus_BB_2\coloneqq\{(b_1,b_2)\in B_1\oplus B_2\mid\varphi_1(b_1)=\varphi_2(b_2)\}\,.\]
Then for any functor $F=F_1\circ\dots\circ F_k$ as in \Cref{prop:homotopyequivalencerelation} the canonical $*$-ho\-mo\-mor\-phism
\[F(B_1\oplus_BB_2)\to F(B_1)\oplus_{F(B)}F(B_2)\]
induced by the projections from $B_1\oplus_BB_2$ to $B_1$ and $B_2$ is an isomorphism.
\end{lem}

\begin{proof}
The proof works exactly as the proof of \cite[Lemmas 2.5]{GueHigTro}.
Due to the previous lemma, each $F_i$ transforms surjections into surjections, so by using induction it suffices to prove the claim for $k=1$.
For the functors $F_b\coloneqq\Cb(X;\blank)$ and $F_0\coloneqq\Cz(X;\blank)$ it is clearly true.

Consider now $F_a\coloneqq\Ca(X;\blank)$.
There is a commutative diagram 
\[\xymatrix{
0\ar[r]&F_0(B_1\oplus_BB_2)\ar[r]\ar[d]^{\cong}&F_b(B_1\oplus_BB_2)\ar[r]\ar[d]^{\cong}&F_a(B_1\oplus_BB_2)\ar[r]\ar[d]&0
\\0\ar[r]&F_0B_1\oplus_{F_0B}F_0B_2\ar[r]&F_bB_1\oplus_{F_bB}F_bB_2\ar[r]&F_aB_1\oplus_{F_aB}F_aB_2\ar[r]&0
}\]
with exact rows, where surjectivity in the lower row is due to the surjectivity assumption and the previous lemma.
The vertical arrows are induce by the projections onto the components and the first two of them are isomorphisms, as we have just noticed. Thus, the third one is also an isomorphism.
\end{proof}

\begin{proof}[Proof of \Cref{prop:homotopyequivalencerelation}]
The proof is completely analogue to \cite[Proposition 2.3]{GueHigTro}. Reflexivity and surjectivity is trivial.

For transitivity, if $\Phi_0\colon A\to F(\Ct([0,1];B))$ is an $F$-homotopy between $\varphi_0$ and $\varphi_1$ and $\Phi_1\colon A\to F(\Ct([1,2];B))$ is an $F$-homotopy between $\varphi_1$ and $\varphi_2$, then applying the previous lemma provides us with an $F$-homotopy
\[A\xrightarrow{(\Phi_0,\Phi_1)}F(\Ct([0,1];B))\oplus_{FB}F(\Ct([0,1];B))\cong F(\Ct([0,2];B))\]
between $\varphi_0$ and $\varphi_2$.
\end{proof}

\begin{defn}
\label{defn:Fhomotopysets}
Let $F$ be a functor as in \Cref{prop:homotopyequivalencerelation}. We denote by $\SimplMor[F]{A}{B}$ denote the set of $F$-homotopy classes of $*$-ho\-mo\-mor\-phisms $A\to F(B)$.
Furthermore, we shall use the following abbreviations: 
\begin{enumerate}
\item In case $F=\Ca(X;\blank)$ we will often write $\SimplMor[X]{A}{B}$ instead of $\SimplMor[\Ca(X;\blank)]{A}{B}$.

\item Even more specifically, if $X=|\Delta^\fm|$ then $\SimplMor[\fm]{A}{B}$ stands for $\SimplMor[|\Delta^\fm|]{A}{B}$.

\item If $n\in\N$, we will write $\SimplMor[n]{A}{B}$ for the set $\SimplMor[\Asymp^n]{A}{B}$, cf.\ \cite[Definition 2.6]{GueHigTro}.
\end{enumerate}
In all cases, the class represented by a $*$-ho\-mo\-mor\-phism $\varphi$ will be denoted by $\llbracket \varphi\rrbracket$.
\end{defn}

\begin{example}\label{example:oneelementdirectedset}
Let $\fo$ denote the one-element directed set. Then $\Simpl^\fo$ is naturally isomorphic to the identity functor on \textCstar-algebras and hence $\SimplMor[\fo]{A}{B}=[A,B]$ is the set of homotopy classes of $*$-ho\-mo\-mor\-phisms $A\to B$.
\end{example}

\begin{lem}[{cf.\ \cite[Lemmas 2.9]{GueHigTro}}]\label{lem:FFprimehomotopic}
Let $F,F'$ be two functors as in \Cref{prop:homotopyequivalencerelation}.
\begin{enumerate}
\item If $\varphi_0,\varphi_1\colon A\to F'B$ are $F'$-homotopic, then $F\varphi_0,F\varphi_1\colon FA\to F\circ F'(B)$ are $F\circ F'$-homotopic.
\item If $\varphi_0,\varphi_1\colon A\to F\circ F'(B)$ are $F$-homotopic, then they are also $F\circ F'$-homotopic.
\end{enumerate}
\end{lem}

\begin{proof}
The first part is trivial and the second part follows from the existence of canonical natural transformations $\Ct([0,1];F'(\blank))\to F'(\Ct([0,1];\blank)$ such that the diagrams
\[\xymatrix{\Ct([0,1];F'B)\ar[r]\ar[dr]^{\ev{t}}&F'(\Ct([0,1];B))\ar[d]^{F'(\ev{t})}\\&F'B}\]
commute for all $t\in[0,1]$.
For $F'=F_b$ and $F'=F_0$ (notation as in the proof of \Cref{lem:asympalgebrapullback}) these map $f\in \Ct([0,1];F'B)$ to the function $\hat f\in F'(\Ct([0,1];B))$ defined by $\hat f(x)(t)=f(t)(x)$.
For $F'=F_a$ we obtain it as the induced right vertical arrow in 
\[\xymatrix{
0\ar[r]&\Ct([0,1];F_0B)\ar[r]\ar[d]&\Ct([0,1];F_bB)\ar[r]\ar[d]&\Ct([0,1];F_aB)\ar[r]\ar@{-->}[d]&0
\\0\ar[r]&F_0(\Ct([0,1];B))\ar[r]&F_b(\Ct([0,1];B))\ar[r]&F_a(\Ct([0,1];B))\ar[r]&0
}\]
and the general case follows by induction.
\end{proof}

\begin{cor}\label{cor:Fhomotopyclasscompositionproduct}
For two functors $F,F'$ as in \Cref{prop:homotopyequivalencerelation} there are canonical maps
\[\SimplMor[F]{A}{B} \times \SimplMor[F']{B}{C} \to \SimplMor[F\circ F']{A}{C}\,,\quad(\llbracket\varphi\rrbracket,\llbracket\psi\rrbracket)\mapsto \llbracket F(\psi)\circ\varphi\rrbracket\,.\]
These products are clearly associative.

In particular, setting $F=\id$ or $F'=\id$ we see that the sets $\SimplMor[F]{A}{B}$ are contravariantly functorial in $A$ and covariantly functorial in $B$ under homotopy classes of $*$-ho\-mo\-mor\-phisms.\qed
\end{cor}

Furthermore, we need the functoriality of the sets $\SimplMor[X]{A}{B}$ in $X$ and, more generally, of $\SimplMor[F]{A}{B}$ in $F$.
\begin{lem}\label{lem:functorialityundernatrualtransformationsoffunctors}
For any natural transformation $\eta\colon F\to F'$ between two functors $F,F'$ as in \Cref{prop:homotopyequivalencerelation} there are induced maps
\[\eta_*\colon\SimplMor[F]{A}{B}\to \SimplMor[F']{A}{B}\,,\quad \llbracket \varphi\rrbracket\mapsto\llbracket \eta_B\circ\varphi\rrbracket\,.\]
If two such natural transformations $\eta_0,\eta_1$ are $F'$-homotopic in the sense that there is a natural transformation $\tilde\eta\colon F\to F'(\Ct([0,1];\blank))$ such that $\eta_t=F'(\ev{t})\circ\tilde\eta$ for $t=0,1$, then $(\eta_0)_*=(\eta_1)_*$
\end{lem}

\begin{proof}
If $\Phi\colon A\to F(\Ct([0,1];B))$ is an $F$-homotopy between $\varphi_0$ and $\varphi_1$, then $\eta_{\Ct([0,1];B)}\circ\Phi$ is an $F'$-homotopy between $\eta_B\circ\varphi_0$ and $\eta_B\circ\varphi_1$. This shows the first part.
For the second part, note that $\tilde\eta_B\circ\varphi$ is an $F'$-homotopy between $(\eta_0)_B\circ\varphi$ and $(\eta_1)_B\circ\varphi$.
\end{proof}

Specializing it to functoriality in $X$ requires the following notion of maps.

\begin{defn}
If $X,X'$ are topological spaces equipped with filters $\filter,\filter'$, then we call a map $f\colon X\to X'$ \emph{filter-proper} if $f^{-1}(\filter')\subset\filter$.

Two continuous filter-proper maps $f_0,f_1$ are called homotopic if there is a continuous and filter-proper homotopy $h\colon X\times[0,1]\to X'$ between them, where the domain is equipped with the canonical filter generated by the sets $U'\times[0,1]$ with $U'\in\filter'$.
\end{defn}

If fixed filter bases $\filterbasis,\filterbasis'$ of $\filter,\filter'$ are given, then being filter-proper is equivalent to demanding that for each $U'\in\filterbasis'$ there exists $U\in\filterbasis$ such that $f(U)\subset U'$.

For locally compact Hausdorff spaces and filters generated by the complements of compact subsets, filter-properness is equivalent to properness.

The following lemma is obvious.

\begin{lem}\label{lem:spacefunctoriality}
Every continuous filter-proper map $f\colon X\to X'$ induces a natural transformation $f^*\colon \Ca(X';\blank)\to \Ca(X;\blank)$ and hence yields a map 
\[f^*\colon \SimplMor[X']{A}{B}\to \SimplMor[X]{A}{B}\,.\]
If two such maps $f_0,f_1$ are homotopic via a continuous filter-proper homotopy $h\colon X\times[0,1]\to X'$, then the induced natural transformation $h^*\colon\Ca(X';\blank)\to \Ca(X;\Ct([0,1];\blank))$ is a $\Ca(X;\blank)$-homotopy between $f_0^*$ and $f_1^*$ and hence the induced maps $f^*\colon \SimplMor[X']{A}{B}\to \SimplMor[X]{A}{B}$ agree.

This constitutes the contravariant functoriality of the sets $\SimplMor[X]{A}{B}$ under homotopy classes of continuous filter-proper maps.\qed
\end{lem}

The three functorialities of the sets $\SimplMor[F]{A}{B}$ and $\SimplMor[X]{A}{B}$ in the variables $A,B,F$ and $A,B,X$, respectively, clearly commute with each other.

We can now discuss the fundamental relation between asymptotic and simpltotic morphisms.
\begin{example}\label{ex:asymptoticsimpltoticrelation}
Recall from \Cref{ex:loccomp} that we denote by $T$ the half-line $[1,\infty)$ equipped with the filter generated by the complements of the compact subsets.
Furthermore, let $\N_+$ denote the set of positive integers directed by the usual order. We consider the following two continuous maps:
\begin{align*}
\tau&\colon |\Delta^{\N_+}|\to T\,,\quad\sum_{n\in{\N_+}}\lambda_n[n]\mapsto\sum_{n\in{\N_+}}\lambda_n\cdot n
\\\iota&\colon T\to |\Delta^{\N_+}|\,,\quad n+t\mapsto (1-t)[n]+t[n+1]\qquad\text{for }n\in{\N_+},t\in[0,1]\,.
\end{align*}
Note that $\tau(|\Delta^{{\N_+}\ab n}|)=[n,\infty)$ and $\iota([n,\infty))\subset |\Delta^{{\N_+}\ab n}|$ for all $n$, so these maps are filter-proper and therefore induce $*$-ho\-mo\-mor\-phisms
\[\xymatrix{\Asymp B\ar@<+.5ex>[r]^{\tau^*}&\Simpl^{\N_+}B\ar@<+.5ex>[l]^{\iota^*}}\]
and hence also maps
\[\xymatrix{\SimplMor[1]{A}{B}\ar@<+.5ex>[r]^{\tau^*}&\SimplMor[\N_+]{A}{B}\ar@<+.5ex>[l]^{\iota^*}}\]
which are clearly natural in both $A,B$.

Obviously $\tau\circ\iota=\id_{T}$ and hence $\iota^*\circ\tau^*=\id$. Furthermore, $\iota\circ\tau$ is homotopic to $\id_{|\Delta^{\N_+}|}$ via the canonical affine linear homotopy $h\colon |\Delta^{\N_+}|\times[0,1]\to |\Delta^{\N_+}|$.
Note that $h(|\Delta^{{\N_+}\ab n}|\times[0,1])\subset|\Delta^{{\N_+}\ab n}|$, so this homotopy is filter-proper.
Therefore, the maps $\iota^*,\tau^*$ are natural bijections between \(\SimplMor[1]{A}{B}\) and \(\SimplMor[\N_+]{A}{B}\)
which are inverse to each other.
\end{example}

\section{The morphism sets of the simpltotic category}
\label{sec:simpltoticmorphismsets}

Any map $\alpha\colon\fm\to\fn$ extends affine linearly to a continuous map $|\alpha|\colon |\Delta^\fm|\to|\Delta^\fn|$. If $\alpha$ is cofinal, then for each $n\in\fn$ there is $m\in\fm$ with $\forall k\geq m\colon\alpha(k)\geq n$ and hence $|\alpha|$ maps $|\Delta^{\fm\ab m}|$ to $|\Delta^{\fn\ab n}|$, i.\,e.\ $|\alpha|$ is filter-proper.
Thus the $*$-ho\-mo\-mor\-phism \(\alpha^*\coloneqq|\alpha|^*\colon \Cb(|\Delta^\fn|;B)\to \Cb(|\Delta^\fm|;B)\) restricts to \(\Cz(|\Delta^\fn|;B)\to \Cz(|\Delta^\fm|;B)\) and thus we obtain the induced $*$-ho\-mo\-mor\-phisms $\alpha^*\coloneqq|\alpha|^*\colon\Simpl^\fn B\to\Simpl^\fm B$ as a special case of the one in \Cref{lem:spacefunctoriality}.

\begin{lem}\label{lem:independenceofcofinalmap}
If a cofinal map $\alpha\colon\fm\to\fn$ between two directed sets $\fm,\fn$ exists, then the induced maps
\[\alpha^*\coloneqq|\alpha|^*\colon \SimplMor[\fn]{A}{B}\to\SimplMor[\fm]{A}{B}
\]
are independent of the choice of the cofinal map $\alpha$.
\end{lem}

\begin{proof}
If $\alpha,\beta\colon\fm\to\fn$ are two cofinal maps, then the canonical affine linear homotopy $h\colon |\Delta^{\fm}|\times[0,1]\to|\Delta^{\fn}|$ between $|\alpha|$ and $|\beta|$ maps $|\Delta^{\fm\ab m}|\times[0,1]$ to $|\Delta^{\fn\ab n}|$ if $m\in\fm$ has been chosen large enough such that $\alpha(k),\beta(k)\geq n$ for all $k\geq m$.
Thus, $h$ is filter-proper and the claim follows from \Cref{lem:spacefunctoriality}.
\end{proof}

Let $\cD$ denote the category whose objects are the directed sets and whose morphism sets $\Mor_\cD(\fm,\fn)$ consist of exactly one arrow if a cofinal map $\fm\to\fn$ exists and are empty otherwise. Note that $\cD$ is cofiltered, because for any two $\fm,\fn$ we can equip $\fm\times\fn$ with the canonical order and then the projections onto the factors are cofinal.
Due to \Cref{lem:independenceofcofinalmap}, the sets $\SimplMor[\fm]{A}{B}$ are contravariantly functorial on $\cD$.
This almost allows us to make the following definition.
\begin{defn}\label{defn:morphismlimit}
The \emph{simpltotic morphism sets} are defined as 
\[\SimplMor{A}{B}\coloneqq \varinjlim_{\fm\in\cD^\op}\SimplMor[\fm]{A}{B}\]
and these are clearly contravariantly functorial in $A$ and covariantly functorial in $B$.
\end{defn}

The problem with this definition is that the category $\cD^\op$ is not small. Therefore, we have to justify why the direct limits exist.  The way to prove this is to show that the directed system of $\SimplMor[\fm]{A}{B}$ becomes stationary for $\fm$ large enough. This is a corollary of the following lemma.

\begin{lem}\label{lem:directedsystemstationary}
Let $A$ be a \textCstar-algebra, $A'\subset A$ dense and invariant under the involution, the $G$-action and the grading automorphism of $A$.
Let $(\Q+i\Q)\langle A'\rangle$ denote the free $(\Q+i\Q)$-algebra generated by $A'$ and define
\[\fm(A')\coloneqq\{(F,k)\mid F\subset (\Q+i\Q)\langle A'\rangle\text{ finite}, k\in\N_+\}\,.\]
This set is directed by the partial order $(E,k)\leq (F,l)\mathrel{{:}{\Leftrightarrow}}E\subset F\wedge k\leq l$.
Let $B$ be another \textCstar-algebra and $\alpha\colon \fn\to\fm(A')$ a cofinal map.
\begin{enumerate}
\item If $\varphi\colon A\to\Simpl^\fn B$ is a $*$-ho\-mo\-mor\-phism, then there exists a $*$-ho\-mo\-mor\-phism $\psi\colon A\to \Simpl^{\fm(A')} B$ such that $\varphi$ is $\fn$-homotopic to $\alpha^*\circ\psi$.
\item If $\varphi,\psi\colon A\to\Simpl^{\fm(A')} B$ are $*$-ho\-mo\-mor\-phisms such that $\alpha^*\circ\varphi$ is $\fn$-homotopic to $\alpha^*\circ\psi$, then $\varphi$ is $\fm(A')$-homotopic to $\psi$.
\end{enumerate}
\end{lem}

\begin{cor}\label{cor:existenceofsimpletoticmorphismset}
The limit in \Cref{defn:morphismlimit} exists and the canonical map 
\[\SimplMor[\fm(A')]{A}{B}\to\SimplMor{A}{B}\]
is a natural bijection.\qed
\end{cor}

\begin{proof}[Proof of \Cref{lem:directedsystemstationary}]
Let us abbreviate $\tilde A\coloneqq (\Q+i\Q)\langle A'\rangle$ and note that it inherits the structure of a graded involutive $G$-$(\Q+i\Q)$-algebra from $A$ by the choice of $A'$.

For the first part, let $\varphi\colon A\to\Simpl^\fn B$ be a $*$-ho\-mo\-mor\-phism. We can choose a lift $\varphi'\colon A'\to\Cb(\Delta^\fn;B)$ of $\varphi|_{A'}$ which commutes with involution, grading and $G$-action.
It induces a graded $*$-ho\-mo\-mor\-phism $\tilde\varphi$ making the diagram
\[\xymatrix{
\tilde A\ar[r]^-{\tilde\varphi}\ar[d]^{\kappa}&\Cb(\Delta^\fn;B)\ar[d]^{\pi_{\fn}}
\\A\ar[r]^-{\varphi}&\Simpl^\fn B
}\]
commute, where $\kappa$ denotes the canonical $*$-ho\-mo\-mor\-phism which restricts to the inclusion of $A'$ into $A$ and $\pi_\fn$ denotes the quotient map.

Using the definition of the quotient norm on $\Simpl^\fn B$ and the directedness of $\fn$, we find for each $(F,k)\in\fm(A')$ an element $\beta(F,k)\in\fn$ such that
\begin{equation}\label{eq:quotientnorminequality}
\forall \tilde a\in F\colon \left\|\tilde\varphi(\tilde a)|_{|\Delta^{\fn\ab g(F,k)}|}\right\|\leq\left\|\pi_\fn(\tilde\varphi(\tilde a))\right\|+\frac1k=\left\|\varphi(\kappa(\tilde a))\right\|+\frac1k\,.
\end{equation}
Since the order relation on $\fm(A')$ is clearly well-founded, we can furthermore use Noetherian recursion to choose the $g(F,k)$ in such a way that the resulting map $\beta\colon\fm(A')\to\fn$ is order-preserving.
This map induces a $*$-ho\-mo\-mor\-phism $\beta^*\colon\Cb(\Delta^\fn;B)\to \Cb(\Delta^{\fm(A')};B)$ as before, but it might not map $\Cz(\Delta^\fn;B)$ into $\Cz(\Delta^{\fm(A')};B)$, because $\beta$ does not have to be cofinal.
Therefore, we do not get an induced $*$-ho\-mo\-mor\-phism $\Simpl^\fn B\to \Simpl^{\fm(A')} B$.

Nevertheless, we claim that this $*$-ho\-mo\-mor\-phism can at least be defined on the sub-\textCstar-algebra $\varphi(A)$ of $\Simpl^{\fn} B$. More precisely, note that $\pi_\fn(\tilde\varphi(\tilde A))$ is dense in $\varphi(A)$ and we claim that there is a $*$-ho\-mo\-mor\-phism $\beta^*\colon \varphi(A)\to \Simpl^{\fm(A')} B$ such that the diagram
\[\xymatrix@C=10ex{
\tilde\varphi(\tilde A)\ar[r]^-{\beta^*|_{\tilde\varphi(\tilde A)}}\ar[d]^{\pi_\fn|_{\tilde\varphi(\tilde A)}}&\Cb(\Delta^{\fm(A')};B)\ar[d]^{\pi_{\fm(A')}}
\\\varphi(A)\ar[r]^-{\beta^*}&\Simpl^{\fm(A')} B
}\]
commutes. Indeed, for $f\in\tilde\varphi(\tilde A)$ we have
\begin{align*}
\|\pi_{\fm(A')}(\beta^*(f))\|&=\lim_{(F,k)\in\fm(A')}\|\beta^*f|_{|\Delta^{\fm(A')\ab(F,k)}|}\|
\\&\leq \lim_{(F,k)\in\fm(A')}\|f|_{|\Delta^{\fn\ab \beta(F,k)}|}\|
\\&\leq \lim_{k\in\N_+}\|\pi_{\fn}(f)\|+\frac1k=\|\pi_{\fn}(f)\|\,,
\end{align*}
where the first inequality uses that $\beta$ was chosen in an order-preserving fashion and the second inequality comes from \eqref{eq:quotientnorminequality}. Thus, we obtain a $1$-Lipschitz graded $*$-ho\-mo\-mor\-phism $\pi_\fn(\tilde\varphi(\tilde A))\to \Simpl^{\fm(A')} B$ and its extension to the closure $\varphi(A)$ is the desired $*$-ho\-mo\-mor\-phism $\beta^*$.

We can now define $\psi\coloneqq \beta^*\circ\varphi$ and it remains to show that $\alpha^*\circ\psi=\alpha^*\circ\beta^*\circ\varphi$ is $\fn$-homotopic to $\varphi$.
To this end, we choose any $\gamma\colon\fn\to\fn$ such that $\gamma\geq\beta\circ\alpha$ and $\gamma\geq\id_{\fn}$. 
The second property implies that $\gamma$ is cofinal and hence $\gamma^*\colon\Simpl^\fn B\to\Simpl^\fn B$ is $\fn$-homotopic to the identity by \Cref{lem:independenceofcofinalmap}.
Let $h\colon \Delta^\fn B\times[0,1]\to\Delta^\fn B$ be the affine linear homotopy between $|\beta|\circ|\alpha|$ and $|\gamma|$. It maps $\Delta^{\fn\ab n}\times[0,1]$ to $\Delta^{\fn\ab\beta(\alpha(n))}$, because $\gamma\geq\beta\circ\alpha$, and hence a similar calculation to the one above shows that the map $h^*\colon \Cb(\Delta^\fn;B)\to\Cb(\Delta^\fn;\Ct([0,1];B))$
gives rise to a $*$-ho\-mo\-mor\-phism $\varphi(A)\to\Simpl^{\fm(A')}(\Ct([0,1];B))$, which is an $\fn$-homotopy between $\alpha^*\circ\beta^*$ and $\gamma^*|_{\varphi(A)}$.
In particular this shows that $\varphi$ and $\alpha^*\circ\psi=\alpha^*\circ\beta^*\circ\varphi$ are both $\fn$-homotopic to $\gamma^*\varphi$, finishing the proof of the first part.

To prove the second part, let $\Phi\colon A\to \Simpl^\fn(\Ct([0,1];B))$ be a $\fn$-homotopy between $\alpha^*\circ\varphi$ and $\alpha^*\circ\psi$. Performing the same construction as in the first part but for $\Phi$, we obtain a order-preserving map $\beta\colon\fm(A')\to \fn$ which induces a $*$-ho\-mo\-mor\-phism $\beta^*\colon\Phi(A)\to\Simpl^{\fm(A')}(\Ct([0,1];B))$. Moreover, since $\alpha$ is cofinal, $\beta$ can be chosen in such a way that it satisfies the additional property $\alpha\circ\beta\geq\id_{\fm(A')}$ and hence $\alpha\circ\beta$ is cofinal.
Then $\beta^*\circ\Phi$ is a $\fm(A')$-homotopy between $\beta^*\circ\alpha^*\circ\varphi=(\alpha\circ\beta)^*\circ\varphi$ and $\beta^*\circ\alpha^*\circ\psi=(\alpha\circ\beta)^*\circ\psi$. The claim now follows from \Cref{lem:independenceofcofinalmap}, which says that $(\alpha\circ\beta)^*\circ\varphi$ and $(\alpha\circ\beta)^*\circ\psi$ are $\fm(A')$-homotopic to $\varphi$ and $\psi$, respectively.
\end{proof}

\begin{cor}\label{cor:asymptoticsimpltoticrelation}
If $A'$ is countable, then the canonical map $\SimplMor[\N_+]{A}{B}\to\SimplMor{A}{B}$ is a natural bijection. 
Combining this with \Cref{ex:asymptoticsimpltoticrelation}, we get a natural bijection $\SimplMor[1]{A}{B}\to\SimplMor{A}{B}$. Note that such a countable $A'$ exists in particular whenever $G$ is countable and $A$ is separable.
\end{cor}

\begin{proof}
There is a canonical cofinal map $\fm(A')\to\N_+, (F,k)\mapsto k$. Since $A'$ is countable, so is $(\Q+i\Q)\langle A'\rangle$, say $(\Q+i\Q)\langle A'\rangle=\{a_1,a_2,a_3,\dots\}$. Then $\N_+\to\fm(A'),k\mapsto (\{a_1,\dots,a_k\},k)$ is also cofinal. Thus, $\N_+$ and $\fm(A')$ are isomorphic in $\cD$ and the claim follows from \Cref{lem:independenceofcofinalmap,lem:directedsystemstationary}.
\end{proof}

In this context, the following lemma is also useful. It is the simpltotic analogue of the fact that if an asymptotic morphism is given by a family of $*$-ho\-mo\-mor\-phisms, then it is $1$-homotopic to each of them.
Recall from \Cref{example:oneelementdirectedset} that we denoted by $\fo$ the one-element directed set, which is clearly a terminal element of $\cD$.

\begin{lem}
\label{lem:simpltoticmorphismgivenbystarhomomorphism}
For every $*$-ho\-mo\-mor\-phism $\varphi\colon A\to\Cb(|\Delta^\fm|;B)$, the class in $\SimplMor[\fm]{A}{B}$ represented by its quotient is the image of the homotopy class of $\ev{x}\circ\varphi$ for any $x\in|\Delta^\fm|$ under the canonical map $[A,B]=\SimplMor[\fo]{A}{B}\to\SimplMor[\fm]{A}{B}$.
\end{lem}

\begin{proof}
The affine linear homotopy between the constant map $|\Delta^\fm|\to |\Delta^\fm|,y\mapsto x$ and the identity map induces a $*$-ho\-mo\-mor\-phism \(A\to \Cb(|\Delta^\fm|;\Ct([0,1];B))\) whose quotient is an $\fm$-homotopy between the two elements.
Note that no cofinality of the homotopy is needed, because we are constructing the homotopy at the level of $\Cb$ and not at the level of the quotient by $\Cz$.
\end{proof}

\section{The composition product}
\label{sec:compositionproduct}

\subsection{Preliminary observations}\label{sec:productpreliminary}

Before going into the construction of the composition product 
\[\SimplMor{A}{B}\times \SimplMor{B}{C}\to \SimplMor{A}{C}\,,\]
we briefly sketch its rough outline first.
\Cref{cor:Fhomotopyclasscompositionproduct} provides us with a product \(\SimplMor[\fm]{A}{B}\times \SimplMor[\fn]{B}{C}\to \SimplMor[\Simpl^\fm\Simpl^\fn]{A}{C}\) and then we need the postcomposition with 
a $*$-ho\-mo\-mor\-phism $\Simpl^\fm\Simpl^\fn C\to\Simpl^\fk C$ for some directed set $\fk$ to go further into $\SimplMor[\fk]{A}{C}$.
To this end, we use that the canonical $*$-ho\-mo\-mor\-phism $\Cb(|\Delta^\fm|;\Cb(|\Delta^\fn|;C))\to \Simpl^\fm\Simpl^\fn C$ is surjective by \Cref{lem:asymptoticalgebrasexact} and we define a continuous map $\Theta\colon |\Delta^\fk|\to |\Delta^\fm|\times|\Delta^\fn|$ such that the $*$-ho\-mo\-mor\-phism 
\[\Cb(|\Delta^\fm|;\Cb(|\Delta^\fn|;C))\subset\Cb(|\Delta^\fm|\times|\Delta^\fn|;C)\xrightarrow{\Theta^*}\Cb(|\Delta^\fk|;C) \to\Simpl^\fk C\]
factors through $\Simpl^\fm\Simpl^\fn C$.

The difficult part is to define $\Theta$ in such a way that, among other properties, 
$\Theta^*$ maps $\Cb(|\Delta^\fm|;\Cz(|\Delta^\fn|;C))$ into $\Cz(|\Delta^\fk|;C)$.
Recall that we also used the symbol $\Delta^\fm$ for the set of all its simplices.
\begin{lem}\label{lem:CbCzcharacterization}
A bounded continuous function $f\in\Cb(|\Delta^\fm|;\Cb(|\Delta^\fn|;C))$ lies in the ideal $\Cb(|\Delta^\fm|;\Cz(|\Delta^\fn|;C))$ if and only if 
\[\forall\sigma\in\Delta^\fm\forall\varepsilon>0\exists n\in\fn\colon\|f|_{|\Delta^\sigma|\times|\Delta^{\fn\ab n}|}\|<\varepsilon\,.\]
\end{lem}
\begin{proof}
Given $f\in\Cb(|\Delta^\fm|;\Cz(|\Delta^\fn|;C))$ and $\varepsilon>0$, the sets 
\[U_n\coloneqq\{x\in|\Delta^\fm|\mid \|f(x)|_{\Delta^{\fn\ab n}}\|<\varepsilon\}\]
for $n\in\fn$ are open and cover $|\Delta^\fm|$. Therefore, each of the compact subsets $|\Delta^\sigma|$ is contained in a finite union $U_{n_1}\cup\dots\cup U_{n_l}$ which in turn is contained in $U_n$ for $n\geq n_1,\dots,n_l$. This implies one direction of the claim. 

The other direction is trivial.
\end{proof}

The lemma shows that the map $\Theta$ should\footnote{A short consideration, which we leave to the reader, even shows that $\Theta$ \emph{must} satisfy this property.} satisfy the following property:
For all maps $\xi\colon\Delta^\fm\to\fn$ there exists $k\in\fk$ such that $\Theta(\Delta^{\fk\ab k})\subset \bigcup_{\sigma\in \Delta^\fm}|\Delta^\sigma|\times|\Delta^{\fn\ab \xi(\sigma)}|$.
Unfortunately, this cannot be satisfied by a product of two simplicial maps, which is disappointing and makes the construction of $\Theta$ a lot more complicated.
Instead we construct a subspace $X\subset|\Delta^\fk|\times|\Delta^\fn|$ equipped with a filter such that the projection onto its first factor $|\Delta^\fk|$ is a homotopy equivalence and then $\Theta$ will be defined as the composition $|\Delta^\fk|\simeq X\xrightarrow{|\mu|\times\id}|\Delta^\fm|\times|\Delta^\fn|$, where $\mu\colon \fk\to\fm$ is a simplicial map.
As a result, we will obtain a natural map
\[\Theta^*\colon \SimplMor[\Simpl^\fm\Simpl^\fn]{A}{C}\xrightarrow{(|\mu|\times\id)^*} \SimplMor[X]{A}{C}\cong \SimplMor[\fk]{A}{C}\]
that we can use to define the composition product.

\subsection{Construction}\label{sec:productconstruction}

Now that we have an idea about what we want, let us start with the construction.

\begin{defn}
Let $\fn^{\Delta^\fm}$ be the set of maps from the set of simplices of $\Delta^\fm$ (which we also denoted by $\Delta^\fm$) 
to $\fn$ and define $\fm\sharp\fn\coloneqq\fm\times \fn^{\Delta^\fm}$, which becomes a directed set when equipped with the canonical partial order $(m,\xi)\geq(m',\xi')\mathrel{{:}{\Leftrightarrow}}m\geq m'\wedge \xi\geq\xi'$.
\end{defn}

This is the directed set which takes the role of the $\fk$ mentioned in \Cref{sec:productpreliminary}. 
The simplicial map $\mu=\mu_{\fm,\fn}\colon \fm\sharp\fn\to\fm$ is defined as the projection onto the first component.

\begin{defn}
For any map $\vartheta\colon\Delta^\fk\to\fn$ we define the space 
\[X_\vartheta\coloneqq\bigcup_{\sigma\in \Delta^{\fk}}|\Delta^\sigma|\times|\Delta^{\fn\ab\vartheta(\sigma)}|\subset |\Delta^{\fk}|\times|\Delta^\fn|\]
and equip it with the filter generated by the subsets 
\[X_\vartheta^{k}\coloneqq X_\vartheta\cap (|\Delta^{\fk\ab k}|\times|\Delta^\fn|)\]
with $k\in\fk$.
\end{defn}

For $\fk=\fm\sharp\fn$ and a special choice of $\vartheta$, the space $X_\vartheta$ will be the space $X$ mentioned at the end of \Cref{sec:productpreliminary}. Before we construct this specific $\vartheta$, let us prove the following general fact.

\begin{lem}\label{lem:Xhomotopyequivalence}
Assume that the map $\vartheta\colon\Delta^\fk\to\fn$ is monotone in the sense that $\vartheta(\sigma)\geq\vartheta(\varsigma)$ for all $\emptyset\not=\varsigma\subset\sigma\in\Delta^\fk$. Then the projection $P_\vartheta\colon X_\vartheta\to |\Delta^\fk|$ onto the first component is a homotopy equivalence.
\end{lem}

\begin{proof}
The projection $P_\vartheta$ is clearly continuous and filter-proper. A homotopy inverse $I_\vartheta\colon |\Delta^\fk| \to X_\vartheta$ can be constructed by  recursion over the $i$-skeleta $\Delta_i^\fk$:
On the zero-skeleton $|\Delta^{\fk}_0|=\fk$ we can define \[I_\vartheta(k)\coloneqq (k,\vartheta(k))\in\fk\times\fn=|\Delta^{\fk}_0|\times|\Delta^\fn_0|\]
and on a higher $i$-simplex $\sigma$ we define it by precomposing 
\begin{align*}
I_\vartheta^\sigma\colon |\Delta^\sigma|\times\{1\}\cup|\partial\Delta^\sigma|\times[0,1]&\to \bigcup_{\varsigma\subset\sigma} |\Delta^\varsigma|\times|\Delta^{\fn\ab\vartheta(\varsigma)}|\subset X_\vartheta
\\(x,t)&\mapsto (1-t)I_\vartheta(x)+t(x,\vartheta(\sigma))
\end{align*}
with any continuous map $h_\sigma\colon|\Delta^\sigma|\to|\Delta^\sigma|\times\{1\}\cup|\partial\Delta^\sigma|\times[0,1]$ which restricts to the identity $|\partial\Delta^\sigma|\to |\partial\Delta^\sigma|\times\{0\}$ on the boundary.
Here, the monotonicity assumption on $\vartheta$ is used to ensure that $|\Delta^{\fn\ab\vartheta(\sigma)}|\subset|\Delta^{\fn\ab\vartheta(\varsigma)}|$, because otherwise the segment between $I_\vartheta(x)$ and $(x,\vartheta(\sigma))$ for $x\in|\partial\Delta^\sigma|$ might not lie in $X_\vartheta$.
Clearly, $I_\vartheta$ is continuous and it is also filter-proper, because $I_\vartheta(|\Delta^\sigma|)\subset X_\vartheta\cap|\Delta^\sigma|\times|\Delta^\fn|$ for all simplices $\sigma$ and hence $I_\vartheta(|\Delta^{\fk\ab k}|)\subset X_\vartheta^k$ for all $k\in\fk$.

Indeed, $P_\vartheta\circ I_\vartheta$ is homotopic to the identity via the affine linear interpolation $(1-s)\cdot P_\vartheta\circ I_\vartheta+s\cdot\id_{|\Delta^{\fk}|}$ and similarily the other composition $I_\vartheta\circ P_\vartheta$ can be homotoped to the identity on $X_\vartheta$ by first adjusting the $|\Delta^\fn|$-component and then the $|\Delta^\fk|$-component in an affine linear manner. Both homotopies are clearly continuous and filter-proper, too.
\end{proof}

The map $\vartheta=\vartheta_{\fm,\fn}\colon\Delta^{\fm\sharp\fn}\to\fn$ is now defined by recursion over the $i$-skeleta $\Delta^{\fm\sharp\fn}_i$ such that it satisfies the monotonicity assumption of \Cref{lem:Xhomotopyequivalence} and the inequalities
\begin{equation}
\vartheta_{\fm,\fn}(\sigma)\geq\xi(\mu_{\fm,\fn}(\sigma))\qquad\text{for all }\sigma\in\Delta^{\fm\sharp\fn}\text{ and }(m,\xi)\in\sigma\,.\label{eq:varthetabiggerxi}
\end{equation}
We will often omit the indices $\fm,\fn$ from $\mu,\vartheta$ if there is no ambiguity.
Although we will not need it, we point out that it is always possible to choose $\vartheta_{\fm,\fn}$ on the zero-skeleton as $\vartheta_{\fm,\fn}(m,\xi)\coloneqq\xi(m)$.

\begin{lem}\label{lem:doublesimpleasymp}
For $\vartheta=\vartheta_{\fm,\fn}$ and $\mu=\mu_{\fm,\fn}$, the $*$-ho\-mo\-mor\-phism 
\[\Cb(|\Delta^\fm|;\Cb(|\Delta^\fn|;C))\subset\Cb(|\Delta^\fm|\times|\Delta^\fn|;C)\xrightarrow{(|\mu|\times\id)^*}\Cb(X_\vartheta;C)\]
descends to a $*$-ho\-mo\-mor\-phism
\[(|\mu|\times\id)^*\colon \Simpl^\fm\Simpl^\fn C\to\Ca(X_\vartheta;C)\]
and postcomposition with it yields a map
\[(|\mu|\times\id)^*\colon \SimplMor[\Simpl^\fm\Simpl^\fn]{A}{C}\to \SimplMor[X_\vartheta]{A}{C}\]
which is natural in both $A$ and $C$.
\end{lem}
\begin{proof}
Property \eqref{eq:varthetabiggerxi} of the map $\vartheta$ implies that we have
\begin{equation*}(|\mu|\times\id)(X_\vartheta^{(m,\xi)})=\bigcup_{\sigma\in \Delta^{\fm\sharp\fn\ab(m,\xi)}}|\Delta^{\mu(\sigma)}|\times|\Delta^{\fn\ab\vartheta(\sigma)}|\subset\bigcup_{\rho\in \Delta^{\fm\ab m}}|\Delta^{\rho}|\times|\Delta^{\fn\ab\xi(\rho)}|
\end{equation*}
and using \Cref{lem:CbCzcharacterization} we see that $(|\mu|\times\id)^*$ maps $\Cb(|\Delta^\fm|;\Cz(|\Delta^\fn|;C))$ into $\Cz(X_\vartheta;C)$ and hence descends to a $*$-ho\-mo\-mor\-phism 
\[\Cb(|\Delta^\fm|;\Simpl^\fn C)\xrightarrow{(|\mu|\times\id)^*}\Ca(X_\vartheta;C)\,.\]
Subsequently the same inclusion of sets implies that it also factors through $\Simpl^\fm\Simpl^\fn C$.
This construction gives us a natural transformation
\((|\mu|\times\id)^*\colon \Simpl^\fm\Simpl^\fn\to \Ca(X_\vartheta;\blank)\)
and hence a we obtain a map
\((|\mu|\times\id)^*\colon \SimplMor[\Simpl^\fm\Simpl^\fn]{A}{C}\to \SimplMor[X_\vartheta]{A}{C}\)
by \Cref{lem:functorialityundernatrualtransformationsoffunctors}.
\end{proof}

Defining $\Theta\coloneqq (|\mu|\times\id)\circ I_\vartheta$, we now obtain the map
\[\Theta^*\colon \SimplMor[\Simpl^\fm\Simpl^\fn]{A}{C}\to \SimplMor[\fm\sharp\fn]{A}{C}\,.\]
\begin{lem}
$\Theta^*$ is independent of all choices.
\end{lem}
\begin{proof}
Independence from the choice of the $h_\sigma$ follows from the fact that $\Theta^*$ is equal to $(P_\vartheta^*)^{-1}\circ(|\mu|\times\id)^*$. For independence from the concrete choice $\vartheta$ for $\vartheta_{\fm,\fn}$, note that if $\vartheta'$ is another choice then we can assume without loss of generality $\vartheta'\geq\vartheta$ and obtain a commutative diagram
\[\xymatrix@R=0ex@C=10ex{
&X_{\vartheta'}\ar[dl]_-{|\mu|\times\id}\ar[dr]^-{P_{\vartheta'}}&
\\|\Delta^\fm|\times|\Delta^\fn|&\cap&|\Delta^{\fm\sharp\fn}|
\\&X_{\vartheta}\ar[ul]^-{|\mu|\times\id}\ar[ur]_-{P_\vartheta}&
}\]
where the inclusion $X_{\vartheta'}\subset X_\vartheta$ is continuous and filter-proper. The claim follows by considering the induced diagram.
\end{proof}

By composing the product from \Cref{cor:Fhomotopyclasscompositionproduct} with $\Theta^*$ we can now define the composition product
\[\SimplMor[\fm]{A}{B}\times\SimplMor[\fn]{B}{C}\to\SimplMor[\fm\sharp\fn]{A}{C}\,,\quad(\llbracket\varphi\rrbracket,\llbracket\psi\rrbracket)\mapsto \llbracket\psi\rrbracket\circ\llbracket\varphi\rrbracket\coloneqq \Theta^*(\llbracket \Simpl^\fm(\psi)\circ\varphi\rrbracket)\,.\]

\begin{thm}\label{thm:compprodpassestolimits}
The composition product passes to the direct limits to yield a composition product \(\SimplMor{A}{B}\times\SimplMor{B}{C}\to\SimplMor{A}{C}\). More precisely:
\begin{enumerate}
\item Given a cofinal map $\beta\colon\fn'\to\fn$, the map 
\[\fm\sharp\beta\colon\fm\sharp\fn'\to\fm\sharp\fn\,,\quad (m,\xi)\mapsto (m,\beta\circ\xi)\]
is cofinal, too, and the diagram
\[\xymatrix{
\SimplMor[\fm]{A}{B}\times\SimplMor[\fn]{B}{C}\ar[r]^-{\circ}\ar[d]^{\id\times\beta^*}&\SimplMor[\fm\sharp\fn]{A}{C}\ar[d]^{(\fm\sharp\beta)^*}
\\\SimplMor[\fm]{A}{B}\times\SimplMor[\fn']{B}{C}\ar[r]^-{\circ}&\SimplMor[\fm\sharp\fn']{A}{C}
}\]
commutes.
\item Given a cofinal map $\alpha\colon\fm'\to\fm$, the set $\fk\coloneqq\fm'\times\fn^{\Delta^{\fm'}}\times\fn^{\Delta^{\fm}}$ equipped with the canonical partial order is directed. Moreover, the two maps 
\[\gamma\colon\fk\to\fm'\sharp\fn,(m,\xi,\zeta)\mapsto (m,\xi)\,,\quad\delta\colon\fk\to\fm\sharp\fn,(m,\xi,\zeta)\mapsto (\alpha(m),\zeta)\]
are cofinal and the diagram
\[\xymatrix@R=0.5ex{
\SimplMor[\fm]{A}{B}\times\SimplMor[\fn]{B}{C}\ar[r]^-{\circ}\ar[dd]^{\alpha^*\times\id}&\SimplMor[\fm\sharp\fn]{A}{C}\ar[dr]^{\delta^*}&
\\&&\SimplMor[\fk]{A}{C}
\\\SimplMor[\fm']{A}{B}\times\SimplMor[\fn]{B}{C}\ar[r]^-{\circ}&\SimplMor[\fm'\sharp\fn]{A}{C}\ar[ur]_{\gamma^*}&
}\]
commutes.

\end{enumerate}
\end{thm}
\begin{proof}
For the first part, it is clear that $\fm\sharp\beta$ is cofinal.

Let $\vartheta=\vartheta_{\fm,\fn}\colon\Delta^{\fm\sharp\fn}\to\fn$ be constructed as usual by recursion satisfying the monotonicity assumption of \Cref{lem:Xhomotopyequivalence} and \eqref{eq:varthetabiggerxi}.
Analogously we choose $\vartheta'=\vartheta_{\fm,\fn'}\colon\Delta^{\fm\sharp\fn'}\to\fn'$ with the only difference that we additionally demand $\beta(n)\geq \vartheta(\beta(\sigma))$ for all $n\geq\vartheta'(\sigma)$, which is possible because $\beta$ is cofinal.
This last property entails that $|\fm\sharp\beta|\times|\beta|$ maps $X_{\vartheta'}=\bigcup_{\sigma\in \Delta^{\fm\sharp\fn'}}|\Delta^\sigma|\times|\Delta^{\fn'\ab\vartheta'(\sigma)}|$ into 
$\bigcup_{\sigma\in \Delta^{\fm\sharp\fn'}}|\Delta^{\fm\sharp\beta(\sigma)}|\times|\Delta^{\fn\ab\vartheta(\beta(\sigma))}|\subset X_\vartheta$ and its restriction to a map $X_{\vartheta'}\to X_\vartheta$ is clearly continuous and filter-proper. 
Therefore, we get a commutative diagram of spaces
\[\xymatrix@C=10ex{
|\Delta^\fm|\times|\Delta^{\fn}|&\ar[l]_-{|\mu|\times\id}X_{\vartheta}\ar[r]^-{P_\vartheta}&|\Delta^{\fm\sharp\fn}|
\\|\Delta^\fm|\times|\Delta^{\fn'}|\ar[u]_{\id\times|\beta|}&\ar[l]_-{|\mu'|\times\id}X_{\vartheta'}\ar[r]^-{P_{\vartheta'}}\ar[u]_{|\fm\sharp\beta|\times|\beta|}&|\Delta^{\fm\sharp\fn'}|\ar[u]_{|\fm\sharp\beta|}
}\]
which induces the following commutative diagram of \textCstar-algebras
\[\xymatrix@C=10ex{
\Simpl^\fm B\ar[r]^-{\Simpl^\fm(\psi)}\ar@{=}[d]&\Simpl^\fm\Simpl^\fn C\ar[r]^-{(|\mu|\times\id)^*}\ar[d]^{\Simpl^\fm(\beta^*)}&\Ca(X_\vartheta;C)\ar[d]^{(|\fm\sharp\beta|\times|\beta|)^*}&\Simpl^{\fm\sharp\fn}C\ar[l]_-{P_\vartheta^*}\ar[d]^{(\fm\sharp\beta)^*}
\\\Simpl^\fm B\ar[r]^-{\Simpl^\fm(\beta^*\circ\psi)}&\Simpl^\fm\Simpl^{\fn'} C\ar[r]^-{(|\mu'|\times\id)^*}&\Ca(X_{\vartheta'};C)&\Simpl^{\fm\sharp\fn'}C\ar[l]_-{P_{\vartheta'}^*}
}\]
where $\mu=\mu_{\fm,\fn}$ and $\mu'=\mu_{\fm,\fn'}$.
The first part now follows by composing with $\varphi\colon A\to\Simpl^\fm B$ and taking the asymptotic homotopy classes.

For the second part, it is also obvious that $\fk$ is directed and $\gamma,\delta$ are cofinal. Similar to the first part, we find a commutative diagram
\[\xymatrix@C=10ex{
|\Delta^\fm|\times|\Delta^{\fn}|&\ar[l]_-{|\mu|\times\id}X_{\vartheta}\ar[r]^-{P_\vartheta}&|\Delta^{\fm\sharp\fn}|
\\&X_{\tilde\vartheta}\ar[u]_{|\delta|\times\id}\ar[d]^{|\gamma|\times\id}\ar[r]^{P_{\tilde\vartheta}}&|\Delta^\fk|\ar[u]_{\delta^*}\ar[d]^{\gamma^*}
\\|\Delta^{\fm'}|\times|\Delta^{\fn}|\ar[uu]_{|\alpha|\times\id}&\ar[l]_-{|\mu'|\times\id}X_{\vartheta'}\ar[r]^-{P_{\vartheta'}}&|\Delta^{\fm\sharp\fn'}|
}\]
where $\vartheta=\vartheta_{\fm,\fn}\colon\Delta^{\fm\sharp\fn}\to\fn,\vartheta'=\vartheta_{\fm',\fn}\colon\Delta^{\fm'\sharp\fn}\to\fn$ are constructed recursively according to the monotonicity assumption of \Cref{lem:Xhomotopyequivalence} and the inequalities \eqref{eq:varthetabiggerxi} and $\vartheta$ is constructed recursively such that it satisfies again the monotonicity assumption and furthermore the inequalities
$\tilde\vartheta\geq\delta^*\vartheta$ and $\tilde\vartheta\geq\gamma^*\vartheta'$.
Clearly, the two vertical maps leaving $X_{\tilde\vartheta}$ are continuous and filter-proper.
The second part now follows exactly as the first part by considering the induced diagram on \textCstar-algebras and passing to the asymptotic homotopy classes.
\end{proof}

\subsection{Properties}\label{sec:productproperties}

In this section we prove basic properties of the composition product defined in \Cref{thm:compprodpassestolimits}, in particular those which show that the \(\SimplMor{A}{B}\) are the morphism sets of a category.

\begin{thm}\label{thm:associativity}
The composition product is associative.
\end{thm}

The theorem follows readily from the following useful lemma.

\begin{lem}\label{lem:Thetaassociativity}
For all directed sets $\fl,\fm,\fn$ there exists a directed set $\fk$ and two cofinal maps $\alpha\colon\fk\to(\fl\sharp\fm)\sharp\fn$, $\beta\colon\fk\to \fl\sharp(\fm\sharp\fn)$ such that the two natural transformations 
$\alpha^*\circ \Theta_{\fl\sharp\fm,\fn}^*\circ(\Theta_{\fl,\fm}^*)_{\Simpl^\fn(\blank)}$ and $\beta^*\circ\Theta_{\fl,\fm\sharp\fn}\circ\Simpl^\fl(\Theta_{\fm,\fn}^*)$ from $\Simpl^\fl\Simpl^\fm\Simpl^\fn$ to $\Simpl^\fk$ are $\Simpl^\fk$-homotopic.
\end{lem}

The directed set $\fk$ is necessary, because there is no obvious cofinal map between the two directed sets
\begin{align*}
(\fl\sharp\fm)\sharp\fn&=(\fl\sharp\fm)\times \fn^{\Delta^{\fl\sharp\fm}}=\fl\times\fm^{\Delta^\fl}\times\fn^{\Delta^{\fl\sharp\fm}}
\\\fl\sharp(\fm\sharp\fn)&=\fl\times(\fm\sharp\fn)^{\Delta^\fl}=\fl\times(\fm\times\fn^{\Delta^\fm})^{\Delta^\fl}=\fl\times\fm^{\Delta^\fl}\times\fn^{\Delta^\fl\times\Delta^\fm}\,.
\end{align*}
There is a canonical map $\Delta^{\fl\sharp\fm}\to \Delta^\fl\times\Delta^\fm$ defined on the vertices by $(l,\zeta)\mapsto (l,\zeta(l))$, but the induced map $\fl\sharp(\fm\sharp\fn)\to (\fl\sharp\fm)\sharp\fn$ is in general not cofinal.

\begin{proof}[Proof of \Cref{thm:associativity}]
Given $\llbracket\chi\rrbracket\in\SimplMor[\fl]{A}{B}$, $\llbracket\varphi\rrbracket\in\SimplMor[\fm]{B}{C}$ and $\llbracket\psi\rrbracket\in\SimplMor[\fn]{C}{D}$, the two compositions are 
$(\llbracket\chi\rrbracket\circ\llbracket\varphi\rrbracket)\circ\llbracket\psi\rrbracket\in\SimplMor[(\fl\sharp\fm)\sharp\fn]{A}{D}$ and $\llbracket\chi\rrbracket\circ(\llbracket\varphi\rrbracket\circ\llbracket\psi\rrbracket)\in \SimplMor[\fl\sharp(\fm\sharp\fn)]{A}{D}$.
These are mapped under $\alpha^*$ and $\beta^*$, respectively, to the classes obtained from 
\[\llbracket\chi\rrbracket \circ \llbracket\varphi\rrbracket \circ \llbracket\psi\rrbracket = \llbracket\Simpl^\fl\Simpl^\fm\psi\circ\Simpl^\fl\varphi\circ\chi\rrbracket \in\SimplMor[\Simpl^\fl\Simpl^\fm\Simpl^\fn]{A}{D}\]
by applying the functoriality under the two natural transformations from \Cref{lem:Thetaassociativity}. Thus, the claim follows from \Cref{lem:functorialityundernatrualtransformationsoffunctors}.
\end{proof}

\begin{proof}[Proof of \Cref{lem:Thetaassociativity}]
We define the directed set $\fk\coloneqq \fl\times\fm^{\Delta^\fl}\times\fn^{\Delta^{\fl\sharp\fm}}\times\fn^{\Delta^\fl\times\Delta^\fm}$ and consider the two canonical cofinal maps 
\begin{align*}
\alpha\colon\fk&\to(\fl\sharp\fm)\sharp\fn\,,\quad (l,\zeta,\eta,\kappa)\mapsto (l,\zeta,\eta)\,,
\\\beta\colon\fk&\to \fl\sharp(\fm\sharp\fn)\,,\quad (l,\zeta,\eta,\kappa)\mapsto (l,\zeta,\kappa)\,.
\end{align*}
The plan is to construct a diagram of natural transformations
\[\xymatrix@C=12ex{
\Ca(X_{\vartheta_{\fl,\fm}};\Simpl^\fn(\blank))\ar[dr]
&\Simpl^\fl\Simpl^\fm\Simpl^\fn
\ar[l]_-{(|\mu_{\fl,\fm}|\times\id)^*_{\Simpl^\fn(\blank)}}
\ar[r]^-{\Simpl^\fl((|\mu_{\fm,\fn}|\times\id)^*)}
&\Simpl^\fl\Ca(X_{\vartheta_{\fm,\fn}};\blank)\ar[dl]
\\\Simpl^{\fl\sharp\fm}\Simpl^\fn\ar[u]^{(P_{\vartheta_{\fl,\fm}}^*)_{\Simpl^\fn(\blank)}}_{\simeq}\ar[d]_{(|\mu_{\fl\sharp\fm,\fn}|\times\id)^*}
&\Ca(X_{\vartheta_1,\vartheta_2};\blank)
&\Simpl^\fl\Simpl^{\fm\sharp\fn}\ar[u]_{\Simpl^\fl(P_{\fm,\fn}^*)}^{\simeq}\ar[d]^{(|\mu_{\fl,\fm\sharp\fn}|\times\id)^*}
\\\Ca(X_{\vartheta_{\fl\sharp\fm,\fn}};\blank)\ar[r]^-{(|\alpha|\times |\nu_{\fm,\fn}|)^*}
&\Ca(X_{\vartheta_1};\blank)\ar[u]^{\simeq}_{P_{\vartheta_1,\vartheta_2}^*}
&\Ca(X_{\vartheta_{\fl,\fm\sharp\fn}};\blank)\ar[l]_{(|\beta|\times \id)^*}
\\\Simpl^{(\fl\sharp\fm)\sharp\fn}\ar[u]_{P_{\vartheta_{\fl\sharp\fm,\fn}}^*}^{\simeq}\ar[r]_{\alpha^*}
&\Simpl^\fk\ar[u]_{P_{\vartheta_1}^*}^{\simeq}
&\Simpl^{\fl\sharp(\fm\sharp\fn)}\ar[l]^{\beta^*}\ar[u]_{P_{\vartheta_{\fl,\fm\sharp\fn}}^*}^{\simeq}
}\]
in which the left pentagon commutes up to $\Ca(X_{\vartheta_1,\vartheta_2};\blank)$-homotopy and the rest commutes strictly. 
If we then replace the $P$-arrows by corresponding $I$-arrows in the other direction, the different parts of the diagram will still commute up to the appropiate notions of $F$-homotopies and the claim follows from \Cref{lem:functorialityundernatrualtransformationsoffunctors}.

The outer functors and natural transformations in the diagram above are clear. 
To fill in the middle part of the diagram, let $\nu_{\fm,\fn}\colon\fm\sharp\fn\to\fn, (m,\xi)\mapsto\xi(m)$ and  define by recursion over the skeleta a function $\vartheta_1\colon\Delta^\fk\to\fm\sharp\fn$ satisfying the monotonicity assumption of \Cref{lem:Xhomotopyequivalence} and the following four properties for all $\sigma\in\Delta^\fk$:
\begin{enumerate}
\item $\forall (m,\xi)\geq\vartheta_1(\sigma)\colon\nu_{\fm,\fn}(m,\xi)\geq\vartheta_{\fl\sharp\fm,\fn}(\alpha(\sigma))$.

This is possible, because $\nu_{\fm,\fn}$ is cofinal and it implies that $|\alpha|\times|\nu_{\fm,\fn}|$ maps $X_{\vartheta_1}$ into $X_{\vartheta_{\fl\sharp\fm,\fn}}$. Moreover, the restricted map is filter-proper, because $\alpha$ is cofinal.

\item $\vartheta_1(\sigma)\geq\vartheta_{\fl,\fm\sharp\fn}(\beta(\sigma))$.

This implies that $|\beta|\times\id$ maps $X_{\vartheta_1}$ into $X_{\vartheta_{\fl,\fm\sharp\fn}}$. Cofinality of $\beta$ implies filter-properness of the restriction.

\item\label{item:vartheta1toXvarthetalm} $\forall (m,\xi)\geq \vartheta_1(\sigma)\colon \mu_{\fm,\fn}(m,\xi)\geq\vartheta_{\fl,\fm}(\mu_{\fl\sharp\fm,\fn}\circ\alpha(\sigma))$.

This is possible, because $\mu_{\fm,\fn}$ is cofinal, and it implies that $|\mu_{\fl\sharp\fm,\fn}\circ\alpha|\times|\mu_{\fm,\fn}|$ maps $X_{\vartheta_1}$ into $X_{\vartheta_{\fl,\fm}}$. This restriction is also filter-proper, because $\mu_{\fl\sharp\fm,\fn}\circ\alpha$ is cofinal.

\item\label{item:vartheta1geqzetakappa} $\forall (l,\zeta,\eta,\kappa)\in\sigma\colon \vartheta_1(\sigma)\geq(\zeta,\kappa)(\mu_{\fl,\fm\sharp\fn}\circ\beta(\sigma))$, compare \eqref{eq:varthetabiggerxi}.
\end{enumerate}
The restricted maps from the first two properties establish the lower part of the diagram and commutativity of the two squares is clear.
The other two properties will be needed shortly.

The upper part of the diagram will be induced by the diagram
\[\xymatrix@C=5em{
X_{\vartheta_1}\ar[r]^{|\beta|\times\id}\ar[d]_{|\alpha|\times|\nu_{\fm,\fn}|}
&X_{\vartheta_{\fl,\fm\sharp\fn}}\ar[r]^-{|\mu_{\fl,\fm\sharp\fn}|\times\id}
&|\Delta^\fl|\times|\Delta^{\fm\sharp\fn}|
\\X_{\vartheta_{\fl\sharp\fm,\fn}}\ar[d]_{|\mu_{\fl\sharp\fm,\fn}|\times\id}
&X_{\vartheta_1,\vartheta_2}\ar[ul]^{\simeq}_{P_{\vartheta_1,\vartheta_2}}\ar[r]^-{|\mu_{\fl,\fm\sharp\fn}\circ\beta|\times\id\times\id}\ar[d]^{|\mu_{\fl\sharp\fm,\fn}\circ\alpha|\times|\mu_{\fm,\fn}|\times\id}
&|\Delta^\fl|\times X_{\vartheta_{\fm,\fn}}\ar[u]_{\id\times P_{\vartheta_{\fm,\fn}}}\ar[d]^{\id\times |\mu_{\fm,\fn}|\times\id}
\\|\Delta^{\fl\sharp\fm}|\times|\Delta^\fn|
&X_{\vartheta_{\fl,\fm}}\times|\Delta^\fn|\ar[l]_-{P_{\vartheta_{\fl,\fm}}\times\id}\ar[r]^-{|\mu_{\fl,\fm}|\times\id\times\id}
&|\Delta^\fl|\times|\Delta^\fm|\times|\Delta^\fn|
}\]
of spaces which we explain now.

For a function $\vartheta_2\colon\Delta^\fk\times\Delta^{\fm\sharp\fn}\to\fn$, which we choose further below, we define the space
\[X_{\vartheta_1,\vartheta_2}\coloneqq \bigcup_{\sigma\in\Delta^\fk}|\Delta^\sigma|\times\bigcup_{\tau\in\Delta^{\fm\sharp\fn\ab\vartheta_1(\sigma)}}|\Delta^\tau|\times|\Delta^{\fn\ab\vartheta_2(\sigma,\tau)}|\subset |\Delta^\fk|\times|\Delta^{\fm\sharp\fn}|\times|\Delta^\fn|\]
equipped with the filter generated by the subsets 
\[X_{\vartheta_1,\vartheta_2}^k\coloneqq X_{\vartheta_1,\vartheta_2}\cap |\Delta^{\fk\ab k}|\times|\Delta^{\fm\sharp\fn}|\times|\Delta^\fn|\]
and let $P_{\vartheta_1,\vartheta_2}\colon X_{\vartheta_1,\vartheta_2}\to X_{\vartheta_1}$ denote the projection onto the first two components, which is a continuous and filter proper map.

Of course, it would be sufficient to define $\vartheta_2$ on the set of cells of $X_{\vartheta_1}$, which is a subset of $\Delta^\fk\times\Delta^{\fm\sharp\fn}$, but for notational convenience we let it be defined on the whole product set.
It should satisfy a monotonicity assumption analogous to that one of \Cref{lem:Xhomotopyequivalence}: If $\emptyset\not=\varsigma\subset\sigma\in\Delta^\fk$ and $\emptyset\not=\rho\subset\tau\in\Delta^{\fm\sharp\fn\ab\vartheta_1(\sigma)}$, then $\vartheta_2(\sigma,\tau)\geq\vartheta_2(\varsigma,\rho)$.
Then a construction analogous to the one in the proof of the aforementioned lemma implies that $P_{\vartheta_1,\vartheta_2}$ is a homotopy equivalence.

Moreover, $\vartheta_2$ should satisfy the following properties for all $(\sigma,\tau)\in\Delta^\fk\times\Delta^{\fm\sharp\fn}$.
\begin{enumerate}\setcounter{enumi}{4}
\item $\vartheta_2(\sigma,\tau)\geq\vartheta_{\fm,\fn}(\tau)$.

This implies that $|\mu_{\fl,\fm\sharp\fn}\circ\beta|\times\id\times\id$ maps $X_{\vartheta_1,\vartheta_2}$ into $|\Delta^\fl|\times X_{\vartheta_{\fm,\fn}}$.

Moreover, using property \ref{item:vartheta1geqzetakappa} and the obvious analogues of \Cref{lem:CbCzcharacterization,lem:doublesimpleasymp} for the spaces $|\Delta^\fl|$, $X_{\vartheta_{\fm,\fn}}$, $X_{\vartheta_1,\vartheta_2}$ instead of $\Delta^\fm,\Delta^\fn, X_\vartheta$ one shows that $|\mu_{\fl,\fm\sharp\fn}\circ\beta|\times\id\times\id$ induces a natural transformation $\Simpl^\fl\Ca(X_{\vartheta_{\fm,\fn}};\blank)\to \Ca(X_{\vartheta_1,\vartheta_2};\blank)$.

\item $\forall(l,\zeta,\eta,\kappa)\in\sigma\colon \vartheta_2(\sigma,\tau)\geq\kappa (\mu_{\fl\sharp\fm,\fn}\circ\alpha(\sigma),\mu_{\fm,\fn}(\tau))$, compare \eqref{eq:varthetabiggerxi}.

This property implies, again using obvious analogues of \Cref{lem:CbCzcharacterization,lem:doublesimpleasymp}, that the map $|\mu_{\fl\sharp\fm,\fn}\circ\alpha|\times|\mu_{\fm,\fn}|\times\id\colon X_{\vartheta_1,\vartheta_2}\to X_{\vartheta_{\fl,\fm}}\times|\Delta^\fn|$, which exists because of property \ref{item:vartheta1toXvarthetalm}, induces a natural transformation $\Ca(X_{\vartheta_{\fl,\fm}};\Simpl^\fn(\blank))\to \Ca(X_{\vartheta_1,\vartheta_2};\blank)$.
\end{enumerate}

A function $\vartheta_2$ with all of these properties can clearly be constructed as before by induction over the $i$-skeleta of $\Delta^\fk\times\Delta^{\fm\sharp\fn}$. Thus, we obtain the diagram of spaces above and it induces the upper part of the diagram of morphism sets.

Commutativity holds trivially except at one pentagon: The two compositions $X_{\vartheta_1,\vartheta_2}\to|\Delta^{\fl\sharp\fm}|\times|\Delta^\fn|$ are given on the vertices as
\begin{align*}
(\mu_{\fl\sharp\fm,\fn}\times\id)\circ(\alpha\times\nu_{\fm,\fn})\circ P_{\vartheta_1,\vartheta_2}(l,\zeta,\eta,\kappa,m,\xi,n)&=(l,\zeta,\xi(m))
\\(P_{\vartheta_{\fl,\fm}}\times\id)\circ(\mu_{\fl\sharp\fm,\fn}\circ\alpha\times\id\times\id)&=(l,\zeta,n)
\end{align*}
and since the first two components of these triples agree, it is straightforward to see that their linear interpolation induces a $\Ca(X_{\vartheta_1,\vartheta_2},\blank)$-homotopy 
\[\Simpl^{\fl\sharp\fm}\Simpl^\fn \to\Ca(X_{\vartheta_1,\vartheta_2},\Ct([0,1],\blank))\,.\]
The claim follows.
\end{proof}

Let $\fo$ denote the one-element directed set, see \Cref{example:oneelementdirectedset}
The directed set $\fm\sharp\fo$ is canonically isomorphic to $\fm$, $X_{\vartheta_{\fm,\fo}}=|\Delta^\fm|\times|\Delta^\fo|\approx|\Delta^\fm|$ and the maps $P_{\vartheta_{\fm,\fo}}$ and $|\mu_{\fm,\fo}|$ simply correspond to the identity map on $|\Delta^\fm|$. Therefore, the natural transformation $\Theta_{\fm,\fn}\colon\Simpl^\fm\Simpl^\fo\to\Simpl^{\fm\sharp\fo}$ is a natural isomorphism corresponding to the identity on $\Simpl^\fm$ and hence the product $\SimplMor[\fm]{A}{B}\times \SimplMor[\fo]{B}{C}\to \SimplMor[\fm\sharp\fo]{A}{C}$ corresponds to the product $\SimplMor[\Simpl^\fm]{A}{B}\times \SimplMor[\Simpl^\fo]{B}{C}\to \SimplMor[\Simpl^\fm\Simpl^\fo]{A}{C}$ introduced in \Cref{cor:Fhomotopyclasscompositionproduct}.

The directed set $\fo\sharp\fm$ is also isomorphic to $\fm$ and in this case $X_{\vartheta_{\fo,\fm}}$ corresponds to a subset of 
\[\bigcup_{\substack{\sigma,\tau\in\Delta^\fm\\\forall m\in\sigma,m'\in\tau\colon m'\geq m}}|\Delta^\sigma|\times|\Delta^{\tau}|\subset|\Delta^\fm|\times|\Delta^\fm|\,.\] 
Using this, one readily checks that 
\[|\Delta^\fm|\approx|\Delta^{\fm\sharp\fo}|\xrightarrow{\Theta_{\fo,\fm}}|\Delta^\fo|\times|\Delta^\fm|\approx|\Delta^\fm|\]
is homotopic to the identity via a homotopy of continuous cofinal maps and hence we see that the product $\SimplMor[\fo]{A}{B}\times \SimplMor[\fm]{B}{C}\to \SimplMor[\fo\sharp\fm]{A}{C}$ also identifies with the product $\SimplMor[\Simpl^\fo]{A}{B}\times \SimplMor[\Simpl^\fm]{B}{C}\to \SimplMor[\Simpl^\fo\Simpl^\fm]{A}{C}$.

This shows that the functoriality of the sets $\SimplMor{A}{B}$ in $A$ and $B$, which comes from the one in \Cref{cor:Fhomotopyclasscompositionproduct}, is given by composition product with the homotopy classes of the $*$-ho\-mo\-mor\-phism. In particular we see that the composition product with $\llbracket \id_A\rrbracket\in\SimplMor{A}{A}$ is an identity morphism over $A$.

Since furthermore $\SimplMor[\fo]{A}{B}\times \SimplMor[\fo]{B}{C}\to \SimplMor[\fo\sharp\fo]{A}{C}\cong\SimplMor[\fo]{A}{C}$ identifies with the composition product of homotopy classes of $*$-ho\-mo\-mor\-phisms, we obtain the following. 

\begin{thm}\label{thm:defSimpltoticCategory}
The simpltotic morphism sets $\SimplMor{A}{B}$ are the morphism sets of a category whose objects are the \textCstar-algebras.
We call this category the \emph{simpltotic category}.

There is a canonical functor from the homotopy category of \textCstar-algebras into the simpltotic category which assigns to the homotopy class of a $*$-ho\-mo\-mor\-phism $\varphi\colon A\to B$ the morphism $\llbracket\varphi\rrbracket\coloneqq\llbracket A\xrightarrow{\varphi}B\cong\Simpl^\fo B\rrbracket \in\SimplMor[\fo]{A}{B}$. \qed
\end{thm}

Finally, we also want to show that there is a functor from the asymptotic category into the simpletotic category which is an isomorphism for separable \textCstar-algebras and extends the functor from \Cref{thm:defSimpltoticCategory}.
To this end, we first observe the following immediate consequence of \Cref{lem:Thetaassociativity,lem:FFprimehomotopic}:
Given directed sets $\fm_1,\dots,\fm_i$ and $1\leq j\leq i-2$, there is a directed set $\fk_j$ and cofinal maps 
\[(\fm_j\sharp\fm_{j+1})\sharp\fm_{j+2}\xleftarrow{\alpha_j}\fk_j\xrightarrow{\beta_j}\fm_j\sharp(\fm_{j+1}\sharp\fm_{j+2})\]
such that the diagram of natural transformations
\[\xymatrix{
\Simpl^{\fm_1}\cdots\Simpl^{\fm_j\sharp\fm_{j+1}}\cdots\Simpl^{\fm_i}\ar[r]
&\Simpl^{\fm_1}\cdots\Simpl^{(\fm_j\sharp\fm_{j+1})\sharp\fm_{j+2}}\cdots\Simpl^{\fm_i}\ar[d]
\\\Simpl^{\fm_1}\cdots\Simpl^{\fm_i}\ar[u]\ar[d]
&\Simpl^{\fm_1}\cdots\Simpl^{\fk_j}\cdots\Simpl^{\fm_i}
\\\Simpl^{\fm_1}\cdots\Simpl^{\fm_{j+1}\sharp\fm_{j+2}}\cdots\Simpl^{\fm_i}\ar[r]
&\Simpl^{\fm_1}\cdots\Simpl^{\fm_j\sharp(\fm_{j+1}\sharp\fm_{j+2})}\cdots\Simpl^{\fm_i}\ar[u]
}\]
induced by the various $\Theta$s and $\alpha_j,\beta_j$
commutes up to $\Simpl^{\fm_1}\cdots\Simpl^{\fk_j}\cdots\Simpl^{\fm_i}$-homotopy.
Thus, if we define a natural transformation $\Simpl^{\fm_1}\cdots\Simpl^{\fm_i}\to\Simpl^\fk$ by iteratively applying a $\Theta^*$ to two neighboring functors, then its $\Simpl^\fk$-homotopy class is independent of the order in which we contract neighboring functors, provided $\fk$ is large enough.

Combining this observation with the natural transformations $\tau^*\colon\Asymp\to\Simpl^{\N_+}$ from \Cref{ex:asymptoticsimpltoticrelation} we obtain for each $n\in\N$ canonical maps
\[F_n\colon\SimplMor[n]{A}{B}=\SimplMor[\Asymp^n]{A}{B}\to\SimplMor[(\Simpl^{\N_+})^n]{A}{B}\to\SimplMor[\fk]{A}{B}\to\SimplMor{A}{B}\,.\]
\begin{thm}\label{thm:asymptoticsimpltoticfunctor}
The maps described above constitute a functor from the asymptotic category to the simpltotic category.
\end{thm}

Note that \Cref{ex:asymptoticsimpltoticrelation} shows that this functor is an isomorphism on morphism sets between separable \textCstar-algebras, if the group $G$ is countable.

\begin{proof}
Since we can choose a natural transformation $(\Simpl^{\N_+})^{n_1+n_2}\to\Simpl^\fk$ which first contracts the first $n_1$ factors to $\Simpl^{\fk_1}$ and the last $n_2$ factors to $\Simpl^{\fk_2}$ before combining them via $\Theta_{\fk_1,\fk_2}^*$, we obtain a commutative diagram
\[\xymatrix{
\SimplMor[(\Simpl^{\N_+})^{n_1}]{A}{B}\times\SimplMor[(\Simpl^{\N_+})^{n_2}]{B}{C}\ar[r]\ar[d]^{F_{n_1}\times F_{n_2}}
&\SimplMor[(\Simpl^{\N_+})^{n_1+n_2}]{A}{C}\ar[dr]^-{F_{n_1+n_2}}\ar[d]
&
\\\SimplMor[\fk_1]{A}{B}\times\SimplMor[\fk_2]{B}{C}\ar[r]
&\SimplMor[\Simpl^{\fk_1}\Simpl^{\fk_2}]{A}{C}\ar[r]_-{\Theta_{\fk_1,\fk_2}^*}
&\SimplMor[\fk]{A}{C}
}\]
and we conclude that
\begin{equation}
\label{eq:comparingasymptoticwithsimpltoticcomposition}
F_{n_1+n_2}(\llbracket\psi\rrbracket\circ\llbracket\varphi\rrbracket)=F_{n_1}(\llbracket\psi\rrbracket)\circ F_{n_2}(\llbracket\varphi\rrbracket)
\end{equation}
for all $\llbracket\varphi\rrbracket\in\SimplMor[n_1]{A}{B}$ and $\llbracket\psi\rrbracket\in\SimplMor[n_2]{B}{C}$.

Note that the maps $\SimplMor[n]{A}{B}\to \SimplMor[n+1]{A}{B}$ in the directed system of asymptotic morphism sets are actually given by taking the product with 
\[\llbracket \psi\colon B\to\Asymp B,b\mapsto [t\mapsto b]\rrbracket\in \SimplMor[1]{B}{B}\]
and $F_1\llbracket \psi\rrbracket\in\SimplMor{B}{B}$ is the identity morphism. Therefore, \eqref{eq:comparingasymptoticwithsimpltoticcomposition} implies that the $F_n$ constitute a cocone and hence we obtain canonical maps
\[F\colon\SimplMor[\infty]{A}{B}\coloneqq\varinjlim_{n\to\infty}\SimplMor[n]{A}{B}\to\SimplMor{A}{B}\]
for all $A,B$.

Then \eqref{eq:comparingasymptoticwithsimpltoticcomposition} also passes to the direct limit over $n_1,n_2$.
Since moreover $F_0$ clearly maps identity morphisms to identity morphisms, we see that the $F$ indeed constitute a functor between the asymptotic and the simpltotic category. 
\end{proof}

\section{Functors on the simpltotic category}
\label{sec:functors}

Chapter 3 of [GHT] shows that functors from the category of $G$-\textCstar-algebras to the category of $H$-\textCstar-algebras determine functors from the asymptotic category of $G$-\textCstar-algebras to the asymptotic category of $H$-\textCstar-algebras if they are exact and continuous. 
Here, continuity is not meant in the sense of category theory, i.\,e.\ that it preserves small limits, and hence we will avoid this denomination. Instead we will call it $\cI$-continuity, which happens to be a special case of \Cref{defn:functorproperties}.\ref{defn:it:functorcont} for the class $\cI$ of all intervals.

Unfortunately, these two properties are not quite sufficient to obtain the analogous result for the simpltotic category. We will therefore additionally assume that the functors are also cocontinuous with respect to a certain type of direct limits, this time in the sense of category theory.\footnote{An alternative to assuming cocontinuity would probably be to replace the assumption of $\cI$-continuity with the one of $\cZ$-continuity for a bigger class of spaces $\cC$, but this condition appears to be less managable in practice.}

Without further ado, let us define and recall these notions.
\begin{defn}\label{defn:functorproperties}
Let $G$, $H$ be two discrete groups and let $\tilde F$ be a functor from the category of $G$-\textCstar-algebras  to the category of $H$-\textCstar-algebras.
\begin{enumerate}
\item\label{defn:it:functorcont} Let $\cZ$ be a class of spaces. Then $\tilde F$ is called \emph{$\cZ$-continuous} if for every $G$-\textCstar-algebra $B$, every $Z\in\cZ$ and every $f\in \tilde F(\Cb(Z;B))$ the function 
\[\mu^Z_B(f)\colon Z\to \tilde F(B)\,,\quad z\mapsto \tilde F(\ev{z})(f)\]
is continuous.
\item\label{defn:it:functorexact} $\tilde F$ is called \emph{exact} if for every short exact sequence $0\to I\to A\to B\to 0$ of $G$-\textCstar-algebras the induced sequence of $H$-\textCstar-algebras $0\to \tilde F(I)\to \tilde F(A)\to \tilde F(B)\to 0$ is exact.
\item\label{defn:it:directedcoconut} $\tilde F$ is called \emph{cocontinuous under direct limits of ideals} if for every directed system $\{A_m\}_{m\in\fm}$ of ideals in a $G$-\textCstar-algebra $A$ ($\fm$ a directed set, $A_m$ ideal in $A$, $A_m\subset A_n$ if $m\leq n$) the canonical map
\[\varinjlim_{m\in\fm}\tilde F(A_m)\to \tilde F\left(\varinjlim_{m\in\fm}A_m\right)\]
is an isomorphism.
\end{enumerate}
\end{defn}

Note that the function $\hat f$ in point 1 is always bounded by $\|f\|$, because each $\tilde F(\ev{z})$ is a $*$-ho\-mo\-mor\-phism and hence has operator norm at most 1. Thus, if \ref{defn:it:functorcont} holds we have $*$-ho\-mo\-mor\-phisms $\mu_B^Z\colon \tilde F(\Cb(Z;B))\to \Cb(Z;\tilde F(B))$ and they are clearly natural in $B$.

\begin{lem}\label{lem:CWcontinuity}
If $\tilde F$ is $\cI$-continuous for the class $\cI$ of all intervals, then it is also $\cCW$-continuous for the class $\cCW$ of all CW-complexes.
\end{lem}

\begin{proof}
Iterated application of $\cI$-continuity shows that $\tilde F$ is also $\cC$-continuous for the class $\cC$ of all cubes, hence also $\cD$-continuous for the class $\cD\coloneqq\{D^n\mid n\in\N\}$ of all disks.
Since a function on a CW-complex is continuous if and only if its restriction to all closed cells are continuous, the claim follows.
\end{proof}

\begin{lem}
Let $\tilde F$ be an $\cI$-continuous and exact functor that is cocontinuous under direct limits of ideals and let $X$ be a CW-complex equipped with a filter as in \Cref{prop:homotopyequivalencerelation}.
Then for every $f\in \tilde F(\Cz(X;B))$, which can be identified with an ideal in $\tilde F(\Cb(X;B))$ by exactness of $\tilde F$, we have $\mu^Z_B(f)\in\Cz(X;\tilde F(B))$. Consequently, $\mu^Z$ induces natural transformations $\tilde F(\Cz(X;\blank))\to \Cz(X;\tilde F(\blank))$ and $\tilde F(\Ca(X;\blank))\to \Ca(X;\tilde F(\blank))$ which we denote by the same letter.
\end{lem}

\begin{proof}
Let $\filterbasis$ be a basis of the filter on $X$ consisting of subcomplexes of $X$.
For each $U\in\filterbasis$ let $\Cb(X,U;B)$ be the ideal in $\Cb(X;B)$ consisting of the functions vanishing on $U$.
Then $\Cz(X;B)=\varinjlim_{U\in\filterbasis}\Cb(X,U;B)$ and hence $\tilde F(\Cz(X;B))=\varinjlim_{U\in\filterbasis}\tilde F(\Cb(X,U;B))$ by cocontinuity. It therefore suffices to show that $\mu^Z_B$ maps each $\tilde F(\Cb(X,U;B))$ into $\Cb(X,U;\tilde F(B))$.
To this end, consider the commutative diagram
\[\xymatrix{
0\ar[r]&\tilde F(\Cb(X,U;B))\ar[r]\ar@{-->}[d]^{\mu^{X,U}_B}&\tilde F(\Cb(X;B))\ar[r]\ar[d]^{\mu^X_B}&\tilde F(\Cb(U;B))\ar[r]\ar[d]^{\mu^U_B}&0
\\0\ar[r]&\Cb(X,U;\tilde F(B))\ar[r]&\Cb(X;\tilde F(B))\ar[r]&\Cb(U;\tilde F(B))\ar[r]&0
}\]
whose rows are exact and whose left vertical arrow is induced by the the right two ones. In other words, the middle verticle map indeed restricts to a $*$-ho\-mo\-mor\-phism between the ideals.

For the second part, the commutative diagram with exact rows
\[\xymatrix{
0\ar[r]&\tilde F(\Cz(X;B))\ar[r]\ar[d]&\tilde F(\Cb(X;B))\ar[r]\ar[d]&\tilde F(\Ca(X;B))\ar[r]\ar@{-->}[d]&0
\\0\ar[r]&\Cz(X;\tilde F(B))\ar[r]&\Cb(X;\tilde F(B))\ar[r]&\Ca(X;\tilde F(B))\ar[r]&0
}\]
induces the right vertical arrow and its naturality in $B$ is clear.
\end{proof}

\begin{cor}\label{cor:functorordernaturaltransformation}
Let $\tilde F$ be an $\cI$-continuous and exact functor that is cocontinuous under direct limits of ideals and let $F=F_1\circ\dots\circ F_k$ be a functor as in \Cref{prop:homotopyequivalencerelation}. 
Then there is a canonical natural transformation $\eta^F\colon \tilde F\circ F\to F\circ \tilde F$.
\end{cor}
\begin{proof}
Induction over $k$.
\end{proof}

\begin{lem}
Let $\tilde F$ be an $\cI$-continuous and exact functor that is cocontinuous under direct limits of ideals and let $F$ be a functor as in \Cref{prop:homotopyequivalencerelation}. 
Then the canonical map 
\begin{align*}
\tilde F_F\colon\SimplMor[F]{A}{B}&\to\SimplMor[F]{\tilde F(A)}{\tilde F(B)}
\\\llbracket \varphi\colon A\to F(B)\rrbracket&\mapsto\llbracket\eta^F_B\circ \tilde F(\varphi)\colon \tilde F(A)\to F(\tilde F(B))\rrbracket
\end{align*}
is well-defined.
\end{lem}
\begin{proof}
If $\Phi\colon A\to F(\Ct([0,1];B))$ is an $F$-homotopy between $\varphi_0$ and $\varphi_1$, then the composition
\begin{align*}
\tilde F(A)&\xrightarrow{\tilde F(\Phi)}\tilde F(F(\Ct([0,1];B)))\xrightarrow{\eta^F_{\Ct([0,1];B)}}F(\tilde F(\Ct([0,1];B)))
\\&\xrightarrow{F(\mu^{[0,1]}_B)}F(\Ct([0,1];\tilde F(B)))
\end{align*}
is an $F$-homotopy between $\eta^F_B\circ \tilde F(\varphi_0)$ and $\eta^F_B\circ \tilde F(\varphi_1)$.
\end{proof}

\begin{lem}
Let $\tilde F$ be an $\cI$-continuous and exact functor that is cocontinuous under direct limits of ideals and let $F,F'$ be functor as in \Cref{prop:homotopyequivalencerelation}. 
Then the maps $\tilde F_F$, $\tilde F_{F'}$ and $\tilde F_{F\circ F'}$ are compatible with the products defined in \Cref{cor:Fhomotopyclasscompositionproduct}, that is, for $\llbracket\varphi\rrbracket\in\SimplMor[F]{A}{B}$ and $\llbracket\psi\rrbracket\in\SimplMor[F']{B}{C}$ we have
\[\tilde F_{F\circ F'}(\llbracket\psi\rrbracket\circ\llbracket\varphi\rrbracket) = \tilde F_{F'}\llbracket\psi\rrbracket\circ \tilde F_F\llbracket\varphi\rrbracket\in\SimplMor[F\circ F']{\tilde F(A)}{\tilde F(C)}\,.\]
\end{lem}
\begin{proof}
The claim follows directly from the commutative diagram
\[\xymatrix@C=10ex{
\tilde F(A)\ar[r]^{\tilde F(\varphi)}\ar[dr]_{\tilde F(F(\psi)\circ\varphi)}
&\tilde F(F(B))\ar[r]^{\eta^F_B}\ar[d]^{\tilde F(F(\psi))}
&F(\tilde F(B))\ar[d]_{F(\tilde F(\psi))}\ar@/^10ex/[dd]^{F(\eta^{F'}_C\circ \tilde F(\psi))}
\\
&\tilde F(F\circ F'(C))\ar[r]^{\eta^F_{F'(C)}}\ar[dr]_{\eta^{F\circ F'}_C}
&F(\tilde F(F'(C)))\ar[d]_{F(\eta^{F'}_C)}
\\&&F\circ F'(\tilde F(C))
}\]
which is readily verified.
\end{proof}

\begin{cor}
Let $\tilde F$ be an $\cI$-continuous and exact functor that is cocontinuous under direct limits of ideals and let $\fm,\fn$ be directed sets. Then the maps $\tilde F_{\Simpl^\fm}$, $\tilde F_{\Simpl^\fn}$ and $\tilde F_{\Simpl^{\fm\sharp\fn}}$ are compatible with the composition products constructed in \Cref{sec:productconstruction}, that is, 
\[\tilde F_{\Simpl^{\fm\sharp\fn}}(\llbracket\psi\rrbracket\circ\llbracket\varphi\rrbracket) = \tilde F_{\Simpl^\fn}\llbracket\psi\rrbracket\circ \tilde F_{\Simpl^\fm}\llbracket\varphi\rrbracket\in\SimplMor[\Simpl^{\fm\sharp\fn}]{\tilde F(A)}{\tilde F(C)}\,.\]
\end{cor}
\begin{proof}
In light of the lemma, we have to prove that the diagram
\[\xymatrix{
\SimplMor[\Simpl^\fm\Simpl^\fn]{A}{C}\ar[r]^{\Theta^*}\ar[d]^{\tilde F_{\Simpl^\fm\Simpl^\fn}}
&\SimplMor[\fm\sharp\fn]{A}{C}\ar[d]^{\tilde F_{\Simpl^{\fm\sharp\fn}}}
\\\SimplMor[\Simpl^\fm\Simpl^\fn]{\tilde F A}{\tilde FC}\ar[r]^{\Theta^*}
&\SimplMor[\fm\sharp\fn]{\tilde F A}{\tilde FC}
}\]
commutes. This follows from commutativity of the diagram
\[\xymatrix@C=8ex{
\tilde F\Simpl^\fm\Simpl^\fn C\ar[r]^{\tilde F(\Theta^*)}\ar[d]^{\eta^{\Simpl^\fm\Simpl^\fn}_C}
&\tilde F\Simpl^{\fm\sharp\fn}C\ar[d]^{\eta^{\Simpl^{\fm\sharp\fn}}_C}
\\\Simpl^\fm\Simpl^\fn \tilde FC\ar[r]^{\Theta^*}
&\Simpl^{\fm\sharp\fn}\tilde FC
}\]
which in turn can be verified in a straightforward manner on a representative in $F(\Cb(|\Delta^\fm|;\Cb(|\Delta^\fn|;C))$ of an element in the top left corner.
\end{proof}

\begin{thm}\label{thm:functorsonsimpletoticcategory}
Let $\tilde F$ be an $\cI$-continuous and exact functor that is cocontinuous under direct limits of ideals.
The maps $\tilde F_{\Simpl^\fm}$ constitute natural transformations of functors over $\cD^\op$ and hence induce maps on the direct limits.
The resulting maps $\tilde F_{\Simpl}\colon\SimplMor{A}{B}\to \SimplMor{\tilde FA}{\tilde FB}$ together with the map $A\mapsto \tilde F_\Simpl(A)\coloneqq \tilde FA$ on the class of objects are a functor from the simpltotic category of $G$-\textCstar-algebras to the simpltotic category of $H$-\textCstar-algebras such that $\tilde F_\Simpl\llbracket f\rrbracket=\llbracket \tilde Ff\rrbracket$ for all $*$-ho\-mo\-mor\-phisms $f$.
\end{thm}
\begin{proof}
Similar to the proof of the preceding corollary, it is straighforward to verify for any cofinal map $\alpha\colon\fm\to\fn$ that the diagram
\[\xymatrix@C=8ex{
\tilde F\Simpl^\fn C\ar[r]^{\tilde F(\alpha^*)}\ar[d]^{\eta^{\Simpl^\fn}_C}
&\tilde F\Simpl^{\fm}C\ar[d]^{\eta^{\Simpl^{\fm}}_C}
\\\Simpl^\fn \tilde FC\ar[r]^{\alpha^*}
&\Simpl^{\fm}\tilde FC
}\]
commutes and the first claim follows from this.

Compatibility of $\tilde F_\Simpl$ with the composition products in the simpletotic category are now a direct consequence of the preceding corollary.

Clearly, $\tilde F_{\Simpl^\fo}$ maps the class of a $*$-ho\-mo\-mor\-phism $f\colon A\to B=\Simpl^\fo B$ to the class of $\tilde Ff\colon\tilde FA\to\tilde FB=\Simpl^\fo\tilde FB$ and hence we get $\tilde F_\Simpl\llbracket f\rrbracket=\llbracket \tilde Ff\rrbracket$.
In particular this also shows that $\tilde F_\Sigma$ maps the identity morphism over $A$ to the identity morphism over $\tilde F_\Simpl A$.
\end{proof}

In the remainder of this section we discuss two important examples of $\cI$-continuous and exact functor that are cocontinuous under direct limits of ideals, namely the maximal tensor product and the full crossed product, cf.\ \cite[Chapter 4]{GueHigTro}. In the following we will always use the symbol $\grtensor$ for the \emph{maximal} graded tensor product of \textCstar-algebras. If we are dealing with ungraded \textCstar-algebras, then we will also use the symbol $\otimes$.

\begin{lem}\label{lem:propertiesoftensorproductfunctors}
For any \textCstar-algebra $D$, the two maximal tensor product functors $B\mapsto B\grtensor D$ and $B\mapsto D\grtensor B$ are $\cI$-continuous, exact and cocontinuous.
\end{lem}

\begin{proof}[References for proof]
For exactness see \cite[Lemma 4.1]{GueHigTro} and \cite[Section 1.9]{Wassermann} and for $\cI$-continuity see \cite[Lemma 4.2]{GueHigTro}. Cocontinuity follows from universal properties, as was mentioned in \cite[II.9.6.5]{BlackadarOA}.
\end{proof}

Note that the lemma stays true if we take group actions on the \textCstar-algebras into account, because the three properties of the functor make no use of it. Therefore, \Cref{thm:functorsonsimpletoticcategory,lem:propertiesoftensorproductfunctors} yield the following.

\begin{cor}\label{cor:simpltotictensorproductfunctors}
For an $H$-\textCstar-algebra $D$ there is a functor from the simpletotic category of $G$-\textCstar-algebras to the simpletotic category of $(G\times H)$-\textCstar-algebras which associates to the class of a $*$-ho\-mo\-mor\-phism $\varphi \colon A\to \Simpl^\fm B$ the class $\llbracket\varphi\rrbracket\grtensor\id$ represented by the composition 
\[A\grtensor D\xrightarrow{\varphi\grtensor\id}(\Simpl^\fm B)\grtensor D\xrightarrow{\eta^{\Simpl^\fm}_B}\Simpl^\fm (B\grtensor D)\,.\]
Similarily, for a $G$-\textCstar-algebra $D$ there is a functor from the simpletotic category of $H$-\textCstar-algebras to the simpletotic category of $(G\times H)$-\textCstar-algebras which associates to the class of a $*$-ho\-mo\-mor\-phism $\psi\colon A\to \Simpl^\fn B$ the class $\id\grtensor\llbracket\psi\rrbracket$ represented by the composition 
\[D\grtensor A\xrightarrow{\id\grtensor\psi}D\grtensor \Simpl^\fn B\xrightarrow{\eta^{\Simpl^\fn}_B}\Simpl^\fn (D\grtensor B)\,.\]
\qed
\end{cor}

\begin{lem}\label{lem:leftrightmaxtensorcommute}
Given $\llbracket\varphi\rrbracket\in\SimplMor{A_1}{A_2}$ and $\llbracket\psi\rrbracket\in\SimplMor{B_1}{B_2}$, the diagram
\[\xymatrix@C=8ex{
A_1\grtensor B_1\ar[r]^{\llbracket\varphi\rrbracket\grtensor\id}\ar[d]_{\id\grtensor\llbracket\psi\rrbracket}
&A_2\grtensor B_1\ar[d]^{\id\grtensor\llbracket\psi\rrbracket}
\\A_1\grtensor B_2\ar[r]^{\llbracket\varphi\rrbracket\grtensor\id}
&A_2\grtensor B_2
}\]
commutes in the simpltotic category.
\end{lem}

\begin{proof}
Let $\varphi\colon A_1\to\Simpl^\fm A_2$ and $\psi\colon B_1\to\Simpl^\fn B_2$ be two representatives of the given classes.
Denoting the two projections from $\fm\times\fn$ onto the factors $\fm,\fn$ by $\pi_1,\pi_2$, respectively, the $*$-ho\-mo\-mor\-phism
\begin{align*}
\Cb(|\Delta^\fm|;A_2)\grtensor\Cb(|\Delta^\fn|;B_2)&\subset \Cb(|\Delta^\fm|\times|\Delta^\fn|;A_2\grtensor B_2)
\\&\xrightarrow{(|\pi_1|\times|\pi_2|)^*} \Cb(|\Delta^{\fm\times\fn}|;A_2\grtensor B_2)
\end{align*}
clearly descends to a $*$-ho\-mo\-mor\-phism $(|\pi_1|\times|\pi_2|)^*\colon\Simpl^\fm A_2\grtensor\Simpl^\fn B_2\to\Simpl^{\fm\times\fn}(A_2\grtensor B_2)$.
We claim that the two compositions in the diagram are both represented by 
\[(|\pi_1|\times|\pi_2|)^*\circ(\varphi\grtensor\psi)\colon A_1\grtensor B_1\to\Simpl^{\fm\times\fn}(A_2\grtensor B_2)\,.\]

To prove this, recall from \Cref{sec:productconstruction,sec:productproperties} the space
\[X_\vartheta\coloneqq\bigcup_{\sigma\in \Delta^{\fm\sharp\fn}}|\Delta^\sigma|\times|\Delta^{\fn\ab\vartheta(\sigma)}|\subset |\Delta^{\fm\sharp\fn}|\times|\Delta^\fn|\]
together with the projection $P_\vartheta\colon X_\vartheta\to |\Delta^{\fm\sharp\fn}|$ onto the first component and the cofinal maps $\mu\colon\fm\sharp\fn\to\fm, (m,\xi)\mapsto m$ and $\nu\colon\fm\sharp\fn\to\fn, (m,\xi)\mapsto\xi(m)$. The diagram
\[\xymatrix@C=10ex{
X_\vartheta\ar[r]^{|\mu|\times\id}\ar[d]_{P_\vartheta}
&|\Delta^\fm|\times|\Delta^\fn|
\\|\Delta^{\fm\sharp\fn}|\ar[r]^{|\mu\times\nu|}
&|\Delta^{\fm\times\fn}|\ar[u]_{|\pi_1|\times|\pi_2|}
}\]
commutes only in the first component. Using the canonical linear homotopy on the second component, it is straightforward to prove that the induced diagram
\[\xymatrix@C=12ex{
\Ca(X_\vartheta;A_2\grtensor B_2)
&\Simpl^\fm\Simpl^\fn(A_2\grtensor B_2)\ar[l]_{(|\mu|\times\id)^*}
&\Simpl^\fm A_2\grtensor\Simpl^\fn B_2\ar[l]_{\Simpl^\fm(\eta^{\Simpl^\fn}_{B_2})\circ\eta^{\Simpl^\fm}_{A_2}}\ar[d]_{(|\pi_1|\times|\pi_2|)^*}
\\\Simpl^{\fm\sharp\fn}\ar[u]^{P_\vartheta^*}(A_2\grtensor B_2)
&&\Simpl^{\fm\times\fn}(A_2\grtensor B_2)\ar[ll]_{|\mu\times\nu|^*}
}\]
commutes up to $\Ca(X_\vartheta;\blank)$-homotopy. This readily implies that the composition 
\[(\id\grtensor\llbracket\psi\rrbracket)\circ(\llbracket\varphi\rrbracket\grtensor\id)\in\SimplMor[\fm\sharp\fn]{A_1\grtensor B_1}{A_2\grtensor B_2}\] is indeed obtained from 
\[\llbracket (|\pi_1|\times|\pi_2|)^*\circ(\varphi\grtensor\psi)\rrbracket\in \SimplMor[\fm\times\fn]{A_1\grtensor B_1}{A_2\grtensor B_2}\]
by pullback under the cofinal map $\mu\times\nu$.

The claim for the other composition follows completely analogously with the roles of $\fm$ and $\fn$ interchanged.
\end{proof}

Denoting the simpltotic category of $G$-\textCstar-algebras by $G\mbox{-}\Simpl$, \Cref{cor:simpltotictensorproductfunctors} and \Cref{lem:leftrightmaxtensorcommute} immediately yield the following adaption of \cite[Theorem 4.6]{GueHigTro}.

\begin{thm}
There is a maximal tensor product functor 
\[\grtensor\colon (G\mbox{-}\Simpl)\times (H\mbox{-}\Simpl)\to (G\times H)\mbox{-}\Simpl\]
which is the maximal tensor product on objects and which associates to a pair of morphisms $(\id,\llbracket\psi\rrbracket)$ or $(\llbracket\varphi\rrbracket,\id)$ the morphism $\id\grtensor\llbracket\psi\rrbracket$ or $\llbracket\varphi\rrbracket\grtensor\id$, respectively, that were introduced in \Cref{cor:simpltotictensorproductfunctors}. \qed
\end{thm}

The second functor of interest is, of course, the full crossed product functor $\Cstar(G;\blank)$ from the category of $G$-\textCstar-algebras to the category of \textCstar-algebras, see, for example, \cite[p.~25]{GueHigTro}

\begin{lem}
The full crossed product functor $\cI$-continuous, exact and cocontinuous under direct limits of ideals.\footnote{Presumably, cocontinuity holds in larger generality, but the author isn't aware of a reference. For simplicity, we restricted ourselves to an easy proof in this special case, which is sufficient for our purposes.}
\end{lem}

\begin{proof}
Exactness and $\cI$-continuity were proven in \cite[Lemmas 4.10 \& 4.11]{GueHigTro}.

To prove cocontinuity under direct limits of ideals, let $(B_m)_{m\in \fm}$ be a directed system of ideals in a \textCstar-algebra $B$.
Then each $B_m$ is an ideal in $\varinjlim_{m\in\fm}B_m=\overline{\bigcup_{m\in\fm}B_m}\subset B$ and hence exactness of the functor implies that each $\Cstar(G,B_m)$ is an ideal in $\Cstar(G,\varinjlim_{m\in\fm}B_m)$.
Since $\bigcup_{m\in\fm}\Cc(G;B_m)$ is dense in $\Cc(G;\varinjlim_{m\in\fm}B_m)$ with respect to the $L^1$-norm and hence also dense in the completion $\Cstar(G,\varinjlim_{m\in\fm}B_m)$ with respect to the full crossed product norm, we see that
\[\varinjlim_{m\in\fm} \Cstar(G,B_m)=\overline{\bigcup_{m\in\fm}\Cstar(G,B_m)}=\Cstar(G,\varinjlim_{m\in\fm}B_m)\,.\qedhere\]
\end{proof}

As a result we obtain the descent functors on the simpltotic category.

\begin{thm}[{cf.~\cite[Theorem 4.12]{GueHigTro}}]
There is a \emph{descent functor} 
\[\Cstar(G,\blank)\colon G\mbox{-}\Simpl\to\C\mbox{-}\Simpl\]
which associates to the class of a $G$-equivariant $*$-ho\-mo\-mor\-phism $\varphi\colon A\to\Simpl^\fm B$ the class of the composition
\[\Cstar(G,A)\xrightarrow{\Cstar(G,\varphi)}\Cstar(G,\Simpl^\fm B)\xrightarrow{\eta^{\Simpl^\fm}_B}\Simpl^\fm\Cstar(G,B)\,.\]
\qed
\end{thm}

\section{Extensions}
\label{sec:extensions}

In this section we transfer the results of \cite[Chapter 5]{GueHigTro} to the simpltotic case. Central to them is the construction of a class $\llbracket\sigma\rrbracket\in\SimplMor{\Sigma A}{J}$,
where we have used the shorthand notation
\[\Sigma A\coloneqq \Cz((0,1))\grtensor A\cong\Cz((0,1);A)\,,\]
associated to each short exact sequence
\[0\to J\to B\xrightarrow{\pi} A\to 0\]
of \textCstar-algebras, which we will obtain from a quasicentral approximate unit.

\begin{defn}[{cf.\ \cite[Definition 5.2]{GueHigTro}}]
Let $B$ be a $G$-\textCstar-algebra and $J\subset B$ a sub-$G$-\textCstar-algebra such that $[B,J]\subset J$.
A \emph{quasicentral approximate unit} for $J$ in $B$ is a net $\{u_m\}_{m\in\fm}$ in $J$ such that
\begin{enumerate}
\item $\forall m\in\fm\colon 0\leq u_m\leq 1$;
\item $\forall j\in J\colon\lim_{m\in\fm}\|u_mj-j\|=0$;
\item $\forall b\in B\colon\lim_{m\in\fm}\|[u_m,b]\|=0$;
\item $\forall g\in G\colon\lim_{m\in\fm}\|gu_m-u_m]\|=0$.
\end{enumerate}
An \emph{approximate unit} for a $G$-\textCstar-algebra $J$ is a quasicentral approximate unit of $J$ in itself, i.\,e.\ condition 3 is superfluous.
\end{defn}

Recall that we only consider discrete groups $G$ and hence we do not need to impose uniformity on compact subsets of the fourth convergence.

\begin{lem}
Quasicentral approximate units exist for all $G$-\textCstar-algebras $B$ with sub-$G$-\textCstar-algebra $J$ satisfying the condition $[B,J]\subset J$.
\end{lem}

\begin{proof}
In the non-equivariant case ($G=1$) this was proven in \cite[Theorem 1]{Arveson} for ideals $J\subset B$, but the proof works exactly the same for the more general condition. 
The general statement can then be derived with the trick from \cite[Lemma 5.3]{GueHigTro}, which works equally well for non-separable $G$-\textCstar-algebras.
\end{proof}

The following properties of quasicentral approximate unit can be proven completely analogously to \cite[Lemma 5.6]{GueHigTro}.

\begin{lem}
Let $B$ be a $G$-\textCstar-algebra, $J\subset B$ a sub-$G$-\textCstar-algebra such that $[B,J]\subset J$ and $\{u_m\}_{m\in\fm}$ a quasicentral approximate unit for $J$ in $B$. Then for any $f\in\Cz((0,1])$:
\begin{enumerate}
\item $\forall g\in G\colon\lim_{m\in\fm}\|gf(u_m)-f(u_m)]\|=0$;
\item $\forall b\in B\colon\lim_{m\in\fm}\|[f(u_m),b]\|=0$;
\item if in addition $f(1)=0$, then $\forall j\in J\colon\lim_{m\in\fm}\|f(u_m)j\|=0$.\qed
\end{enumerate}
\end{lem}

We can now introduce the simpletotic morphisms associated to quasicentral approximate units.
This is done in the following proposition, which is essentially the simpletotic analogue of \cite[Proposition 5.5]{GueHigTro} including the slight generalization mentioned in its proof. 
Since we have decided to reformulate it, we will also give a reformulation of the proof.

\begin{prop}\label{prop:extensionsimpletoticmorphism}
Let
\[0\to J\to B\xrightarrow{\pi} A\to 0\]
be a short exact sequence of $G$-\textCstar-algebras and $J_0\subset J$ a sub-$G$-\textCstar-algebra such that $[B,J_0]\subset J_0$. Here, $J$ has been identified with its image in $B$.
Given a quasicentral approximate unit $\{u_m\}_{m\in\fm}$ for $J$ in $B$, we extend it affine linearily to a continuous function 
\[u\colon |\Delta^\fm|\to J\,,\quad x\mapsto u_x\,.\]
Then there is a simpltotic morphism
\[\sigma\colon \Sigma A\to \Simpl^\fm J\]
which maps an elementary tensor $f\grtensor \pi(b)$ to the class of the function $\Delta^\fm\to J, x\mapsto f(u_x)b$.
\end{prop}

\begin{proof}
According to the first part of the lemma, The $*$-ho\-mo\-mor\-phism 
\[\Cz((0,1))\to\Simpl^\fm J_0\subset\Simpl^\fm J\subset\Simpl^\fm B\,,\quad f\mapsto [x\mapsto f(u_x)]\]
is $G$-equivariant with respect to the trivial $G$-action on the domain. The second part of the lemma implies that it commutes with the inclusion as constant functions $B\hookrightarrow\Simpl^\fm B,b\mapsto [x\mapsto b]$.
Therefore, the universal property of the maximal tensor product yields an induced $G$-equivariant $*$-ho\-mo\-mor\-phism
\(\Sigma B\to \Simpl^\fm B\)
which maps elementary tensors $f\grtensor b$ to $[x\mapsto f(u_x)b]$, whose image is clearly contained in $\Simpl^\fm J$. It vanishes on $\Sigma J$ by the third part of the lemma and hence factors through the quotient $\Simpl^\fm B/\Simpl^\fm J$, which is isomorphic to $\Simpl^\fm A$ by \Cref{lem:asymptoticalgebrasexact}.
The resulting $*$-ho\-mo\-mor\-phism is $\sigma$.
\end{proof}

\begin{lem}[{cf.\ \cite[Lemma 5.7]{GueHigTro}}]\label{lem:extensionsimpletoticmorphismclassunique}
The class in $\SimplMor{\Sigma A}{J}$ of the $*$-ho\-mo\-mor\-phism $\sigma$ depends only on the short exact sequence and the sub-\textCstar-algebra $J_0$, but not on the choice of the approximate unit.
\end{lem}

\begin{proof}
First of all we note the following: If $\{u_m\}_{m\in\fm}$ is a quasicentral approximate unit for $J_0$ in $B$ and $\alpha\colon\fk\to\fm$ is a cofinal map, then the simpletotic morphism associated to the quasicentral approximate unit $\{u_{\alpha(k)}\}_{k\in\fk}$ is obtained from the simpletotic morphism associated to $\{u_m\}_{m\in\fm}$ by composing with the $*$-ho\-mo\-mor\-phism $\alpha^*\colon\Simpl^\fm J\to\Simpl^\fk J$. Hence their classes in $\SimplMor{\Sigma A}{J}$ agree.

Now, if $\{u_m\}_{m\in\fm}$ and $\{v_n\}_{n\in\fn}$ are two quasicentral approximate units for $J_0$ in $B$, then the first observation allows us to assume $\fm=\fn$ without loss of generality. 
Let $B_1$ denote the $G$-\textCstar-algebra of continuous functions $[0,1]\to B$ which are constant modulo $J$.
Applying the construction of \Cref{prop:extensionsimpletoticmorphism} to the short exact sequence 
\[0\to\Ct([0,1];J)\to B_1\to A\to 0\]
and the quasicentral approximate unit $\{t\mapsto (1-t)u_m+tv_m\}_{m\in\fm}$ for $\Ct([0,1];J_0)$ in $B_1$ yields a $\Simpl^\fm$-homotopy $\Sigma A\to \Simpl^\fm \Ct([0,1];J)$ between the simpletotic morphisms associated to $\{u_m\}_{m\in\fm}$ and $\{v_m\}_{m\in\fm}$.
\end{proof}

\begin{prop}[{cf.\ \cite[Lemma 5.8]{GueHigTro}}]\label{prop:extensionsimpltoticmorphismnatural}
A commuting diagram of short exact sequences of $G$-\textCstar-algebras
\[\xymatrix{
0\ar[r]&J_0\ar[r]\ar[d]&B_0\ar[r]\ar[d]&A_0\ar[r]\ar[d]&0
\\0\ar[r]&J_1\ar[r]&B_1\ar[r]&A_1\ar[r]&0
}\]
gives rise to a commuting diagram
\[\xymatrix{
\Sigma A_0\ar[r]^-{\llbracket\sigma_0\rrbracket}\ar[d]&J_0\ar[d]
\\\Sigma A_1\ar[r]^-{\llbracket\sigma_1\rrbracket}&J_1
}\]
in the simpltotic category.
\end{prop}

This can be proven completely analogously to \cite[Lemma 5.8]{GueHigTro}. We shall also give a simpler alternative proof, which does not even require the most general form of \Cref{prop:extensionsimpletoticmorphism} but only its basic case $J_0=J$.

\begin{proof}
The \emph{mapping cylinder} of an equivariant $*$-ho\-mo\-mor\-phism $\varphi\colon C_0\to C_1$ between $G$-\textCstar-algebras $C_0,C_1$ is defined as 
\[\Zyl(\varphi)\coloneqq\{(c,f)\in C_0\oplus\Ct([0,1];C_1)\mid \varphi(c)=f(0)\}\,.\]
Then the diagram of short exact sequences gives rise to a short exact sequence
\[0\to\Zyl(J_0\to J_1)\to \Zyl(B_0\to B_1)\to \Zyl(A_0\to A_1)\to 0\]
and we choose a quasicentral approximate unit $\{(u_m,f_m)\}_{m\in\fm}$ of $\Zyl(J_0\to J_1)$ in $\Zyl(B_0\to B_1)$.
Associated to it is the simpltotic morphism
\[\Sigma \Zyl(A_0\to A_1)\to \Simpl^\fm\Zyl(J_0\to J_1)\]
and combining it with the canonical $*$-ho\-mo\-mor\-phisms 
\begin{align*}
A_0&\hookrightarrow\Zyl(A_0\to A_1)
&\Zyl(J_0\to J_1)&\to\Ct([0,1];J_1)
\\a&\mapsto (a,t\mapsto a)
&(j,f)&\mapsto f
\end{align*}
yields a $*$-ho\-mo\-mor\-phism
\(\psi\colon\Sigma A_0\to\Simpl^\fm \Ct([0,1];J_1)\).

Now, note that $\{u_m\}_{m\in\fm}$ and $\{f_m(1)\}_{m\in\fm}$ are then quasicentral approximate units for $J_0$ and $J_1$ in $B_0$ and $B_1$, respectively, and let $\sigma_0$, $\sigma_1$ be the simpltotic morphisms associated to them.
Then $\psi$ is exactly the $\Simpl^\fm$-homotopy which shows that the claimed diagram commutes in the simpltotic category.
\end{proof}

\begin{prop}[{cf.\ \cite[Proposition 5.9]{GueHigTro}}]\label{prop:extensionsimpltoticmorphismtensorproduct}
Let $\llbracket\sigma\rrbracket\in\SimplMor{\Sigma A}{J}$
be the morphism associated to a short exact sequence
\[0\to J\to B\to A\to 0\]
of $G$-\textCstar-algebras.
Let $D$ be another $G$-\textCstar-algebra and $\llbracket\sigma_D\rrbracket\in\SimplMor{\Sigma A\grtensor D}{J\grtensor D}$ the morphism associated to the short exact sequence
\[0\to J\grtensor D\to B\grtensor D\to A\grtensor D\to 0\,.\]
Then $\sigma_D=\sigma\grtensor \id_D$.
\end{prop}

\begin{proof}
Let $\{u_m\}_{m\in\fm}$ be a quasicentral approximate unit for $J$ in $B$ and $\{v_n\}_{n\in\fn}$ an approximate unit for $D$. We may assume that both are indexed over the same directed set $\fm=\fn$ and then $\{u_m\grtensor v_m\}_{m\in\fm}$ is a quasicentral approximate unit for $J\grtensor D$ in $B\grtensor D$.
In complete analogy to the proof of \cite[Proposition 5.9]{GueHigTro} one can then show that the simpltotic morphism associated to $\{u_m\grtensor v_m\}_{m\in\fm}$ is obtained by applying the continuous and exact functor $\blank\grtensor D$ to the simpltotic morphism associated to $\{u_m\}_{m\in\fm}$ and the claim follows.
\end{proof}

\begin{prop}[{cf.\ \cite[Proposition 5.11]{GueHigTro}}]
The class 
\[\llbracket\sigma\rrbracket\in\SimplMor{\Sigma \C}{\Sigma \C}\]
associated to the short exact sequence
\[0\to \Sigma \C\to \Cz([0,1))\to \C\to 0\]
is the identity morphism in the simpltotic category.
\end{prop}

\begin{proof}
This proposition can be proven completely analogously to \cite[Proposition 5.11]{GueHigTro}, but it can also be derived from it as follows.

Since the algebras in the sequence are separable, we can choose a quasicentral approximate unit for $\Sigma \C$ in $\Cz([0,1))$ which is indexed by $\N_+$, see \cite[Theorem 1]{Arveson} with \cite[Lemma 5.3]{GueHigTro}.
It is immediate from the construction that the composition of the associated simpltotic morphism 
$\sigma\colon \Sigma \C\to\Simpl^{\N_+}(\Sigma \C)$ with the $*$-ho\-mo\-mor\-phism $\iota^*\colon \Simpl^{\N_+}(\Sigma \C)\to\Asymp(\Sigma \C)$ from \Cref{ex:asymptoticsimpltoticrelation} is the asymptotic morphism which was associated to the sequence in \cite[Proposition 5.5]{GueHigTro} and is the identity morphism in the asymptotic category by \cite[Proposition 5.11]{GueHigTro}. Hence \Cref{thm:asymptoticsimpltoticfunctor} shows that $\llbracket\sigma\rrbracket=\llbracket\tau^*\iota^*\sigma\rrbracket=F_1\llbracket\iota^*\sigma\rrbracket$ is the identity morphism in the simpltotic category.
\end{proof}

The remainder of \cite[Chapter 5]{GueHigTro} starting with Definition 5.13 can be transferred to the simpltotic category with essentially the identical proofs, using what we have developed for our theory so far. Therefore, we are going to state these analogues without any further proof. 

We recall the definition of the mapping cone and two canonically associated $*$-ho\-mo\-mor\-phisms to fix the notation.
\begin{defn}[{cf., e.g., \cite[Definition 5.13]{GueHigTro}}]
The \emph{mapping cone} $\Cone(\theta)$ of a $*$-ho\-mo\-mor\-phism $\Theta\colon B\to A$ is the \textCstar-algebra 
\[\Cone(\theta)\coloneqq\{(b,f)\in B\oplus\Cz([0,1);A)\mid\theta(b)=f(0)\}\,.\]
Associated to it are the two $*$-ho\-mo\-mor\-phisms
\begin{align*}
\alpha\colon\Cone(\theta)&\twoheadrightarrow B&\beta\colon\Sigma A&\hookrightarrow\Cone(\theta)
\\(b,f)&\mapsto b&f&\mapsto (0,f)\,.
\end{align*}
\end{defn}

If $\theta=\pi\colon B\to A$ is an epimorphism with kernel $J\coloneqq\ker(\pi)$, then we have a canonical embedding and a canonical epimorphism
\begin{align*}
\gamma\colon J&\hookrightarrow\Cone(\pi)&\pi_1\colon\Cz([0,1);B)&\twoheadrightarrow \Cone(\theta)
\\j&\mapsto (j,0)&f&\mapsto (f(0),\pi\circ f)
\end{align*}
and a short exact sequence
\begin{equation}\label{eq:ConeSES}
0\to\Sigma J\to \Cz([0,1),B)\xrightarrow{\pi_1} \Cone(\pi)\to 0
\end{equation}

\begin{prop}[{cf.\ \cite[Proposition 5.14]{GueHigTro}}]
The class 
\[\llbracket\sigma\rrbracket\in\SimplMor{\Sigma \Cone(\pi)}{\Sigma J}\]
associated to the short exact sequence \eqref{eq:ConeSES} and the class of the $*$-ho\-mo\-mor\-phism 
\[\Sigma\gamma\coloneqq \id\grtensor\gamma\colon \Sigma J \to \Sigma \Cone(\pi)\]
are mutually inverse to each other. \qed
\end{prop}

\begin{lem}[{cf.\ \cite[Lemma 5.15]{GueHigTro}}]
Let $\llbracket\sigma\rrbracket\in\SimplMor{\Sigma A}{J}$ be the class associated to the short exact sequence
\[0\to J\to A\to B\xrightarrow{\pi} 0\,.\]
Then the composition 
\[\Sigma^2A\xrightarrow{\id\grtensor\llbracket\sigma\rrbracket}\Sigma J\xrightarrow{\llbracket\Sigma\tau\rrbracket}\Sigma\Cone(\pi)\]
in the simpltotic category agrees with the class of $\Sigma\beta\colon\Sigma^2 A\to\Sigma\Cone(\pi)$. \qed
\end{lem}

\begin{prop}[{cf.\ \cite[Proposition 5.16]{GueHigTro}}]
Let $\theta\colon B\to A$ be a $*$-ho\-mo\-mor\-phism and let $D$ be a \textCstar-algebra. Then the sequence of pointed sets
\[\SimplMor{D}{\Cone(\theta)}\xrightarrow{\alpha_*}\SimplMor{D}{B}\xrightarrow{\theta_*}\SimplMor{D}{A}\]
is exact. \qed
\end{prop}

\begin{prop}[{cf.\ \cite[Proposition 5.18]{GueHigTro}}]
Let $\theta\colon B\to A$ be a $*$-ho\-mo\-mor\-phism and let $D$ be a \textCstar-algebra. If the vertical suspension maps $\Sigma\coloneqq\id\grtensor\blank$ in the commuting diagram
\[\xymatrix{
&\SimplMor{A}{D}\ar[d]^{\Sigma}\ar[r]^{\theta^*}
&\SimplMor{B}{D}\ar[d]^{\Sigma}
\\\SimplMor{\Cone(\theta)}{\Sigma D}\ar[r]^-{\beta^*}
&\SimplMor{\Sigma A}{\Sigma D}\ar[r]^{\Sigma\theta^*}
&\SimplMor{\Sigma B}{\Sigma D}
}\]
are isomorphisms, then the bottom row is exact. \qed
\end{prop}

\section{$\EE$-theory}
\label{sec:Etheory}

In this section we are finally going to use our new simpletotic category to generalize $\EE$-theory to non-separable $G$-\textCstar-algebras. 
Let us recall its definition in the separable case and with a countable discrete group $G$:
\begin{itemize}
\item In the ungraded picture of \cite[Chapter 6]{GueHigTro}, the $\EE$-theory of two separable ungraded $G$-\textCstar-algebras $A,B$ is defined as
\[\EE_G(A,B)\coloneqq \SimplMor[\infty]{\Sigma A\otimes\Kom(\Hilbert\otimes\ell^2G)}{\Sigma B\otimes\Kom(\Hilbert\otimes\ell^2G)}\]
where $\Hilbert$ denotes a fixed separable infinite dimensional ungraded Hilbert space.
\item In the graded picture of \cite[Section 2.7]{HigGue}, the $\EE$-theory of two separable graded $G$-\textCstar-algebras $A,B$ is defined as
\begin{equation}
\label{eq:Etheoryseparabledefinition}
\EE_G(A,B)\coloneqq \SimplMor[\infty]{\grS\grtensor A\grtensor\Kom(\hat\Hilbert\grtensor\ell^2G)}{B\grtensor\Kom(\hat\Hilbert\grtensor\ell^2G)}
\end{equation}
where $\hat\Hilbert\coloneqq\Hilbert\oplus\Hilbert$ is the graded Hilbert space whose even and odd part are both isomorphic to a fixed separable Hilbert space $\Hilbert$ and $\grS$ is the \textCstar-algebra $\Cz(\R)$ but graded into even and odd functions.
\end{itemize}

We will develop our $\EE$-theory for non-separable \textCstar-algebras in analogy to the graded picture \eqref{eq:Etheoryseparabledefinition}. 
Let us note right away that we do not have to restrict ourselves to countable $G$, because there is no need for the \textCstar-algebra $\Kom(\ell^2G)$ to be separable in our set-up. Henceforth $G$ will denote any discrete group.

One could now define $\EE$-theory for non-separable \textCstar-algebras simply by replacing the asymptotic morphism sets $\SimplMor[\infty]{\blank}{\blank}$ in \eqref{eq:Etheoryseparabledefinition} by the corresponding simpltotic morphism sets $\SimplMor{\blank}{\blank}$. This would give us a viable bivariant $\K$-theory with all products and exact sequences, but the stability is not quite yet as we want it: Tautologically there are isomorphisms $\EE_G(A,B)\cong\EE_G(A\grtensor\Kom(\hat\Hilbert_0),B)\cong\EE_G(A,B\grtensor\Kom(\hat\Hilbert_0))$ for all \emph{separable} graded non-zero $G$-Hilbert spaces $\hat\Hilbert_0$, because $\hat\Hilbert_0\grtensor\Hilbert\grtensor\ell^2G\cong\Hilbert\grtensor\ell^2G$ (cf.\ \Cref{lem:switchoffequivariance} below), but in $\EE$-theory for non-separable \textCstar-algebras it would be nice if they would hold for non-separable $G$-Hilbert spaces $\hat\Hilbert_0$, too.

To achieve this, we need to modify the definition further so that instead of a fixed separable Hilbert space $\Hilbert$ we consider arbitrary large unseparable Hilbert spaces as well. 
This requires a bit of work, which will be carried out in the following subsection.

\subsection{Definition}
We need a few lemmas about ($G$-)Hilbert spaces. 

\begin{lem}\label{lem:switchoffequivariance}
Let $\Hilbert_1,\Hilbert_2$ be two $G$-Hilbert spaces. If they are isometrically isomorphic (possibly non-equivariantly), then $\Hilbert_1\otimes\ell^2G,\Hilbert_2\otimes\ell^2G$ are equivariantly isometrically isomorphic.
This applies in particular if $\Hilbert_2=\Hilbert_1$ as Hilbert spaces, but $\Hilbert_2$ is equipped with the trivial $G$-action.
\end{lem}
\begin{proof}
Given a unitary operator $U\colon\Hilbert_1\to\Hilbert_2$, then $V(v\otimes\delta_g)\coloneqq (g\cdot U(g^{-1}v))\otimes\delta_g$ defines an equivariant unitary operator $V\colon \Hilbert_1\otimes\ell^2G,\Hilbert_2\otimes\ell^2G$, cf.\ \cite[Lemma 2.3]{MingoPhillips}.
\end{proof}

\begin{lem}\label{lem:isometriesexist}
Let $\Hilbert_1,\Hilbert_2$ be $G$-Hilbert spaces such that $\dim(\Hilbert_2)\geq\dim(\Hilbert_1)$ (as cardinal numbers).
Then there exists an equivariant isometry
\[\ell^2G\otimes\Hilbert_1\hookrightarrow\ell^2G\otimes\Hilbert_2\,.\]
\end{lem}

\begin{proof}
Because of \Cref{lem:switchoffequivariance} we may assume without loss of generality that $G$ acts trivially on $\Hilbert_1,\Hilbert_2$. Then the equivariant isometry is induced by any isometry $\Hilbert_1\to\Hilbert_2$.
\end{proof}

\begin{lem}\label{lem:isometriesarehomotopic}
Let $\Hilbert_1,\Hilbert_2$ be $G$-Hilbert spaces, $\Hilbert_2$ infinite dimensional. Then any two equivariant isometries
\[V_1,V_2\colon \Hilbert_1\hookrightarrow\ell^2G\otimes\Hilbert_2\]
are equivariantly $*$-strongly homotopic.
Hence the induced $*$-ho\-mo\-mor\-phisms 
\[\Ad_{V_i}\colon \Kom(\Hilbert_1)\hookrightarrow\Kom(\ell^2G\otimes\Hilbert_2)\,,\quad T\mapsto V_iTV_i^*\]
are equivariantly homotopic.
\end{lem}

\begin{proof}
Again we can assume that the $G$-action on $\Hilbert_2$ is trivial. 
Since $\Hilbert_2$ is infinite dimensional, there are isometric isomorphisms
\[\Hilbert_2\cong \Hilbert_2\otimes\ell^2\N\cong \Hilbert_2\otimes(\ell^2\N\oplus\ell^2\N)\,.\]
It is well-known that the isomorphism $\ell^2\N\cong\ell^2\N\oplus\ell^2\N$ is $*$-strongly homotopic to the inclusions as one of the summands. Therefore, we can assume that the images of $V_1,V_2$ are contained in orthogonal subspaces.
But then they can be homotoped into each other by a simple rotation.
\end{proof}

\begin{rem}
Note that the three lemmas above also hold for graded Hilbert spaces, if the assumptions are satisfied for the even and odd parts separately.
\end{rem}

A direct consequence of the last two lemmas is the next corollary.

\begin{cor}\label{cor:directedsystemoverGHilbert}
Let $\Hilbert_1,\Hilbert_2$ be $G$-Hilbert spaces whose even and odd parts $\Hilbert_i^\pm$ satisfy $\dim(\Hilbert_2^\pm)\geq\max\{\dim(\Hilbert_1^\pm),\aleph_0\}$ and let $A,B$ be $G$-\textCstar-algebras.
Then there are canonical maps
\begin{align*}
\SimplMor{A\grtensor\Kom(\ell^2G)}{B\grtensor\Kom(\hat\Hilbert_1\grtensor\ell^2G)}
&\to \SimplMor{A\grtensor\Kom(\ell^2G)}{B\grtensor\Kom(\hat\Hilbert_2\grtensor\ell^2G)}\,,
\\\SimplMor[\fm]{A\grtensor\Kom(\ell^2G)}{B\grtensor\Kom(\hat\Hilbert_1\grtensor\ell^2G)}
&\to \SimplMor[\fm]{A\grtensor\Kom(\ell^2G)}{B\grtensor\Kom(\hat\Hilbert_2\grtensor\ell^2G)}
\end{align*}
induced by adjoining with an equivariant isometry.
Together they form a directed system of sets.\qed
\end{cor}

\begin{lem}[{cf.\ \cite[Lemma 6.25]{GueHigTro}}]\label{lem:removecompactsfromfirstentry}
Let $\Hilbert$ be an infinite dimensional Hilbert space, $\hat\Hilbert\coloneqq\Hilbert\oplus\Hilbert$ the graded Hilbert space with even and odd part isomorphic to $\Hilbert$ and $p\in\Hilbert\subset\hat\Hilbert$ a rank one projection in the even part of $\hat\Hilbert$.
Then the canonical map
\[\SimplMor{A\otimes\Kom(\hat\Hilbert)}{B\otimes\Kom(\hat\Hilbert)}\to \SimplMor{A}{B\otimes\Kom(\hat\Hilbert)}\]
induced by the $*$-ho\-mo\-mor\-phism $\id\otimes p\colon A\to A\otimes\Kom(\hat\Hilbert)$ is a bijection.
\end{lem}
\begin{proof}
We claim that an inverse is given by taking the tensor product with $\Kom(\hat\Hilbert)$ and identifying $\Kom(\hat\Hilbert\otimes\hat\Hilbert)$ with $\Kom(\hat\Hilbert)$ via an isomorphism $\hat\Hilbert\otimes\hat\Hilbert\cong\hat\Hilbert$, which exists because the even and odd parts of both sides have the same infinite cardinality.

Indeed, first precomposing with $\id\otimes p$ and then tensoring with $\Kom(\hat\Hilbert)$ has the same effect as precomposing with the $*$-ho\-mo\-mor\-phism $A\otimes\Kom(\hat\Hilbert)\to A\otimes\Kom(\hat\Hilbert)$ which is obtained by conjugation with the isometry $\hat\Hilbert\cong (p\hat\Hilbert)\otimes\hat\Hilbert \subset\hat\Hilbert\otimes\hat\Hilbert\cong\hat\Hilbert$ and thus homotopic to the identity by \Cref{lem:isometriesarehomotopic} (applied non-equivariantly, i.\,e.\ with $G=1$).
On the other hand, tensoring with $\Kom(\hat\Hilbert)$ first and then precomposing with $\id\otimes p$ has the same effect as postcomposing with the $*$-ho\-mo\-mor\-phism $B\otimes\Kom(\hat\Hilbert)\to B\otimes\Kom(\hat\Hilbert)$ obtained by conjugation with the isometry $\hat\Hilbert\cong \hat\Hilbert\otimes(p\hat\Hilbert) \subset\hat\Hilbert\otimes\hat\Hilbert\cong\hat\Hilbert$ and the claim follows.
\end{proof}

This last lemma implies that for countable $G$ and separable graded $G$-\textCstar-algebras $A,B$ the $\EE$-theory groups defined by \eqref{eq:Etheoryseparabledefinition} are in bijection with the sets
\[\SimplMor{\grS\grtensor A\grtensor\Kom(\ell^2G)}{B\grtensor\Kom(\hat\Hilbert\grtensor\ell^2G)}\,,\]
again with the same fixed $\hat\Hilbert\coloneqq\Hilbert\oplus\Hilbert$.
This prompts us to make the following definition for general $G,A,B$.

\begin{defn}\label{defn:Etheory}
Let $G$ be a discrete group and $A,B$ graded $G$-\textCstar-algebras. Then we define the \emph{$\EE$-theory of $A,B$} as the set
\[\EE_G(A,B)\coloneqq \varinjlim_{\hat\Hilbert}\SimplMor{\grS\grtensor A\grtensor\Kom(\ell^2G)}{B\grtensor\Kom(\hat\Hilbert\grtensor\ell^2G)}\]
where the limit runs over the directed system obtained from \Cref{cor:directedsystemoverGHilbert}.
\end{defn}

\begin{rem}\label{rem:GHilbertvsHilbert}
Note that by \Cref{lem:switchoffequivariance} we could equally well let $\hat\Hilbert$ in the definition run only over all non-equivariant graded Hilbert spaces.
This is useful for some of the upcoming proofs, but we do not want to restrict completely to this cofinal directed subsystem, because we cannot exclude that considering $G$-Hilbert spaces might provide the more natural set-up in future applications.
\end{rem}

We have to justify that the direct limit exists and that we retrieve the old definition for countable $G$ and separable $A,B$.
\begin{lem}\label{lem:directlimitoverHilbertexists}
Let $D\subset A\grtensor\Kom(\ell^2G)$ be a dense subset which is closed under involution, grading automorphism and $G$-action and let $\hat\Hilbert'$ be a graded $G$-Hilbert space whose even and odd parts have dimension greater or equal to the cardinality $|D|$ of $D$.
The direct limit 
\[\varinjlim_{\hat\Hilbert}\SimplMor{A\grtensor\Kom(\ell^2G)}{B\grtensor\Kom(\hat\Hilbert\grtensor\ell^2G)}\]
exists and is isomorphic to 
\[\SimplMor{A\grtensor\Kom(\ell^2G)}{B\grtensor\Kom(\hat\Hilbert'\grtensor\ell^2G)}\,.\]\end{lem}

\begin{proof}
In this proof, we let the limit run only over non-equivariant Hilbert spaces, as proposed in \Cref{rem:GHilbertvsHilbert}. In particular, we also use \Cref{lem:switchoffequivariance} to assume w.l.o.g.\ that $G$ acts trivially on $\hat\Hilbert'$.

Note that the sets $\SimplMor[\fm]{A\grtensor\Kom(\ell^2G)}{B\grtensor\Kom(\hat\Hilbert\grtensor\ell^2G)}$ constitute a directed system indexed by pairs $(\fm,\hat\Hilbert)$ and the direct limit we are looking for is equal to the direct limit of this system.
We have seen that these sets are canonically isomorphic to $\SimplMor[\fm(D)]{A\grtensor\Kom(\ell^2G)}{B\grtensor\Kom(\hat\Hilbert\grtensor\ell^2G)}$ for all $\fm\geq\fm(D)$ and hence we only have to show that the latter become stationary for all $\hat\Hilbert$ large enough.

To this end, assume $A\not=0$. (Otherwise the claim is trivial.) Then $D$ is infinite and we note that the cardinality $|\fm(D)|$ of $\fm(D)$ is equal to $|D|$.
Consider a graded Hilbert space $\hat\Hilbert(D)$ whose even and odd part have dimension equal to $|D|$ and a graded Hilbert space $\hat\Hilbert$ whose even and odd part have dimension at least $|D|$.

Consider a simpltotic morphism $\varphi\colon A\grtensor\Kom(\ell^2G)\to\Simpl^{\fm(D)}(B\grtensor\Kom(\hat\Hilbert\grtensor\ell^2G))$. For each $d\in D$ we choose a representative $f_d\in\Cb(|\Delta^{\fm(D)}|;B\grtensor\Kom(\hat\Hilbert\grtensor\ell^2G))$ of $\varphi(d)$. 
Note that $|\Delta^{\fm(D)}|$ contains a dense subset with cardinality $|D|$, because each of its $k$-skeleta contains a dense subset of cardinality $|\fm(D)|^k\cdot\aleph_0=|D|$.
Therefore, the image of $f_d$ is contained in the sub-\textCstar-algebra $B\grtensor\Kom(\hat\Hilbert_d\grtensor\ell^2G)$ for a graded Hilbert subspace $\hat\Hilbert_d\subset\hat\Hilbert$ whose even and odd part have dimension at most $|D|$. 
This implies that the image of $\varphi$ is contained in $\Simpl^{\fm(D)}(B\grtensor\Kom(\hat\Hilbert'\grtensor\ell^2G))$
for $\hat\Hilbert'\coloneqq\bigoplus_{d\in D}\hat\Hilbert_d$.
Since even and odd part of $\hat\Hilbert'$ have dimension at most $|D|^2=|D|$, it is contained in the image of some isometry $\hat\Hilbert(D)\hookrightarrow\hat\Hilbert$, which shows that $\llbracket\varphi\rrbracket$ lies in the image of the map
\begin{equation}\label{eq:HilbertDtoHilbert}
\SimplMor[\fm(D)]{A\grtensor\Kom(\ell^2G)}{B\grtensor\Kom(\hat\Hilbert(D)\grtensor\ell^2G)}\to \SimplMor[\fm(D)]{A\grtensor\Kom(\ell^2G)}{B\grtensor\Kom(\hat\Hilbert\grtensor\ell^2G)}\,.
\end{equation}

The same argument applied to $\Simpl^{\fm(D)}$-homotopies shows that the map \eqref{eq:HilbertDtoHilbert} is not only surjective, but also injective, proving the stationarity of the directed system.

The claim now follows from \Cref{cor:existenceofsimpletoticmorphismset}.
\end{proof}

\begin{cor}\label{cor:Etheoryexistsandistheoldoneintheseparablecountablecase}
The $\EE$-theory $\EE_G(A,B)$ defined in \Cref{defn:Etheory} exists.
Furthermore, if $\EE_G^\Asymp(A,B)$ denotes the groups defined by \eqref{eq:Etheoryseparabledefinition}, then there is a canonical map $\EE_G^\Asymp(A,B)\to\EE_G(A,B)$. It is a bijection if $G$ is countable and $A$ is separable.
\end{cor}

\begin{proof}
Existence follows directly from \Cref{lem:directlimitoverHilbertexists} for any suitable $D\subset \grS\grtensor A\grtensor\Kom(\ell^2G)$.

Let $\hat\Hilbert$ be any graded Hilbert space with infinite dimensional even and odd parts.
Then the functor from from \Cref{thm:asymptoticsimpltoticfunctor} and \Cref{lem:removecompactsfromfirstentry} provide us with a map
\begin{align*}
\EE_G^\Asymp(A,B)&\coloneqq \SimplMor[\infty]{\grS\grtensor A\grtensor\Kom(\hat\Hilbert\grtensor\ell^2G)}{B\grtensor\Kom(\hat\Hilbert\grtensor\ell^2G)}
\\&\to\SimplMor{\grS\grtensor A\grtensor\Kom(\hat\Hilbert\grtensor\ell^2G)}{B\grtensor\Kom(\hat\Hilbert\grtensor\ell^2G)}
\\&\cong\SimplMor{\grS\grtensor A\grtensor\Kom(\ell^2G)}{B\grtensor\Kom(\hat\Hilbert\grtensor\ell^2G)}
\to \EE_G(A,B)\,.
\end{align*}
If $G$ is countable and $A,\hat\Hilbert$ are separable, then the first map is a bijection by  \Cref{cor:asymptoticsimpltoticrelation}, because $\grS\grtensor A\grtensor\Kom(\hat\Hilbert\grtensor\ell^2G)$ is separable.
Furthermore, \Cref{lem:directlimitoverHilbertexists} implies that the last map is also a bijection, because a countable $D$ can be chosen in this case.
\end{proof}

\subsection{Algebraic structure}

In this section we are going to introduce the algebraic structure of $\EE$-theory: The group structure, composition product, exterior tensor product and descent functor. We will only give the constructions, which are very similar to the corresponding ones in \cite[Section 2]{HigGue} and \cite[Chapter 6]{GueHigTro}, but we omit the proofs, because they are essentially identical to the ones in the cited references.

We already mentioned in the proof of \Cref{lem:directlimitoverHilbertexists} that $\EE$-theory is actually a direct limit of a directed system indexed by pairs $(\fm,\hat\Hilbert)$ and hence we have
\[\EE_G(A,B)\cong \varinjlim_{\fm}\varinjlim_{\hat\Hilbert}\SimplMor[\fm]{\grS\grtensor A\grtensor\Kom(\ell^2G)}{B\grtensor\Kom(\hat\Hilbert\grtensor\ell^2G)}\,,\]
that is, we were allowed to interchange the order of the direct limits.
In order to define the group structure on $\EE_G(A,B)$ it is therefore sufficient to define compatible group structures on all direct limits of the form 
\begin{equation}\label{eq:limitonlyoverH}
\varinjlim_{\hat\Hilbert}\SimplMor[\fm]{\grS\grtensor A\grtensor\Kom(\ell^2G)}{B\grtensor\Kom(\hat\Hilbert\grtensor\ell^2G)}\,.
\end{equation}
To this end, note that there are obvious canonical injective $*$-ho\-mo\-mor\-phisms
\[\Simpl^\fm(B\grtensor\Kom(\ell^2G\grtensor\hat\Hilbert))\oplus \Simpl^\fm(B\grtensor\Kom(\ell^2G\grtensor\hat\Hilbert'))\to \Simpl^\fm(B\grtensor\Kom(\ell^2G\grtensor(\hat\Hilbert\oplus\hat\Hilbert')))\]
which immediately gives rise to well-defined addition maps
\begin{align*}
&\SimplMor[\fm]{\grS\grtensor A\grtensor\Kom(\ell^2G)}{B\grtensor\Kom(\ell^2G\grtensor\hat\Hilbert)}\times \SimplMor[\fm]{\grS\grtensor A\grtensor\Kom(\ell^2G)}{B\grtensor\Kom(\ell^2G\grtensor\hat\Hilbert')}
\\&\qquad\to \SimplMor[\fm]{\grS\grtensor A\grtensor\Kom(\ell^2G)}{B\grtensor\Kom(\ell^2G\grtensor(\hat\Hilbert\oplus\hat\Hilbert'))}
\end{align*}
for all graded Hilbert spaces $\hat\Hilbert,\hat\Hilbert'$.
They clearly induce an addition map on the direct limit \eqref{eq:limitonlyoverH}.

The following properties are straightforward to verify. In particular, the existence of the additive inverse is proven exactly as in \cite[Lemma 2.1]{HigGue}.

\begin{lem}
The addition on the sets \eqref{eq:limitonlyoverH} defined above is associative and commutative. The zero simpltotic morphism is a neutral element and an inverse of a class represented by $\varphi\colon \grS\grtensor A\grtensor\Kom(\ell^2G)\to \Simpl^\fm(B\grtensor\Kom(\hat\Hilbert\grtensor\ell^2G))$ is obtained by composing it with the grading automorphism $\grS\to\grS,f\mapsto (x\mapsto f(-x))$ and reversing the grading on $\hat\Hilbert$. Thus, these sets are abelian groups. 

Furthermore, the group structure is natural in $\fm$. \qed
\end{lem}

\begin{defn}
We equip $\EE_G(A,B)$ with the group structure obtained by taking the direct limit over $\fm$ of the abelian groups described by the preceding lemma.
\end{defn}

Before we come to the products, we have to recall from \cite[Section 1.3]{HigGue} that the graded \textCstar-algebra $\grS$ is also a coalgebra with counit $\nu\colon \grS\to\C,f\mapsto f(0)$ and comultiplication $\Delta\colon\grS\to\grS\grtensor\grS,f\mapsto f(X\grtensor 1+1\grtensor X)$, where $X\colon\R\to\R$ denotes the identity function, considered as an degree one, essentially self-adjoint, unbounded multiplier of $\grS$.
Formulas involving $\Delta$ and $\nu$ are often readily verified on the generators $u(x)=\exp(-x^2)$ and $v(x)=x\exp(-x^2)$, for which $\Delta(u)=u\grtensor u$, $\Delta(v)=u\grtensor v+v\grtensor u$, $\eta(u)=1$ and $\eta(v)=0$.

\begin{thm}[{cf.\ \cite[Theorem 2.3]{HigGue}}]
The $\EE$-theory groups $\EE_G(A,B)$ are the morphism sets of an additive category $\EE_G$ whose objects are the graded $G$-\textCstar-algebras. 
The composition product 
\[\EE_G(A,B)\times\EE_G(B,C)\to\EE_G(A,C)\]
maps the classes represented by two morphisms
\begin{align*}
\varphi&\in\SimplMor{\grS\grtensor A\grtensor\Kom(\ell^2G)}{B\grtensor\Kom(\hat\Hilbert\grtensor\ell^2G)}
\\\psi&\in\SimplMor{\grS\grtensor B\grtensor\Kom(\ell^2G)}{C\grtensor\Kom(\hat\Hilbert'\grtensor\ell^2G)}
\end{align*}
in the simpltotic category to the element represented by the composition
\begin{align*}
\grS\grtensor A\grtensor\Kom(\ell^2G)&\xrightarrow{\Delta\grtensor\id_{\Kom(\ell^2G)}}\grS\grtensor \grS\grtensor A\grtensor\Kom(\ell^2G)
\\&\xrightarrow{\id_{\grS}\grtensor\varphi}\grS\grtensor B\grtensor\Kom(\hat\Hilbert\grtensor\ell^2G)
\\&\xrightarrow{\psi\grtensor\id_{\Kom(\hat\Hilbert)}}\grS\grtensor C\grtensor\Kom(\ell^2G\grtensor\hat\Hilbert\grtensor\hat\Hilbert')
\end{align*}
where $\Delta\colon\grS\grtensor\grS$ denotes the comultiplication.

There is a functor from the homotopy category of graded $G$-\textCstar-algebras and graded $*$-ho\-mo\-mor\-phisms to $\EE_G$ which is the identity on objects and maps the homotopy class of a $*$-ho\-mo\-mor\-phism $\varphi\colon A\to B$ to the $\EE_G$ represented by the composition 
\[\grS\grtensor A\grtensor\Kom(\ell^2G)\xrightarrow{\eta\otimes\id_{A\grtensor\Kom(\ell^2G)}}A\grtensor\Kom(\ell^2G)\xrightarrow{\varphi\otimes\id_{\grtensor\Kom(\ell^2G)}}B\grtensor\Kom(\ell^2G\otimes\C)\]
where $\eta\colon\grS\to\C$ denotes the counit of $\grS$.
\qed
\end{thm}

Furthermore, we have some operations which come from continuous and exact functors presented in \ref{sec:functors}.

\begin{thm}[{cf.\ \cite[Theorem 2.4]{HigGue}}]\label{thm:maximaltensorproduct}
There is a functorial maximal tensor product
\[\EE_G(A,B)\otimes\EE_H(C,D)\to\EE_{G\times H}(A\grtensor C,B\grtensor D)\]
which maps the classes represented by two morphisms
\begin{align*}
\varphi&\in\SimplMor{\grS\grtensor A\grtensor\Kom(\ell^2G)}{B\grtensor\Kom(\hat\Hilbert\grtensor\ell^2G)}
\\\psi&\in\SimplMor{\grS\grtensor C\grtensor\Kom(\ell^2H)}{D\grtensor\Kom(\hat\Hilbert'\grtensor\ell^2H)}
\end{align*}
to the class represented by the composition
\begin{align*}
\grS\grtensor A\grtensor C\grtensor\Kom(\ell^2(G\times H))
&\xrightarrow{\Delta\grtensor\id}\grS\grtensor \grS\grtensor A\grtensor C\grtensor\Kom(\ell^2G\grtensor\ell^2H)
\\&\cong\grS\grtensor A\grtensor\Kom(\ell^2G)\grtensor \grS\grtensor C\grtensor\Kom(\ell^2H)
\\&\xrightarrow{\varphi\otimes\psi}C\grtensor\Kom(\ell^2G\grtensor\hat\Hilbert)\grtensor D\grtensor\Kom(\ell^2H\grtensor\hat\Hilbert')
\\&\cong C\grtensor D\grtensor\Kom(\ell^2(G\times H)\grtensor\hat\Hilbert\grtensor\hat\Hilbert')\,.
\end{align*}\qed
\end{thm}

In equivariant $\EE$-theory, one often considers a different maximal tensor product, namely a functor 
\[\EE_G(A,B)\otimes\EE_G(C,D)\to\EE_{G}(A\grtensor C,B\grtensor D)\,,\]
see \cite[Theorem 6.21]{GueHigTro}.
In order to obtain it, we only need to specialise the one from \Cref{thm:maximaltensorproduct} to the case $G=H$ and apply the following restriction functor to the diagonal inclusion $G\to G\times G$.

\begin{lem}
Let $f\colon H\to G$ be a homomorphism between discrete groups. Then there is a restriction functor $f^*\colon \EE_G\to \EE_H$ which maps a $G$-\textCstar-algebras to itself considered as an $H$-\textCstar-algebras via $f$ and whose image of a class in $\EE_G(A,B)$ represented by a $G$-equivariant morphism
\[\grS\grtensor A\grtensor\Kom(\ell^2G)\to B\grtensor\Kom(\hat\Hilbert\grtensor\ell^2G)\]
is obtained as follows: Tensoring $\varphi$ with $\Kom(\ell^2H)$ yields a $G\times H$-equivariant morphism 
\[\grS\grtensor A\grtensor\Kom(\ell^2G\grtensor\ell^2H)\to B\grtensor\Kom(\hat\Hilbert\grtensor\ell^2G\grtensor\ell^2H)\]
but we consider it only as a $H$-equivariant morphism via the homomorphism $H\to G\times H,h\mapsto (f(h),h)$.
Then \Cref{lem:switchoffequivariance} allows us to switch off the $H$-action on $\ell^2G$ and afterwards we apply   \Cref{lem:removecompactsfromfirstentry} to remove the Hilbert space $\ell^2G$ from the left hand side. As a result, we obtain a $H$-equivariant morphism
\[\grS\grtensor A\grtensor\Kom(\ell^2H)\to B\grtensor\Kom(\hat\Hilbert\grtensor\ell^2G\grtensor\ell^2H)\]
which represents the image in $\EE_H(A,B)$. \qed
\end{lem}

\begin{thm}[{cf.\ \cite[Theorem 2.13]{HigGue}}]
There is a descent functor from the equivariant $\EE$-theory category $\EE_G$ to the non-equivariant $\EE$-theory category $\EE\coloneqq\EE_1$ which maps the $G$-\textCstar-algebra $A$ to the full crossed product $\Cstar(G;A)$ and which maps the $\EE$-theory class of a $G$-equivariant $*$-ho\-mo\-mor\-phism $A\to B$ to the $\EE$-theory class of the induced $*$-ho\-mo\-mor\-phism $\Cstar(G;A)\to \Cstar(G;B)$.\qed
\end{thm}

Comparing all of these algebraic structures with the corresponding ones defined in \cite{GueHigTro,HigGue}, the following becomes evident.
\begin{thm}
The canonical maps $\EE_G^\Asymp(A,B)\to\EE_G(A,B)$ defined in \Cref{cor:Etheoryexistsandistheoldoneintheseparablecountablecase}
constitute an additive functor which is compatible with exterior products, restriction functors, full crossed products and the canonical functors from the homotopy category $G\mbox{-}\mathrm{C^*\mbox{-}alg}$ of graded $G$-\textCstar-algebras and graded $*$-ho\-mo\-mor\-phisms to $\EE$-theory in the sense that the following diagrams of functors commute.
\[\xymatrix{
\EE_G^\Asymp\otimes\EE_H^\Asymp\ar[r]^{\otimes}\ar[d]&\EE_{G\times H}^\Asymp\ar[d]
&\EE_G^\Asymp\ar[r]\ar[d]&\EE_H^\Asymp\ar[d]
\\\EE_G\grtensor\EE_H\ar[r]^{\grtensor}&\EE_{G\times H}
&\EE_G\ar[r]&\EE_H
\\\EE_G^\Asymp\ar[r]^{\Cstar(G;\blank)}\ar[d]&\EE^\Asymp\ar[d]
&G\mbox{-}\mathrm{C^*\mbox{-}alg}\ar[r]\ar[dr]&\EE_G^\Asymp\ar[d]
\\\EE_G\ar[r]^{\Cstar(G;\blank)}&\EE
&&\EE_G
}\]
\qed
\end{thm}

Therefore, all calculations performed in the asymptotic picture of $\EE$-theory are still valid in the simpltotic picture. This is important, because it is a lot easier to work with the asymptotic version of the composition product than with the simpltotic one.

\subsection{Properties}
\label{sec:EtheoryProperties}
Most of the properties of $\EE$-theory given in \cite[Section 2]{HigGue} and \cite[Chapter 6,7]{GueHigTro} carry over to our new definition. We shall only list them, because their proofs are essentially identical to those in the asymptotic case.

The following stability property of $\EE$-theory is obvious from \Cref{lem:switchoffequivariance,lem:removecompactsfromfirstentry,defn:Etheory}.
As advertised at the beginning of the section, we do not require any separability assumption.
\begin{thm}\label{thm:Etheorystable}
For every graded $G$-Hilbert space $\hat\Hilbert$ there are canonical natural isomorphisms
\[\EE_G(A,B\grtensor\Kom(\hat\Hilbert))\cong\EE_G(A,B)\cong \EE_G(A\grtensor\Kom(\hat\Hilbert),B)\,.\]
\qed
\end{thm}

Next we note that Bott periodicity passes over directly from the asymptotic case as follows. To see this, we recall the set-up from \cite[Section 1.11]{HigGue}: Let $V$ be a finite dimensional euclidean vector space and let $\Cl(V)$ denote the complex Clifford algebra over $V$. We choose the sign convention in which $v^2=\|v\|\cdot 1$ for all $v\in V$ and equip it with the grading in which a product $v_1\cdots v_n$ of vectors has degree $n$ modulo 2. Let $\cC(V)\coloneqq \Cz(V;\Cl(V))$ equiped with the grading coming from the grading of $\Cl(V)$ alone. The function $C\colon V\to\Cl(V),v\mapsto v$ is a degree one, essentially self-adjoint, unbounded operator on $\cC(V)$ with domain the compactly supported functions, the so-called \emph{Clifford operator}.

\begin{thm}[{cf.\ \cite[Theorem 2.6]{HigGue}}]
\label{thm:finitedimBottperiodicity}
For all finite dimensional Euclidean vector spaces $V$, the $*$-ho\-mo\-mor\-phism
\[\beta\colon\grS\to\cC(V)\,,\quad f\mapsto f(C)\]
represents an isomorphism in $\EE^\Asymp(\C,\cC(V))$ and hence also in $\EE(\C,\cC(V))$. \qed
\end{thm}
The inverse of this isomorphism is the fundamental class of the Dirac operator on $V$, which is represented by an asymptotic morphism 
\[\alpha\colon\grS\grtensor\cC(V)\to \Kom(L^2(V;\Cl(V)))\,,\]
see \cite[Proposition 1.5]{HigGue}.
We will review this element in more detail later on in \Cref{sec:infinitedimBott}.

By taking exterior tensor products with the classes represented by $\alpha$ and $\beta$ we immediately obtain the following corollary.
\begin{cor}
There are canonical natural isomorphisms
\[\EE_G(A,B\grtensor\cC(V))\cong\EE_G(A,B)\cong \EE_G(A\grtensor\cC(V),B)\]
for all finite dimensional Euclidean vector spaces $V$.
\qed
\end{cor}

Moreover, the isomorphism 
\[\Cl(\R^{2k})\cong\Cz(\R^{2k})\otimes \Kom(\C^{2^{k-1}}\oplus\C^{2^{k-1}})\cong\Sigma^k\Kom(\C^{2^{k-1}}\oplus\C^{2^{k-1}})\]
and \Cref{thm:Etheorystable} yield the better-known form of Bott periodicity.
\begin{cor}
There are canonical natural isomorphisms
\[\EE_G(A,\Sigma^{2k}B)\cong\EE_G(A,B)\cong \EE_G(\Sigma^{2k}A,B)\]
for all $k\in\N$.
\qed
\end{cor}

Just as in the asymptotic case, Bott periodicity also implies the following result about suspension.
\begin{cor}[{cf.\ \cite[Theorem 2.7]{HigGue}}]
The \emph{suspension map}
\[\EE_G(A,B)\xrightarrow{\blank\grtensor\id_{\Cz(0,1)}}\EE_G(\Sigma A,\Sigma B)\]
is an isomorphism.\qed
\end{cor}

Using the lemmas and propositions from \Cref{sec:extensions} instead of the corresponding ones from \cite[Chapter 5]{GueHigTro}, we now also get the long exact sequences in full generality.

\begin{thm}[{cf.\ \cite[Theorems 6.15, 6.18 \& 6.20]{GueHigTro}}]
Let $D$ be a graded $G$-\textCstar-algebra and let
\[0\to J\to B\to A\to 0\]
be a short exact sequence of $G$-\textCstar-algebras.
Denote by $\llbracket\sigma\rrbracket\in\EE_G(\Sigma A,J)$ the $\EE$-theory class represented by the associated simpltotic morphism $\sigma\colon\Sigma A\to\Simpl^\fm J$ from \Cref{prop:extensionsimpletoticmorphism}.
Then there are two long exact sequences
\begin{gather*}
\dots\to\EE_G(D,\Sigma A)\xrightarrow{\llbracket\sigma\rrbracket\circ\blank}\EE_G(D,J)\to\EE_G(D,B)\to \EE_G(D,A)
\\\EE_G(A,D)\to\EE_G(B,D)\to\EE_G(J,D)\xrightarrow{-\circ\llbracket\sigma\rrbracket}\EE_G(\Sigma A,D)\to\dots
\end{gather*}
which become six-term exact sequences after applying Bott periodicity.
\qed
\end{thm}

Another property which can be proven exactly as in the asymptotic case is additivity, but here it holds even for uncountable directs sums.

\begin{prop}[{cf.\ \cite[Proposition 7.1]{GueHigTro}}]
Let $I$ be an index set and let $A_i$ for $i\in I$ and $B$ be $G$-\textCstar-algebras. Then the group homomorphism
\[\EE_G\left(\bigoplus_{i\in I}A_i,B\right)\to\prod_{i\in I}\EE_G(A_i,B)\]
induced by the projections is an isomorphism.
\end{prop}

\begin{proof}
The inverse can be defined as follows. Given an element 
\[\left(\llbracket \varphi_i\colon \grS\grtensor A_i\grtensor\Kom(\ell^2G)\to \Simpl^{\fm_i}(B\grtensor\Kom(\ell^2G\grtensor\hat\Hilbert_i))\rrbracket\right)_{i\in I}\in \prod_{i\in I}\EE_G(A_i,B)\,,\]
we may assume that all $\fm_i$ are equal. Otherwise we just replace them with $\fm\coloneqq \prod_{i\in I}\fm_i$.
Summing up the $\varphi_i$ yields a simpltotic morphism
\begin{align*}
\grS\grtensor \bigoplus_{i\in I}A_i\grtensor\ell^2G&\cong\bigoplus_{i\in I}(\grS\grtensor A_i\grtensor\ell^2G)
\\&\to \bigoplus_{i\in I}\Simpl^{\fm}(B\grtensor\Kom(\ell^2G\grtensor\hat\Hilbert_i))
\\&\hookrightarrow \Simpl^{\fm}(B\grtensor\Kom(\ell^2G\grtensor\hat\bigoplus_{i\in I}\Hilbert_i))
\end{align*}
and one can show exactly as in the proof of \cite[Proposition 7.1]{GueHigTro} that this construction yields the inverse 
\[\prod_{i\in I}\EE_G(A_i,B)\to\EE_G\left(\bigoplus_{i\in I}A_i,B\right)\]
we are looking for.
\end{proof}

\section{Infinite dimensional Bott periodicity in the light of simpltotic morphisms}
\label{sec:infinitedimBott}
Higson, Kasparov and Trout developed in \cite{HigsonKasparovTrout} a version of Bott periodicity for infinite dimensional Euclidean spaces, which is similar to the one we reviewed in \Cref{thm:finitedimBottperiodicity}. We will revisit it in this final section and see that it naturally fits in the language of simpltotic morphisms.
This will illustrate the concept mentioned at the end of the introduction, that rescalling differential operators differently in different directions should enable us to obtain simpltotic morphisms.

Recall from \cite[Section 2]{HigsonKasparovTrout} and \Cref{sec:EtheoryProperties} that for a finite dimensional Euclidean vector spaces $V$ we denoted by $\Cl(V)$ the complex Clifford algebra over $V$ which is subject to the defining relation $v^2=\|v\|\cdot 1$ for all $v\in V$ and in which a product of vectors $v_1\cdots v_n$ has grading degree $n$ modulo 2.
Furthermore one defines $\cC(V)\coloneqq \Cz(V;\Cl(V))$ equiped with the grading coming from the grading on $\Cl(V)$ alone. The function $C_V\colon V\to\Cl(V),v\mapsto v$ is a degree one, essentially self-adjoint, unbounded multiplier of $\cC(V)$ and gives rise to the $*$-ho\-mo\-mor\-phism 
\[\beta_V\colon\grS\to\cC(V)\,,\quad f\mapsto f(C_V)\,.\]
Note that this is not the $\beta\colon\grS\to\grS\grtensor\cC(V)$ defined in \cite[Section 2]{HigsonKasparovTrout}. The latter $*$-ho\-mo\-mor\-phism would be $(\id_\grS\grtensor\beta)\circ\Delta$ in our terminology.

Now, let $E$ be an infinite dimensional Euclidean vector space and let $\fm(E)$ denote the directed set of all finite dimensional subspaces of $E$. For any two finite dimensional subspaces $U\subset V\subset E$ one defines a $*$-ho\-mo\-mor\-phism $\beta_{U,V}$ as the composition 
\begin{align*}
\grS\grtensor\cC(U)&\xrightarrow{\Delta\grtensor\id_{\cC(U)}}\grS\grtensor\grS\grtensor\cC(U)
\\&\xrightarrow{\id_\grS\grtensor\beta_{U^\perp\cap V}\grtensor\id_{\cC(U)}}\grS\grtensor\cC(U^\perp\cap V)\grtensor\cC(U)
\\&\cong\grS\grtensor\cC((U^\perp\cap V)\oplus U)\cong\grS\grtensor\cC(V)
\end{align*}
and readily verifies $\beta_{U,W}=\beta_{V,W}\circ\beta_{U,V}$ for any three finite dimensional subspaces $U\subset V\subset W$ (cf.\ \cite[Proposition 3.2]{HigsonKasparovTrout}).
Hence the \textCstar-algebras $\grS\grtensor\cC(V)$ for finite dimensional $V\subset E$ form a directed system over $\fm(E)$.
Moreover, if a discrete group $G$ acts isometrically on $E$, then it also acts on the directed system via canonical $*$-ho\-mo\-mor\-phisms $\grS\grtensor\cC(V)\to\grS\grtensor\cC(gV)$ induced by $g\in G$.
\begin{defn}[{cf.\ \cite[Definition 3.3]{HigsonKasparovTrout}}]
For an infinite dimensional Euclidean vector space with isometric action by a discrete group $G$ one defines
\[\grS\cC(E)\coloneqq \varinjlim_{V\in\fm(E)}\grS\grtensor\cC(V)\,,\]
which is a $G$-\textCstar-algebra.
\end{defn}
This $G$-\textCstar-algebra $\grS\cC(E)$ is not of the form $\grS\grtensor\cC(E)$ for some $G$-\textCstar-algebra $\cC(E)$. The main result of \cite{HigsonKasparovTrout} is the following.
\begin{thm}[{\cite[Theorem 3.5]{HigsonKasparovTrout}}]
If $E$ has countable dimension and $G$ is a countable discrete group, then the canonical map
\[\beta\colon\grS\cong\grS\grtensor\cC(0)\to\grS\cC(E)\]
induces an isomorphism on the equivariant $\K$-theory of the underlying ungraded $G$-\textCstar-algebras.
\end{thm}
The proof works by constructing an asymptotic morphism
\[\varphi\colon \grS\cC(E)\to\Asymp\left(\varinjlim_{V\in\fm(E)}\grS\grtensor\Kom(\Hilbert(V))\right)\]
and a $*$-ho\-mo\-mor\-phism
\begin{equation}\label{defn:HKTgamma}
\gamma\colon \grS\to\varinjlim_{V\in\fm(E)}\grS\grtensor\Kom(\Hilbert(V))
\end{equation}
such that $\gamma$ induces an isomorphism on equivariant $\K$-theory and $(\gamma)_*^{-1}\circ\varphi_*$ is a left inverse to $\beta_*$.
Finally, an adaption of Atiyah's ``rotation argument'' \cite{Atiyah_Bottperiodicity} is used to prove that it is also a right inverse.

In this section, we want to revisit the technical proof from \cite[Section 5]{HigsonKasparovTrout} that $\gamma$ induces an isomorphism on equivariant $\K$-theory, because it is here that our notion of simpltotic morphism shows its advantages.
The same ideas can also be applied to the proof of \cite[Proposition 5.1(ii)]{HigsonKasparovTrout}, that is, that a similar $*$-ho\-mo\-mor\-phism 
\[\gamma\grtensor 1\colon \grS\cC(E)\to\varinjlim_{V\in\fm(E)}\grS\grtensor\Kom(\Hilbert(V))\grtensor\cC(V)\]
induces an isomorphism, too. The latter is needed for the rotation argument.

We start by recalling that $\Hilbert(V)$ in the formulas above denotes the $\Z_2$-graded Hilbert space
\[\Hilbert(V)\coloneqq L^2(V,\Cl(V))\,.\]
The Dirac operator on $\Hilbert(V)$ is  defined by
\[D_V\coloneqq\sum_{i=1}^{\dim V}\hat e_i\frac{\partial}{\partial x_i}\]
where $e_1,\dots,e_n$ is an orthonormal basis of $V$, $x_1,\dots,x_n$ is the dual coordinate system on $V$ and $\hat e_i$ denotes the twisted right multiplication
\(\hat e_i\xi=(-1)^{\deg\xi}\xi e_i\).
Furthermore one defines the operator
\[B_V\coloneqq D_V+C_V=\sum_{i=1}^{\dim V}(\hat e_i\frac{\partial}{\partial x_i}+e_ix_i)\,,\]
where $C_V$ also denotes left multiplication with the function $C_V\colon V\to\Cl(V),v\mapsto v$.
All of $D_V,C_V,B_V$ are essentially self-adjoint, degree one, unbounded operator $D_V$ with domain the Schwartz subspace of $\Hilbert(V)$ and do not depend on the choice of the orthonormal basis.

Functional calculus yields the $*$-ho\-mo\-mor\-phism
\[\gamma_V\colon\grS\to\Kom(\Hilbert(V))\,,\quad f\mapsto f(B_V)\]
for all $V\in\fm(E)$ and for $U\subset V\in\fm(E)$ one defines furthermore a $*$-ho\-mo\-mor\-phism $\gamma_{U,V}$ as the composition
\begin{align*}
\grS\grtensor\Kom(\Hilbert(U))&\xrightarrow{\Delta\grtensor\id_{\Kom(\Hilbert(U))}}\grS\grtensor\grS\grtensor\Kom(\Hilbert(U))
\\&\xrightarrow{\id_\grS\grtensor\gamma_{U^\perp\cap V}\grtensor\id_{\Kom(\Hilbert(U))}}\grS\grtensor\Kom(\Hilbert(U^\perp\cap V))\grtensor\Kom(\Hilbert(U))
\\&\cong\grS\grtensor\Kom(\Hilbert((U^\perp\cap V)\oplus U))\cong\grS\grtensor\Kom(\Hilbert(V))\,.
\end{align*}
These turn the $G$-\textCstar-algebras $\grS\grtensor\Kom(\Hilbert(V))$ into a directed system over $\fm(E)$.
The map $\gamma$ mentioned in \eqref{defn:HKTgamma} is now simply the canonical map from $\grS\cong\grS\grtensor\Kom(\Hilbert(0))$ into this direct limit.

The proof in \cite[Section 5]{HigsonKasparovTrout} that $\gamma$ induces an isomorphism on $\K$-theory for countable dimensional $E$ and countable $G$ works as follows.
For each $V\in\fm(E)$, the function 
\[\zeta_V\colon v\mapsto \pi^{-\dim(V)/4}\exp(-\frac12\|v\|^2)\]
is a vector of norm one in $\Hilbert(V)$.
Then one obtains isometries 
\begin{align*}
T_{U,V}\colon\Hilbert(U)&\to\Hilbert(U^\perp\cap V)\grtensor\Hilbert(U)\cong\Hilbert((U^\perp\cap V)\oplus U)=\Hilbert(V)
\\\xi&\mapsto \zeta_{U^\perp\cap V}\grtensor \xi
\end{align*}
for all $U\subset V\in\fm(E)$.
These turn the Hilbert spaces $\Hilbert(V)$ into a directed system over $\fm(E)$ and one defines the $\Z_2$-graded Hilbert space
\[\Hilbert(E)\coloneqq\varinjlim_{V\in\fm(E)}\Hilbert(V)\,.\]
The group $G$ acts on the directed system via canonical isometries $g_*\colon\Hilbert(V)\to\Hilbert(gV)$ for $g\in G$ and hence it also acts isometrically on its direct limit $\Hilbert(E)$. Note that all $\zeta_V$ are mapped to the same unit vector in the direct limit, which we denote by $\zeta\in\Hilbert(E)$, and this vector is invariant under the $G$-action. Let $P\in\Kom(\Hilbert(E))$ denote the projection onto the subspace spanned by $\zeta$.

Now, the idea of the proof is to construct an asymptotic morphism 
\[\psi\colon \varinjlim_{V\in\fm(E)}\grS\grtensor \Kom(\Hilbert(V))\to\Asymp(\grS\grtensor\Kom(\Hilbert(E)))\]
such that $\psi\circ\gamma$ is $1$-homotopic to the $*$-ho\-mo\-mor\-phism $\sigma\colon f\mapsto f\grtensor P$. Since $\sigma$ induces an isomorphism on $\K$-theory, $\sigma_*^{-1}\circ\psi_*$ is a left inverse to $\gamma_*$ and another rotation argument is used to show that it is also a right inverse.

This $\psi$ is defined as follows: Fix a cofinal sequence $V_0\subset V_1\subset V_2\subset\dots$ of finite dimensional subspaces of $E$ such that each $g\in G$ maps each $V_n$ into some $V_{n+1}$. Note that this is where the countability assumption is used.
Let $\Hilbert(V_n^\perp)$ for the infinite dimensional Euclidean space $V_n^\perp\subset E$ be defined in complete analogy to $\Hilbert(E)$, let $W_n\coloneqq V_n\cap V_{n-1}^\perp$ and $t_n\coloneqq 1+t^{-1}n$. Then
\[B_{n,t}\coloneqq t_{n+1}B_{W_{n+1}}+t_{n+2}B_{W_{n+2}}+t_{n+3}B_{W_{n+3}}+\dots\]
can be interpreted as a degree one, essentially self-adjoint, unbounded operator on $\Hilbert(V_n^\perp)$ with compact resolvent and hence it yields an asymptotic morphism
\[\psi_n'\colon\grS\to\Asymp(\Kom(\Hilbert(V_n^\perp)))\,,\quad f\mapsto [t\mapsto f(B_{n,t})]\]
One obtains asymptotic morphisms $\psi_n$ as the compositions
\begin{align*}
\grS\grtensor\Kom(\Hilbert(V_n))&\xrightarrow{\Delta\grtensor\id_{\Kom(\Hilbert(V_n))}}\grS\grtensor\grS\grtensor\Kom(\Hilbert(V_n))
\\&\xrightarrow{\id_\grS\grtensor\psi_n'\grtensor\id_{\Kom(\Hilbert(V_n))}}\grS\grtensor\Asymp(\Kom(\Hilbert(V_n^\perp)))\grtensor\Kom(\Hilbert(V_n))
\\&\to \Asymp(\grS\grtensor\Kom(\Hilbert(V_n^\perp)\grtensor\Hilbert(V_n)))
\cong \Asymp(\grS\grtensor\Kom(\Hilbert(E)))
\end{align*}
and together they give rise to the equivariant asymptotic morphism 
\[\psi\colon \varinjlim_{V\in\fm(E)}\grS\grtensor \Kom(\Hilbert(V))\cong\varinjlim_{n\in\N}\grS\grtensor \Kom(\Hilbert(V_n))\to\Asymp(\grS\grtensor\Kom(\Hilbert(E)))\,.\]

Let us also briefly recall why $\sigma_*^{-1}\circ\psi_*$ is indeed a left inverse for $\gamma_*$:
Since the kernel of $B_{0,t}$ is spanned by the vector $\zeta$, $\psi_0'$ is $1$-homotopic to the $*$-ho\-mo\-mor\-phism $f\mapsto \eta(f)P$ via the $1$-homotopy
\[f\mapsto \left[t\mapsto\left(s\mapsto \begin{cases}f(s^{-1}B_{0,t})&s>0\\f(0)P&s=0\end{cases}\right)\right]\]
and hence $\psi\circ\gamma=\psi_0$ is $1$-homotopic to $\sigma$. We omit recalling the rotation argument which shows that it is also a right inverse.

As we can see, the operators $B_{n,t}$ used in \cite[Section 5]{HigsonKasparovTrout} already incorporate the idea of rescaling by different factors in different directions. 
Using this rescaling trick, we will now present a more sophisticated construction in which the asymptotic morphism $\psi$ will be replaced by a simpltotic morphism 
\begin{equation}
\label{eq:Psidef}
\Psi\colon \varinjlim_{V\in\fm(E)}\grS\grtensor \Kom(\Hilbert(V))\to\Simpl^{\fm(E)}(\grS\grtensor\Kom(\Hilbert(E)))\,.
\end{equation}
In contrast to the definition of $\psi$, this does not require the choice of a cofinal sequence of finite dimensional subspaces and hence we can dispose of the countability assumption on $\dim(E)$ and $G$.

We define $\Psi$ by constructing in a compatible way for each $V\in\fm(E)$ an asymptotic morphism
\[\Psi_V\colon \grS\grtensor \Kom(\Hilbert(V))\to\Simpl^{\fm(E)\ab V}(\grS\grtensor\Kom(\Hilbert(E)))\cong\Simpl^{\fm(E)}(\grS\grtensor\Kom(\Hilbert(E)))\,.\]
To do so, we consider the barycentric subdivision of $\Delta^{\fm(E)\ab V}$. The barycenter of vertices $U_0,\dots,U_n\in\fm(E)$ will be labeled with the subspace $U_0+\dots+U_n\in\fm(E)$.
This way, the vertices of each $n$-simplex $\sigma=(v_0,\dots,v_n)$ of the barycentric subdivision will be labeled by an increasing sequence of finite dimensional subspaces $V_0\subset V_1\subset\dots\subset V_n\in\fm(E)$, all of which contain $V$. Define $W_0\coloneqq V^\perp\cap V_0$ and $W_k\coloneqq V_{k-1}^\perp\cap V_k$ for $k=1,\dots,n$ and let $P_{W_k}$ denote the projection onto the one-dimensional subspace spanned by the unit vector $\zeta_{W_k}$. It is readily checked that for each $f\in\grS$ the maps
\begin{align*}
\Psi_{V,\sigma,k}(f)\colon|\Delta^\sigma|&\to\Kom(\Hilbert(W_k)))
\\\sum_{i=0}^n\lambda_iv_i&\mapsto\begin{cases}f\left(\frac{1}{\lambda_k+\dots+\lambda_n}B_{W_k}\right)&\lambda_k+\dots+\lambda_n>0\\P_{W_k}&\lambda_k+\dots+\lambda_n=0\end{cases}
\end{align*}
are continuous and hence we obtain $*$-ho\-mo\-mor\-phisms 
\[\Psi_{V,\sigma,k}\colon\grS\to\Cb(|\Delta^\sigma|;\Kom(\Hilbert(W_k)))\,.\]
If we let $P_{V_n^\perp}$ denote the projection onto the one-dimensional subspace spanned by the canonical unit vector $\zeta_{V_n^\perp}\in\Hilbert(V_n^\perp)$, defined in complete analogy to $P\in\Kom(\Hilbert(E))$ and $\zeta\in\Hilbert(E)$, then we obtain another $*$-ho\-mo\-mor\-phism $\Psi_{V,\sigma}'$ as the composition
\begin{align*}
\grS&\xrightarrow{\Delta^{n}}\grS\grtensor\dots\grtensor\grS
\\&\xrightarrow{\Psi_{V,\sigma,0}\grtensor\dots\grtensor\Psi_{V,\sigma,n}\grtensor P_{V_n^\perp}}\widehat\bigotimes\vphantom{\bigotimes}_{k=0}^n\Cb(|\Delta^\sigma|;\Kom(\Hilbert(W_k)))\grtensor\Kom(\Hilbert(V_n^\perp))
\\&\to \Cb(\Kom(\Hilbert(W_0\oplus\dots\oplus W_n\oplus V_n^\perp)))=\Cb(|\Delta^\sigma|;\Kom(\Hilbert(V^\perp)))\,.
\end{align*}
The similarity to the asymptotic morphisms $\psi_n'$ defined earlier is that the value $\Psi_{V,\sigma}'(f)(\sum_{i=0}^n\lambda_i v_i)$ should be seen as a way of giving a meaning to the meaningless formula
\[f\left(\frac{1}{\lambda_0+\dots+\lambda_n}B_{W_0}+\frac{1}{\lambda_1+\dots+\lambda_n}B_{W_1}+\dots+\frac1{\lambda_n}B_{W_n}+\infty\cdot B_{V_n^\perp}\right)\,.\]
The $\Psi_{V,\sigma}'$ fit together along the faces of the simplices in the barycentric subdivision, yielding a $*$-ho\-mo\-mor\-phism
\[\Psi_V'\colon \grS\to \Cb(|\Delta^{\fm(E)\ab V}|;\Kom(\Hilbert(V^\perp)))\,.\]
We now define $\Psi_V$ as the quotient of the $*$-ho\-mo\-mor\-phism $\Psi_V''$ defined as the composition
\begin{align*}
\grS\grtensor\Kom(\Hilbert(V))&\xrightarrow{\Delta\grtensor\id_{\Kom(\Hilbert(V))}}\grS\grtensor\grS\grtensor\Kom(\Hilbert(V))
\\&\xrightarrow{\id_\grS\grtensor\Psi_V'\grtensor\id_{\Kom(\Hilbert(V))}}\grS\grtensor\Cb(|\Delta^{\fm(E)\ab V}|;\Kom(\Hilbert(V^\perp)))\grtensor\Kom(\Hilbert(V))
\\&\to \Cb(|\Delta^{\fm(E)\ab V}|;\grS\grtensor\Kom(\Hilbert(V^\perp)\grtensor\Hilbert(V)))
\\&\cong \Cb(|\Delta^{\fm(E)\ab V}|;\grS\grtensor\Kom(\Hilbert(E)))\,.
\end{align*}
A straightforward computation shows that these $\Psi_V$ are compatible with the maps $\gamma_{U,V}$ and thus they induce a $*$-ho\-mo\-mor\-phism $\Psi$ as in \eqref{eq:Psidef}.

It follows from \Cref{lem:simpltoticmorphismgivenbystarhomomorphism} that $\Psi\circ\gamma=\Psi_0$ is $\fm(E)$-homotopic to the class represented by the $*$-ho\-mo\-mor\-phism $\ev{0}\circ\Psi_0''=\sigma$ and hence $\sigma_*^{-1}\circ\Psi_*$ is left inverse to $\gamma_*$.
No countability assumption on $\dim(E)$ and $G$ were needed, because our constructions did not depend on any choices.
We leave it to the reader to carry out the rotation argument analogous to the one on \cite[p.\ 31]{HigsonKasparovTrout} to show that $\sigma_*^{-1}\circ\Psi_*$ is also a right inverse.

\begin{conclusion}
This illustrates our point that simpltotic morphisms provide the most natural framework to carry out the idea of obtaining $\EE$-theory elements by rescaling differential operators differently in different directions.
This rescaling was necessary in the infinite dimensional case to obtain compact operators via functional calculus.

Therefore we believe that our model of $\EE$-theory is the correct entity to do index theory on infinite dimensional manifolds.
\end{conclusion}

\bibliographystyle{alpha}

\ \\
\textsc{
Mathematisches Institut, 
Georg--August--Universit\"at G\"ottingen,
Bunsenstr. 3-5, 
D-37073 G\"ottingen, 
Germany}

\noindent
\textit{E-mail address:} \url{christopher.wulff@mathematik.uni-goettingen.de}

\end{document}